\documentclass[11pt]{amsart}
\usepackage{amsmath,amsthm,amstext}
\usepackage{amsfonts,amssymb}
\usepackage[mathscr]{eucal}
\usepackage[]{hyperref}
\usepackage{setspace,url}
\usepackage{tikz}

\newtheorem{theorem}{Theorem}[section]
\newtheorem{lemma}[theorem]{Lemma}
\newtheorem{proposition}[theorem]{Proposition}

\theoremstyle{definition}
\newtheorem{remark}[theorem]{Remark}

\numberwithin{equation}{section}

\newcommand{\NN}{\ensuremath{\mathbb{N}}}

\newcommand{\R}{\mathbb{R}}

\title[Ground state on dumbbell graph]{Ground state on the dumbbell graph}

\author[J. L. Marzuola]{Jeremy L. Marzuola}
\address{Mathematics Department, University of North Carolina - Chapel Hill, Chapel Hill, NC 27599, USA}
\email{marzuola@math.unc.edu}

\author[D.E. Pelinovsky]{Dmitry E. Pelinovsky}
\address{Department of Mathematics, McMaster University, Hamilton, Ontario, L8S 4K1, Canada}
\email{dmpeli@math.mcmaster.ca}
\address{Department of Applied Mathematics, Nizhny Novgorod State Technical University, 24 Minin Street, Nizhny Novgorod, 603950, Russia}

\begin{document}

\begin{abstract}
We consider standing waves in the focusing nonlinear Schr\"odinger (NLS) equation on a dumbbell graph
(two rings attached to a central line segment subject to the Kirchhoff boundary conditions at the junctions).
In the limit of small $L^2$ norm, the ground state (the orbitally stable standing wave
of the smallest energy at a fixed $L^2$ norm) is represented by a constant solution. However, when the
$L^2$ norm is increased, this constant solution undertakes two bifurcations, where the first is the pitchfork (symmetry breaking)
bifurcation and the second one is the symmetry preserving bifurcation. As a result of the first
symmetry breaking bifurcation, the standing wave becomes more localized in one of the two rings.
As a result of the second symmetry preserving bifurcation, the standing wave becomes localized in the central line segment.
In the limit of large norm solutions, both standing waves are represented by
a truncated solitary wave localized in either the ring or the central line segment. 
Both waves are stable local constrained minimizers of the energy for the fixed $L^2$ norm but the 
asymmetric wave supported in the ring has a smaller energy. 
The analytical results are confirmed by numerical approximations of the ground state on the dumbbell graph.
\end{abstract}

\maketitle

\section{Introduction}
\label{sec:Intro}

Nonlinear Schr\"{o}dinger (NLS) equations on quantum graphs have been recently studied
in many physical and mathematical aspects \cite{N}. In the physical literature,
mostly in the context of Bose-Einstein condensation, various types of graphs
have been modeled to show formation and trapping of standing waves \cite{[BS],[HTM],[Sob],[SSBM],[ZS]}.
In the mathematical literature, existence, variational properties, stability, and
scattering have been studied on star graphs, including the $Y$-shaped graphs
\cite{ACFN0,ACFN2,AST}.

More complicated graphs may lead to resonances and nontrivial bifurcations
of standing waves. For example, standing waves were studied on
the tadpole graph (a ring attached to a semi-infinite line) \cite{CFN,NPS}.
Besides the standing waves supported in the ring that bifurcate from eigenvalues of
the linear operators closed in the ring, the tadpole graph also admits
the standing waves localized in the ring with the tails extended
in the semi-infinite line. These standing waves bifurcate from the end-point
resonance of the linear operators defined on the tadpole graph and include the positive
solution, which is proved to be orbitally stable in the evolution of the cubic NLS equation
near the bifurcation point \cite{NPS}.
The positive solution bifurcating from the end-point resonance bears the lowest energy at the fixed $L^2$ norm,
called the ground state. Other positive states on the tadpole graph also exist in
parameter space far away from the end-point resonance but they do not bear smallest energy
and they do not branch off the ground state \cite{CFN}.

The present contribution is devoted to analysis of standing waves (and the ground state) on the dumbbell graph
shown in Figure \ref{fig1}. The dumbbell graph is represented by two rings (with equal length normalized
to $2\pi$) connected by the central line segment (with length $2L$). At the junctions
between the rings and the central line segment, we apply the
Kirchhoff boundary conditions to define the coupling.
These boundary conditions ensure continuity of functions and
conservation of the current flow through the network junctions; they also
allow for self-adjoint extension of
the Laplacian operator defined on the dumbbell graph.

\begin{figure}[htp]
\includegraphics[width=12cm]{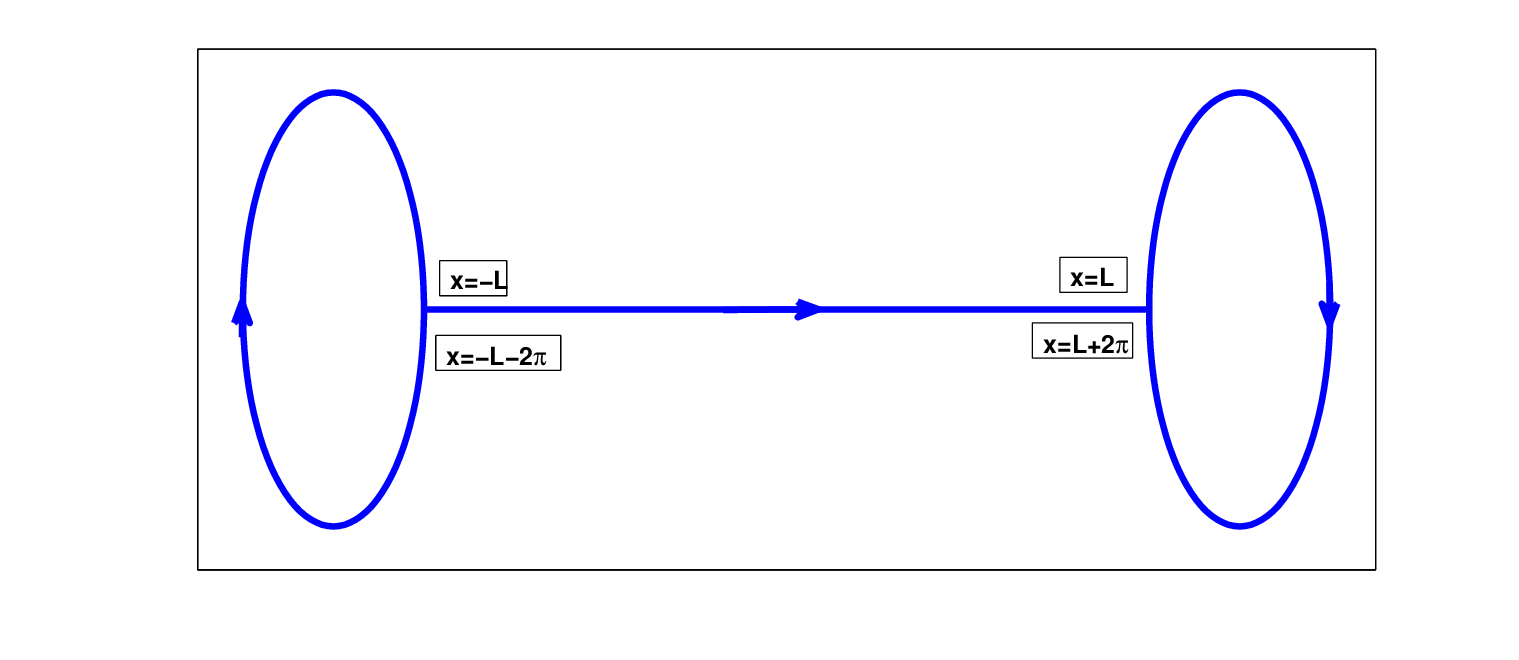}
\caption{Schematic representation of the dumbbell graph.}
\label{fig1}
\end{figure}

Let the central line segment be placed on $I_0 := [-L,L]$, whereas the end rings are placed
on $I_- := [-L-2\pi,-L]$ and $I_+ := [L,L+2\pi]$. The Laplacian operator
is defined piecewise by
\begin{equation*}
\Delta \Psi = \left[  \begin{array}{c}
u_{-}'' (x),  \  x \in I_-, \\
u_0'' (x), \ x \in I_0, \\
u_+'' (x), \ x \in I_+,
\end{array}  \right], \quad \mbox{\rm acting on} \quad \Psi = \left[  \begin{array}{c}
u_{-} (x),  \  x \in I_-, \\
u_0(x), \ x \in I_0, \\
u_+(x), \ x \in I_+,
\end{array}  \right],
\end{equation*}
subject to the Kirchhoff boundary conditions at the two junctions:
\begin{eqnarray}
\left\{ \begin{array}{l}
u_-(-L-2\pi) = u_-(-L) = u_0(-L), \\
u_-'(-L) - u_-'(-L - 2\pi) = u_0'(-L), \end{array} \right.
\label{bc-1}
\end{eqnarray}
and
\begin{eqnarray}
\left\{ \begin{array}{l}
u_+(L + 2\pi) = u_+(L) = u_0(L), \\
u_+'(L) - u_+'(L + 2\pi) = u_0'(L). \end{array} \right.
\label{bc-2}
\end{eqnarray}
The Laplacian operator $\Delta$ is equipped with the domain $\mathcal{D}(\Delta)$ given by
a subspace of $H^2(I_- \cup I_0 \cup I_+)$ closed with the boundary conditions
(\ref{bc-1}) and (\ref{bc-2}). By Theorem 1.4.4 in \cite{BK},
the Kirchhoff boundary conditions are symmetric and
the operator $\Delta$ is self-adjoint on its domain $\mathcal{D}(\Delta)$.

The cubic NLS equation on the dumbbell graph  is given by
\begin{equation}\label{NLS}
i\frac{\partial}{\partial t} \Psi = \Delta \Psi + 2 |\Psi|^{2} \Psi\ , \quad \Psi \in \mathcal{D}(\Delta),
\end{equation}
where the nonlinear term $|\Psi|^{2}\Psi$ is also defined piecewise on $I_- \cup I_0 \cup I_+$.
The energy of the cubic NLS equation (\ref{NLS}) is given by
\begin{equation}
\label{energy}
E(\Psi) = \int_{I_- \cup I_0 \cup I_+} \left( |\partial_x \Psi|^2 - |\Psi|^4 \right) dx,
\end{equation}
and it is conserved in the time evolution of the NLS equation (\ref{NLS}). The
energy is defined in the energy space $\mathcal{E}(\Delta)$ given by
\begin{equation*}
\label{energy-space}
\mathcal{E}(\Delta) := \left\{ \Psi \in H^1(I_- \cup I_0 \cup I_+) : \;\;
\begin{array}{l} u_-(-L-2\pi) = u_-(-L) = u_0(-L) \\
u_+(L + 2\pi) = u_+(L) = u_0(L) \end{array} \right\}.
\end{equation*}
Local and global wellposedness of the cubic NLS equation (\ref{NLS}) both
in energy space $\mathcal{E}(\Delta)$ and domain space $\mathcal{D}(\Delta)$
can be proved using standard techniques, see \cite{ACFN2}.

Standing waves of the focusing NLS equation (\ref{NLS}) are given by the solutions of the form
$\Psi(t,x) = e^{i\Lambda t} \Phi(x)$, where $\Lambda$ and $\Phi \in \mathcal{D}(\Delta)$ are considered to be real.
This pair satisfies the stationary NLS equation
\begin{equation}
\label{statNLS}
-\Delta \Phi - 2 |\Phi|^{2} \Phi =  \Lambda \Phi \qquad \Lambda \in \mathbb{R}\,,\;\Phi\in \mathcal{D}(\Delta).
\end{equation}
The stationary NLS equation (\ref{statNLS}) is the Euler--Lagrange equation of the energy
functional $H_{\Lambda} := E - \Lambda Q$, where the charge
\begin{equation}
\label{charge}
Q(\Psi) = \int_{I_- \cup I_0 \cup I_+} |\Psi|^2 dx
\end{equation}
is another conserved quantity in the time evolution of the NLS equation (\ref{NLS}).

We shall now define the ground state of the NLS equation on the dumbbell graph as the standing wave of
smallest energy $E$ at a fixed value of $Q$, that is, a solution of the constrained minimization
problem
\begin{equation}
\label{minimizer}
E_0 = \inf\{ E(\Psi) : \quad  \Psi \in \mathcal{E}(\Delta), \quad Q(\Psi) = Q_0\}.
\end{equation}

By Theorem 1.4.11 in \cite{BK}, although the energy space $\mathcal{E}(\Delta)$ is only defined by the continuity boundary conditions,
the Kirchhoff boundary conditions for the derivatives are natural boundary conditions
for critical points of the energy functional $E(\Psi)$ in the space $\mathcal{E}(\Delta)$.
In other words, using test functions and the
weak formulation of the Euler--Lagrange equations in the energy space $\mathcal{E}(\Delta)$,
the derivative boundary conditions are also obtained
in addition to the continuity boundary conditions. By bootstrapping arguments, we conclude that
any critical point of the energy functional $H_{\Lambda}$ in $\mathcal{E}(\Delta)$ is also
a solution of the stationary NLS equation (\ref{statNLS}) in $\mathcal{D}(\Delta)$.
On the other hand, solutions of the stationary NLS equation (\ref{statNLS}) in $\mathcal{D}(\Delta)$
are immediately the critical points of the energy functional $H_{\Lambda}$. Therefore,
the set of standing wave solutions of the stationary NLS equation (\ref{statNLS})
is equivalent to the set of critical points of the energy functional $H_{\Lambda}$.

Non-existence of ground states on graphs was proved by Adami {\em et al.} \cite{AST}
for some non-compact graphs. For example, the graph consisting of one ring connected to
two semi-infinite lines does not have a ground state. On the other hand, the tadpole
graph with one ring and one semi-infinite line escapes the non-existence condition of \cite{AST}
and has a ground state, in agreement with the results of \cite{NPS}.  Because $I_- \cup I_0 \cup I_+$ is compact,
existence of the global constrained minimizer in (\ref{minimizer}) follows from the standard
results in calculus of variations. As a minimizer of the energy functional $H_{\Lambda}$,
the ground state is orbitally stable in the time evolution of the NLS equation (\ref{NLS}), see for instance \cite{GSS1}.
The main question we would like to answer is how the ground state looks like on the dumbbell graph
depending on the parameter $Q_0$ for the charge $Q$. Until now, no rigorous analysis
of the NLS equation (\ref{NLS}) on a compact graph has been developed. On the other hand,
ground states on compact intervals subject to Dirichlet or periodic boundary conditions
have been considered in the literature \cite{Carr,FHK}.

The dumbbell graph resembles the geometric configuration that arises typically in the double-well potential
modeled by the Gross--Pitaevskii equation \cite{FS,GMW,KPP,MW,PP}. From this analogy, one can anticipate
that the ground state is a symmetric state distributed everywhere in the graph in the limit of small values of $Q_0$ but
it may become an asymmetric standing wave residing in just one ring as a result of a pitchfork bifurcation
for larger values of $Q_0$. We show in this paper that this intuitive picture is correct.

We show that the ground state is indeed represented by a constant (symmetric) solution
for small values of $Q_0$. For larger values of $Q_0$, the constant solution
undertakes two instability bifurcations. At the first bifurcation associated
with the anti-symmetric perturbation, a family of positive asymmetric standing waves
is generated. The asymmetric wave has the lowest energy at
the fixed $Q_0$ near the symmetry breaking bifurcation.
At the second bifurcation associated with the symmetric perturbation of the constant solution,
another family of positive symmetric standing waves is generated.
The symmetric wave does not have the lowest energy at the fixed $Q_0$
near the bifurcation and retains this property in the limit of large $Q_0$, 
though it does become a stable local constrained minimizer of the energy.
It is rather surprising that both the precedence of the symmetry-breaking bifurcation of the constant solution
and the appearance of the asymmetric wave as a ground state in the limit of large $Q_0$
do not depend on the value of the length parameter $L$ relative to $\pi$.

Our main result is formulated as the following two theorems. We also include numerical approximations
of the standing waves of the stationary NLS equation (\ref{statNLS}) in order
to illustrate the main result. The numerical work relies on
the Petviashvili's and Newton's iterative methods which are commonly used
for approximation of standing waves of the NLS equations \cite{Pelin-book,Yang-book}.

\begin{theorem}
There exist $Q_0^*$ and $Q_0^{**}$ ordered as $0 < Q_0^* < Q_0^{**} < \infty$
such that the ground state of the constrained minimization problem (\ref{minimizer}) for $Q_0 \in (0,Q_0^*)$
is given (up to an arbitrary rotation of phase) by the constant solution of the stationary NLS equation (\ref{statNLS}):
\begin{equation}
\label{ground-1}
\Phi(x) = p, \quad \Lambda = -2p^2, \quad Q_0 = 2 (L + 2\pi) p^2.
\end{equation}
The constant solution undertakes the symmetry breaking bifurcation at $Q_0^*$
and the symmetry preserving bifurcation at $Q_0^{**}$, which result in
the appearance of new positive non-constant solutions.
The asymmetric standing wave is a ground state of (\ref{minimizer}) for $Q \gtrsim Q_0^*$
but the symmetric standing wave is not a ground state of (\ref{minimizer}) for $Q \gtrsim Q_0^{**}$.
\label{theorem-graph}
\end{theorem}

\begin{theorem}
In the limit of large negative $\Lambda$, there exist two standing wave solutions of the stationary NLS equation (\ref{statNLS}).
One solution is a positive asymmetric wave localized in the ring:
\begin{equation}
\label{ground-2}
\Phi(x) = |\Lambda|^{1/2} {\rm sech}(|\Lambda|^{1/2} (x - L - \pi)) + \tilde{\Phi}(x), \quad
Q_0 = 2 |\Lambda|^{1/2} + \tilde{Q}_0,
\end{equation}
and the other solution is a positive symmetric wave localized in the central line segment:
\begin{equation}
\label{ground-3}
\Phi(x) = |\Lambda|^{1/2} {\rm sech}(|\Lambda|^{1/2} x) + \tilde{\Phi}(x), \quad Q_0 = 2 |\Lambda|^{1/2} + \tilde{Q}_0,
\end{equation}
where $\| \tilde{\Phi} \|_{H^2(I_- \cup I_0 \cup I_+)} \to 0$ and $|\tilde{Q}_0| \to 0$ as $\Lambda \to -\infty$ in both cases.
The symmetric wave satisfying (\ref{ground-3}) is a local constrained minimizer of the energy for sufficiently large $Q_0$, 
but the energy of the asymmetric wave satisfying (\ref{ground-2}) is smaller.
\label{theorem-limit}
\end{theorem}

\begin{remark}
It follows from Lemmas \ref{lemma-2}, \ref{lemma-3} and Remark \ref{remark-4}
that the constant standing wave (\ref{ground-1})
undertakes a sequence of bifurcations, where the first two bifurcations
at $Q_0^*$ and $Q_0^{**}$ lead to the positive asymmetric and symmetric standing waves respectively.
We also show numerically that these same positive waves are connected to the truncated solitary waves
(\ref{ground-2}) and (\ref{ground-3}) as $\Lambda \to -\infty$.  See Figures \ref{stat2}, \ref{stat3},
\ref{statcomp}, and \ref{stat8}.
\end{remark}

\begin{remark}
We show in Lemma \ref{lemma-limit} that the energy difference between the truncated
solitary wave localized in the central segment and
the one localized in the ring is exponentially small as $\Lambda \to -\infty$ 
but the energy of the asymmetric wave is smaller than the energy of the symmetric wave. 
In the numerical iterations
of the Petviashvili's and Newton's methods, both
standing waves arise naturally when the initial data is concentrated either in a loop or in a central link.
Figures \ref{stat9} and \ref{stat10} illustrate that both symmetric and asymmetric waves are local 
constrained minimizers of energy, so that they are orbitally stable in the time evolution of the cubic NLS equation (\ref{NLS}).
\end{remark}

The paper is organized as follows. Section 2 reports a complete characterization of the
linear spectrum of the Laplacian operator
on the dumbbell graph. Section 3 is devoted to the analytical characterization of the constant standing
wave (\ref{ground-1}) and the first two instability bifurcations when parameter $Q_0$ is increased.
Section 4 describes the analytical characterization of the two
standing waves localized in the central segment and at one of the two rings in the limit of large values
of $Q_0$. The proofs of Theorems \ref{theorem-graph} and \ref{theorem-limit}
are furnished by the individual results of Sections 3 and 4.
Section 5 reports numerical approximations of the standing waves of the stationary NLS equation (\ref{statNLS}).

\section{Linear spectrum of the Laplacian on the dumbbell graph}

The linear spectrum of the Laplacian on the dumbbell graph is defined by solutions of the spectral problem
\begin{equation}
\label{spectrum}
-\Delta U = \lambda U,  \quad \lambda \in \mathbb{R}, \quad U \in \mathcal{D}(\Delta).
\end{equation}
Because $I_- \cup I_0 \cup I_+$ is compact, the spectrum of $-\Delta$ is purely discrete. Let us denote it by $\sigma(-\Delta)$.
Because $\Delta$ is self-adjoint with the domain $\mathcal{D}(\Delta)$ in $L^2(I_- \cup I_0 \cup I_+)$,
the spectrum $\sigma(-\Delta)$ consists of real positive eigenvalues of equal algebraic and geometric multiplicities.
The distribution of eigenvalues is given by the following result.

\begin{proposition}
\label{prop-spectrum}
$\sigma(-\Delta)$ consists of a simple zero eigenvalue with the constant eigenfunction
and the union of the following three countable sequences of eigenvalues:
\begin{itemize}
\item A sequence of double eigenvalues $\{ n^2 \}_{n \in \mathbb{N}}$.
The corresponding eigenfunctions are compactly supported on either $I_-$ or $I_+$ and
are odd with respect to the middle point in $I_{\pm}$.

\item A sequence of simple eigenvalues $\{ \omega_n^2 \}_{n \in \mathbb{N}}$, where $\omega_n$ is  given by
a positive root of the transcendental equation
\begin{equation}
\label{transc-1}
D_L^{({\rm even})}(\omega) := 2 \tan(\omega \pi) + \tan(\omega L) = 0.
\end{equation}
The corresponding eigenfunctions are distributed everywhere in $I_- \cup I_0 \cup I_+$ and are
even with respect to the middle point in $I_0$.

\item A sequence of simple eigenvalues $\{ \Omega_n^2 \}_{n \in \mathbb{N}}$, where $\Omega_n$ is given by
a positive root of the transcendental equation
\begin{equation}
\label{transc-2}
D_L^{({\rm odd})}(\omega) := 2 \tan(\omega \pi) - \cot(\omega L) = 0.
\end{equation}
The corresponding eigenfunctions are distributed everywhere in $I_- \cup I_0 \cup I_+$ and are
odd with respect to the middle point in $I_0$.
\end{itemize}
\end{proposition}

\begin{proof}
Let us decompose $U$ in the components $\{u_-,u_0,u_+\}$ defined on $\{I_-,I_0,I_+ \}$ respectively.
We first observe the following reduction of the spectral problem (\ref{spectrum}): if $u_0$ is identically
zero, then $u_+$ and $u_-$ are uncoupled, and each satisfies the over-determined boundary-value problem
\begin{equation}
\label{spectrum-1}
\left\{ \begin{array}{l} -u_{\pm}''(x) = \lambda u_{\pm}(x),  \quad x \in I_{\pm}, \\
u_{\pm}(x) \in H^2_{\rm per,0}(I_{\pm}), \end{array} \right.
\end{equation}
where $H^2_{\rm per,0}(I_{\pm})$ denotes the subspace of $H^2_{\rm per}(I_{\pm})$ subject to the additional
Dirichlet boundary conditions at the end points of $I_{\pm}$. The over-determined problem (\ref{spectrum-1})
can be solved in the space of functions which are odd with respect to the middle point in $I_{\pm}$.
In this way, a complete set of solutions of the boundary--value problem (\ref{spectrum-1})
is given by the set of double eigenvalues $\{ n^2 \}_{n \in \mathbb{N}}$ with two linearly independent eigenfunctions
\begin{equation}
\label{eigenvalues-1}
u_+(x) = \sin[n(x-L-\pi)], \;\; u_-(x) = 0
\end{equation}
and
\begin{equation}
\label{eigenvalues-2}
u_+(x) = 0, \;\; u_-(x) = \sin[n(x+L+\pi)].
\end{equation}

We next consider other solutions of the spectral problem (\ref{spectrum}), for which $u_0$ is not identically zero.
By the parity symmetry on $I_- \cup I_0 \cup I_+$, the eigenfunctions are either even or odd with respect to the middle point in $I_0$.
Since $\sigma(-\Delta)$ consists of real positive eigenvalues of equal algebraic and geometric multiplicities,
we parameterize $\lambda = \omega^2$. For even functions, we normalize the eigenfunction $U$ by
\begin{equation}
\label{eigenvalues-3}
u_0(x) = \cos(\omega x), \quad x \in I_0, \quad u_-(-x) = u_+(x), \quad x \in I_+.
\end{equation}
The most general solution of the differential equation (\ref{spectrum}) on $I_+$ is given by
\begin{equation}
\label{eigenvalues-4}
u_+(x) = A \cos[\omega (x - L - \pi)] + B \sin[\omega (x-L-\pi)],
\end{equation}
where the coefficients $A$ and $B$, as well as the spectral parameter $\omega$ is found from
the Kirchhoff boundary conditions in (\ref{bc-2}). However, since $u_+(L+2\pi) = u_+(L)$,
we have $B \neq 0$ if and only if $\sin(\pi \omega) = 0$. This condition is satisfied for $\omega = n \in \mathbb{N}$,
when  the uncoupled eigenfunction (\ref{eigenvalues-1}) arise for the $B$-term of the decomposition (\ref{eigenvalues-4}).
Therefore, without loss of generality, we can consider
other eigenfunctions by setting $B = 0$. Then, the Kirchhoff boundary conditions (\ref{bc-2}) yield
the constraints
\begin{align*}
\left\{ \begin{array}{l}
A \cos(\omega \pi) = \cos(\omega L),  \\
2 A \omega \sin(\omega \pi) = -\omega \sin(\omega L). \end{array} \right.
\end{align*}
Eliminating $A := \frac{\cos(\omega L)}{\cos(\omega \pi)}$, we obtain the dispersion relation (\ref{transc-1}) that admits
a countable set of positive roots $\{ \omega_n \}_{n \in \mathbb{N}}$ in addition to the zero root $\omega = 0$
that corresponds to the constant eigenfunction.

For odd functions, we normalize the eigenfunction $U$ by
\begin{equation}
\label{eigenvalues-5}
u_0(x) = \sin(\omega x), \quad x \in I_0, \quad u_-(-x) = -u_+(x), \quad x \in I_+.
\end{equation}
Representing the most general solution of the differential equation (\ref{spectrum}) on $I_+$ by
(\ref{eigenvalues-4}), we have the same reasoning to set $B = 0$. Then,
the Kirchhoff boundary conditions (\ref{bc-2}) yield
the constraints
\begin{align*}
\left\{ \begin{array}{l} A \cos(\omega \pi) = \sin(\omega L),  \\
2 A \omega \sin(\omega \pi) = \omega \cos(\omega L). \end{array} \right.
\end{align*}
Eliminating $A := \frac{\sin(\omega L)}{\cos(\omega \pi)}$, we obtain the dispersion relation (\ref{transc-2}) that admits
a countable set of positive roots $\{ \Omega_n \}_{n \in \mathbb{N}}$.
The root $\omega = 0$ is trivial since it corresponds to the zero solution for $U$.
Therefore, $\lambda = 0$ is a simple eigenvalue of
the spectral problem (\ref{spectrum}) with constant eigenfunction $U$. All assertions of the proposition are proved.
\end{proof}

\begin{remark}
When $L$ is a rational multiplier of $\frac{\pi}{2}$, one double eigenvalue
in the sequence $\{ n^2 \}_{n \in \mathbb{N}}$ is actually a triple eigenvalue.
Indeed, if $L = \frac{\pi m}{2 n}$, then in addition to the eigenfunctions (\ref{eigenvalues-1}) and (\ref{eigenvalues-2}),
we obtain the third eigenfunction for $\lambda = n^2$. If $m$ is even, the third
eigenfunction is given by
\begin{equation}
\label{eigenvalues-6}
\left\{ \begin{array}{lr} u_0(x) = \cos(n x), & x \in I_0, \\
u_+(x) = (-1)^{n+\frac{m}{2}} \cos[n(x-L-\pi)], & x \in I_+, \\
u_-(x) = u_+(-x), & x \in I_- \end{array} \right.
\end{equation}
whereas if $m$ is odd, the eigenfunction is given by
\begin{equation}
\label{eigenvalues-7}
\left\{ \begin{array}{lr} u_0(x) = \sin(n x), & x \in I_0, \\
u_+(x) = (-1)^{n+\frac{m-1}{2}} \cos[n(x-L-\pi)], & x \in I_+, \\
u_-(x) = -u_+(-x), & x \in I_- \end{array} \right.
\end{equation}
We say that a resonance occurs if $L$ is a rational multiplier of $\frac{\pi}{2}$.
\end{remark}

\begin{figure}[htp]
\includegraphics[width=12cm]{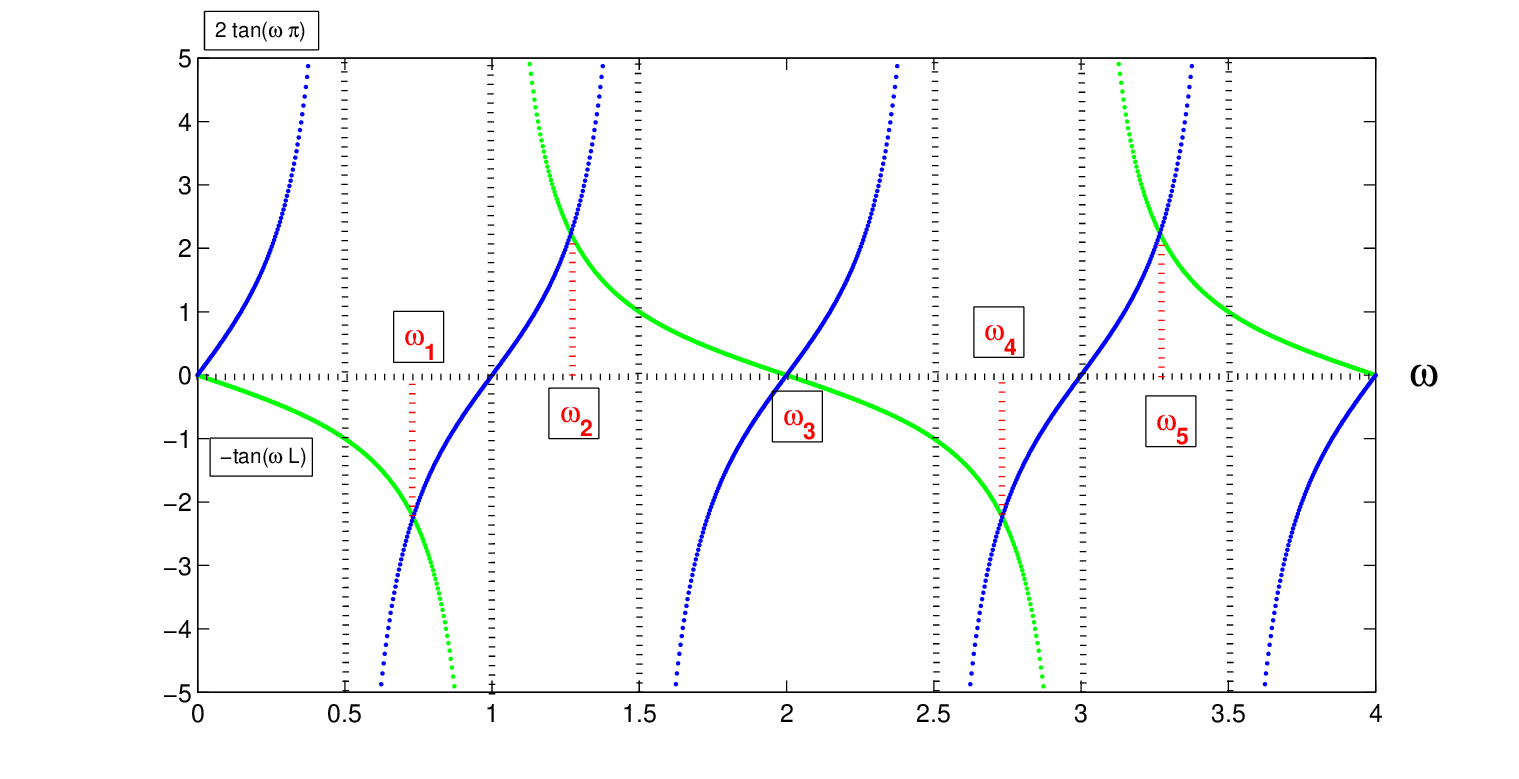}
\includegraphics[width=12cm]{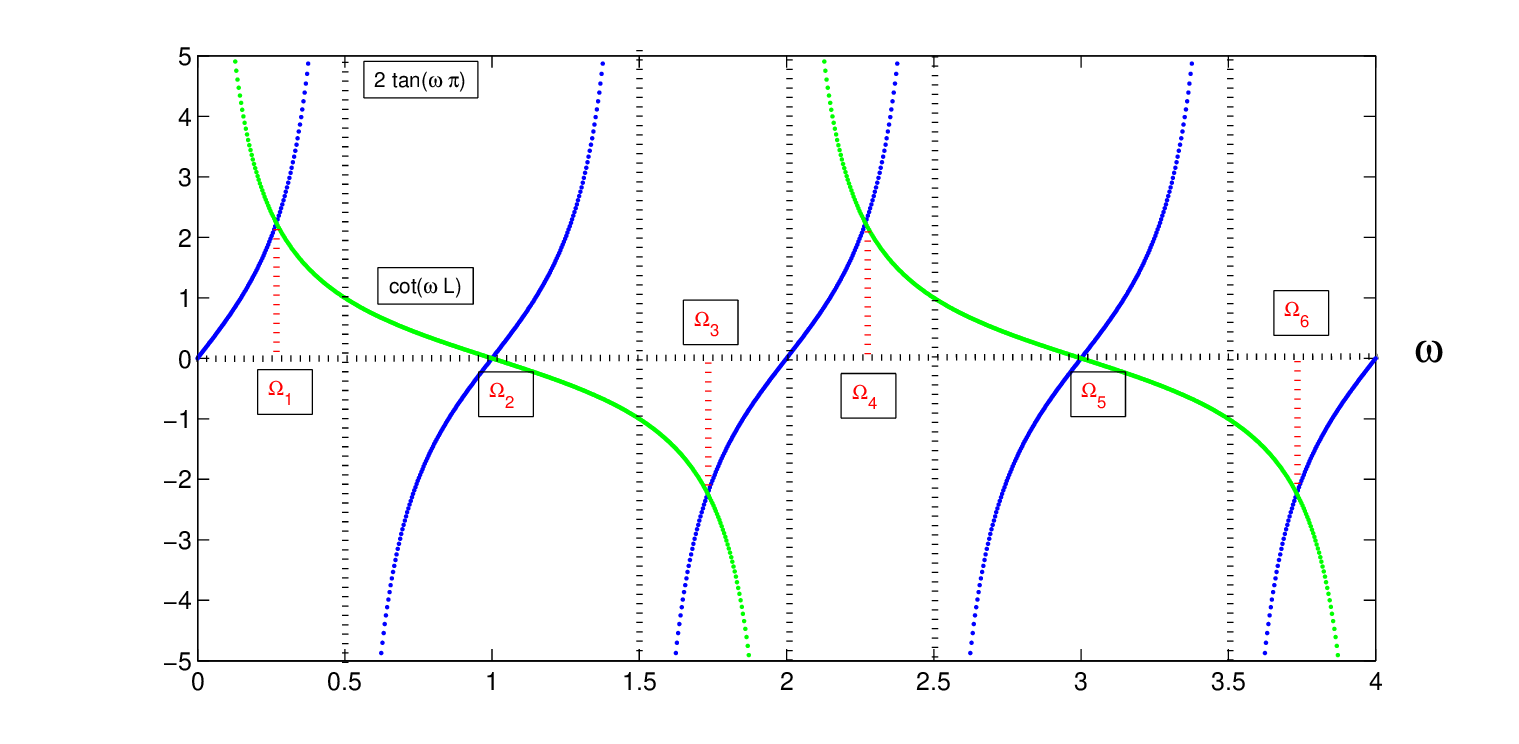}
\caption{Graphical solutions of the dispersion relations (\ref{transc-1}) (top) and (\ref{transc-2}) (bottom) for $L = \frac{\pi}{2}$.}
\label{fig2}
\end{figure}

Figure \ref{fig2} show graphical solutions of the dispersion relations (\ref{transc-1}) and (\ref{transc-2}) for $L = \frac{\pi}{2}$
with the roots $\{ \omega_n \}_{n \in \mathbb{N}}$ and $\{ \Omega_n \}_{n \in \mathbb{N}}$ clearly marked.
The graphical solution persists for any $L < \pi$ as seen from the vertical asymptotics of the functions
$\tan(\omega \pi)$, $\tan(\omega L)$, and $\cot(\omega L)$. As a result, for every $L < \pi$
the first positive roots of the dispersion relations (\ref{transc-1}) and (\ref{transc-2}) satisfy
\begin{equation}
\label{order-1}
0 < \Omega_1 < \frac{1}{2} < \omega_1 < \min\left\{1,\frac{\pi}{2L}\right\} \leq \Omega_2 < \ldots
\end{equation}
With similar analysis, it follows that for every $L \geq \pi$, the first roots satisfy
\begin{equation}
\label{order-2}
0 < \Omega_1 < \frac{\pi}{2L} \leq \omega_1 \leq \min\left\{\frac{1}{2},\frac{\pi}{L}\right\} < \Omega_2 < \ldots
\end{equation}
In either case, the smallest positive eigenvalue $\Omega_1^2 \in \sigma(-\Delta)$ corresponds to the odd eigenfunction
with respect to the middle point in $I_0$, whereas the second positive eigenvalue $\omega_1^2 \in \sigma(-\Delta)$
corresponds to the even eigenfunction. As is shown in Section \ref{sec-bifurcation},
the order of eigenvalues in (\ref{order-1}) or (\ref{order-2}) is important for analysis of
the first two bifurcations of the constant standing wave of the stationary NLS equation (\ref{statNLS}).

\begin{remark}
\label{remark-roots}
Since $\Omega_1 < \frac{1}{2}$, $\Omega_1 < \omega_1 < \Omega_2$, and $\omega_1 < 1$, the first two
positive eigenvalues in $\sigma(-\Delta)$ are simple, according to Proposition \ref{prop-spectrum}.
\end{remark}

For further references, we give the explicit expression for the eigenfunction $U$
corresponding to the smallest positive eigenvalue $\Omega_1^2$ in $\sigma(-\Delta)$:
\begin{equation}
\label{eigenfunction-Omega-1}
U(x) = \left\{ \begin{array}{l}
-\frac{\sin(\Omega_1 L)}{\cos(\Omega_1 \pi)} \cos(\Omega_1(x+L+\pi)), \quad x \in I_-, \\
\sin(\Omega_1 x), \quad \quad \quad \quad \quad \quad \quad \quad \quad \quad x \in I_0, \\
\frac{\sin(\Omega_1 L)}{\cos(\Omega_1 \pi)} \cos(\Omega_1(x-L-\pi)), \quad \quad x \in I_+.
\end{array} \right.
\end{equation}
The value $\Omega_1$ is found from the first positive root of the transcendental equation
\begin{equation}
\label{eigenvalue-Omega-1}
2 \sin(\omega \pi) \sin(\omega L) = \cos(\omega \pi) \cos(\omega L).
\end{equation}
The following proposition summarizes on properties of the root $\Omega_1$.

\begin{proposition}
\label{prop-root}
For every $L > 0$, the first positive root $\Omega_1$ of the dispersion relation (\ref{eigenvalue-Omega-1}) satisfies
$\Omega_1 < \frac{1}{2}$ and $\Omega_1 < \frac{\pi}{2L}$. Moreover, $\Omega_1 \to \frac{1}{2}$ as $L \to 0$ and $2 L \Omega_1 \to \pi$ as $L \to
\infty$.
\end{proposition}

\begin{proof}
The bounds $\Omega_1 < \frac{1}{2}$ and $\Omega_1 < \frac{\pi}{2L}$ follow from the orderings (\ref{order-1}) and (\ref{order-2}).

From the bound $\Omega_1 < \frac{1}{2}$ and the algebraic equation (\ref{eigenvalue-Omega-1}), we realize
that $\sin(L \Omega_1) \to 0$ and $\cos(\pi \Omega_1) \to 0$ as $L \to 0$. Therefore, $\Omega_1 \to \frac{1}{2}$ as $L \to 0$.

On the other hand, from the bound $\Omega_1 < \frac{\pi}{2L}$ and the algebraic equation (\ref{eigenvalue-Omega-1}), we realize
that $\sin(\pi \Omega_1) \to 0$ and $\cos(L \Omega_1) \to 0$ as $L \to \infty$. Therefore, $2 L \Omega_1 \to \pi$ as $L \to \infty$.
\end{proof}

\section{Proof of Theorem \ref{theorem-graph}}
\label{sec-bifurcation}

As we discussed in the introduction, critical points of the energy functional
$H_{\Lambda} := E - \Lambda Q$ in $\mathcal{E}(\Delta)$
are equivalent to the strong solutions of the stationary NLS equation (\ref{statNLS}).
Among all solutions of the stationary NLS equation (\ref{statNLS}), we are interested
in the solutions that yield the minimum of energy denoted by $E_0$ at the fixed charge
denoted by $Q_0$. These solutions correspond to the ground states
of the constrained minimization problem (\ref{minimizer}).
There exists a mapping between parameters $\Lambda$ and $Q_0$
for the ground state solutions.

The first elementary result describes the local bifurcation of the constant standing wave
(\ref{ground-1}) as a ground state of the minimization problem (\ref{minimizer}) for small charge $Q_0$.

\begin{lemma}
\label{lemma-1}
There exists $Q_0^*$ such that the constant standing wave (\ref{ground-1}) is a ground
state of the constrained minimization problem (\ref{minimizer}) for $Q_0 \in (0,Q_0^*)$.
\end{lemma}

\begin{proof}
The spectrum of $-\Delta$ equipped with the domain $\mathcal{D}(\Delta) \subset L^2(I_- \cup I_0 \cup I_+)$
is given by eigenvalues described in Proposition \ref{prop-spectrum}.
The simple zero eigenvalue is the lowest eigenvalue in
$\sigma(-\Delta)$ corresponding to the constant eigenfunction.

By the local bifurcation theory \cite{Kapitula,Pelin-book},
the ground state of the constrained minimization problem
(\ref{minimizer}) for small values of $Q_0$ is the standing wave bifurcating
from the constant eigenfunction of $-\Delta$ in $\mathcal{D}(\Delta)$. Since
the constant solution $\Phi(x) = p$ exists for
the stationary NLS equation (\ref{statNLS}) with $\Lambda = -2p^2$ for every $\Lambda < 0$,
the bifurcating ground state is the constant standing wave given by (\ref{ground-1}).
The relation between $Q_0$ and $\Lambda$ is computed explicitly from the definition (\ref{charge}):
$$
Q_0 = 2 (L+2\pi) p^2 = (L+2\pi) |\Lambda|.
$$
This concludes the proof of the lemma.
\end{proof}

We shall now consider bifurcations of new standing waves of the stationary NLS equation (\ref{statNLS})
from the constant solution (\ref{ground-1}). For both $L < \pi$ and $L \geq \pi$, the lowest nonzero eigenvalue in
$\sigma(-\Delta)$ in Proposition \ref{prop-spectrum} is the positive eigenvalue $\Omega_1^2$,
which corresponds to the odd eigenfunction in $I_- \cup I_0 \cup I_+$, see orderings
(\ref{order-1}) and (\ref{order-2}). This smallest nonzero eigenvalue induces the symmetry-breaking
bifurcation of the ground state of the constrained minimization problem (\ref{minimizer})
at $Q_0^*$. This bifurcation only marks the first bifurcation in a sequence of bifurcations
of new standing waves from the constant solution (\ref{ground-1}).
The second bifurcation is induced due to the positive eigenvalue $\omega_1^2$ in $\sigma(-\Delta)$,
which corresponds to the even eigenfunction in $I_- \cup I_0 \cup I_+$.
The corresponding result is described in the following lemma.

\begin{lemma}
\label{lemma-2}
For $\Lambda < -\frac{1}{2} \Omega_1^2$,
the constant solution (\ref{ground-1}) is no longer the ground state
of the constrained minimization problem (\ref{minimizer}) for
$Q_0 > Q_0^* := \frac{1}{2} (L+2\pi) \Omega_1^2$. Moreover,
the constant solution (\ref{ground-1}) is a saddle point of $E(\Psi)$ under fixed $Q(\Psi) = Q_0$
with one negative eigenvalues for $Q_0 \in (Q_0^*,Q_0^{**})$ and
two negative eigenvalues for $Q_0 \gtrsim Q_0^{**}$, where $Q_0^{**} := \frac{1}{2} (L+2\pi) \omega_1^2$.
\end{lemma}

\begin{proof}
When the perturbation to the constant solution (\ref{ground-1}) is decomposed into
the real and imaginary parts $U$ and $W$, the second variation of $H_{\Lambda}
= E - \Lambda Q$ is defined by two Schr\"{o}dinger operators $L_+$ and $L_-$ as follows:
\begin{equation}
\label{second-variation}
\delta^2 H_{\Lambda} := \langle L_+ U, U \rangle_{L^2} + \langle L_- W, W \rangle_{L^2},
\end{equation}
where
\begin{align}
\label{L-plus}
L_+ & =  - \Delta - \Lambda - 6 \Phi^2, \\
\label{L-minus}
L_- & =  - \Delta - \Lambda - 2 \Phi^2.
\end{align}
Because $\Phi^2$ is bounded on $I_- \cup I_0 \cup I_+$, the operators $L_+$ and $L_-$ are also
defined on the domain $\mathcal{D}(\Delta) \subset L^2(I_- \cup I_0 \cup I_+)$.
If $\Phi$ is the constant solution (\ref{ground-1}), then
$$
L_- = -\Delta \quad \mbox{\rm and} \quad L_+ = -\Delta + 2 \Lambda.
$$
Therefore, the smallest eigenvalue of $L_-$ is located at zero and it is simple.
It corresponds to the phase rotation of the standing wave $\Phi$.

Let us now recall the theory of constrained minimization from Shatah--Strauss \cite{SS} and Weinstein \cite{W}.
If $L_+$ has only one negative eigenvalue, then $\Phi$ is not a minimizer of energy $H_{\Lambda}$.
Nevertheless, it is a constrained minimizer of energy $H_{\Lambda}$
if a certain slope condition is satisfied. Associated with the constraint $Q(\Psi) = Q_0$, we can introduce
the constrained $L^2$ space by
\begin{equation}
\label{constrained-L-2}
L^2_c := \left\{ U \in L^2(I_- \cup I_0 \cup I_+) : \quad \langle U, \Phi \rangle_{L^2} = 0 \right\}.
\end{equation}
The number of negative eigenvalues of $L_+$ is reduced by one in the constrained space
(\ref{constrained-L-2}) if and only if $\frac{d}{d\Lambda} Q(\Phi) \leq 0$ \cite{SS,W}.
Moreover, if $\frac{d}{d \Lambda} Q(\Phi) < 0$, the constrained minimizer is non-degenerate
with respect to the real part of the perturbation $U$.

The smallest eigenvalue of $L_+$ is negative and the second eigenvalue of $L_+$ is located at
$\Omega_1^2 + 2 \Lambda$. Since $\frac{d}{d\Lambda} Q(\Phi) = -(L+2\pi) < 0$,
the constant standing wave (\ref{ground-1}) is a constrained minimizer of $E(\Psi)$
if $\Lambda \in (-\frac{1}{2} \Omega_1^2,0)$ but it is no longer the ground state
if $\Lambda \in (-\infty,-\frac{1}{2} \Omega_1^2)$.

By Remark \ref{remark-roots}, the eigenvalues $\Omega_1^2$ and $\omega_1^2$ in $\sigma(-\Delta)$
are simple. As a result, operator $L_+$ has exactly two negative eigenvalues for $-\frac{1}{2} \omega_1^2 < \Lambda < -\frac{1}{2} \Omega_1^2$
and exactly three negative eigenvalues for $\Lambda \lesssim -\frac{1}{2} \omega_1^2$.
The number of negative eigenvalues of $L_+$ is reduced by one in the constrained space $L^2_c$.
Therefore, the constant solution (\ref{ground-1})
is a saddle point of $E(\Psi)$ under fixed $Q(\Psi) = Q_0$
with exactly one negative eigenvalue for $Q_0 \in (Q_0^{*},Q_0^{**})$
and exactly two negative eigenvalues for $Q_0 \gtrsim Q_0^{**}$,
where $Q_0^{*} := \frac{1}{2} (L + 2\pi) \Omega_1^2$ and
$Q_0^{**} := \frac{1}{2} (L + 2\pi) \omega_1^2$.
\end{proof}

\begin{remark}
Lemmas \ref{lemma-1} and \ref{lemma-2} prove the first assertion of Theorem \ref{theorem-graph}.
\end{remark}

We shall now describe the first (symmetry-breaking) bifurcation of the constant standing wave (\ref{ground-1})
at $\Lambda = -\frac{1}{2} \Omega_1^2$ or $Q_0 = Q_0^*$. We will show that it is a pitchfork bifurcation, which leads to a family
of positive asymmetric standing waves of the stationary NLS equation (\ref{statNLS})  for $\Lambda \lesssim -\frac{1}{2} \Omega_1^2$.
The asymmetric states are the ground states of the constrained minimization problem (\ref{minimizer})
for $Q_0 \gtrsim Q_0^*$. The corresponding result is described in the following lemma.

\begin{lemma}
\label{lemma-3}
Let $\Lambda_0 := -\frac{1}{2}\Omega_1^2$. There exists $\Lambda_* \in \left(-\infty,\Lambda_0 \right)$ such that
the stationary NLS equation (\ref{statNLS}) with $\Lambda \in (\Lambda_*,\Lambda_0)$
admits a positive asymmetric standing wave
$\Phi$, which converges to the constant solution (\ref{ground-1})
in the $H^2$-norm as $\Lambda \to \Lambda_0$.
Moreover, there exists $\tilde{Q}_0^{*} \in (Q_0^*,\infty)$ such that
the positive asymmetric standing wave is a ground state
of the constrained minimization problem (\ref{minimizer}) for $Q_0 \in (Q_0^*,\tilde{Q}_0^{*})$.
\end{lemma}

\begin{proof}
By Proposition \ref{prop-spectrum}, the eigenfunction $U$
corresponding to the eigenvalue $\Omega_1^2$ in $\sigma(-\Delta)$ is odd
with respect to the central point $x = 0$, whereas the constant
solution (\ref{ground-1}) is even. Therefore,
we have the symmetry-breaking bifurcation, which is similar
to the one studied in \cite{KPP}. In order to unfold the bifurcation,
we study how the odd mode $U$ can be continued as $\Lambda$ is defined near
the bifurcation value $\Lambda_0 := -\frac{1}{2}\Omega_1^2$.
We use the explicit expression for $U$ given by (\ref{eigenfunction-Omega-1})
and the characterization of the values of $\Omega_1$ given by Proposition \ref{prop-root}.
By Remark \ref{remark-roots}, $\Omega_1^2$ is a simple eigenvalue in $\sigma(-\Delta)$.

Using a simplified version of the Lyapunov--Schmidt reduction method \cite{KPP},
we consider a regular perturbation expansion for
solutions of the stationary NLS equation (\ref{statNLS})
near the constant solution (\ref{ground-1}) at $\Lambda = \Lambda_0$. Thus, we expand
\begin{equation}
\label{bif-sol}
\Lambda = -\frac{1}{2} \Omega_1^2 + a^2 \Omega + \mathcal{O}(a^4), \quad
\Phi(x) = \frac{1}{2} \Omega_1 + a U(x) + a^2 \Phi_2(x) + a^3 \Phi_3(x) + \mathcal{O}_{H^2}(p^4),
\end{equation}
where $a$ is a small parameter for the amplitude of the critical odd eigenfunction $U$
of the operator $-\Delta$ in $\mathcal{D}(\Delta)$,
whereas the corrections $\Omega$ and $\{ \Phi_n \}_{n \geq 2} \in \mathcal{D}(\Delta)$
are defined uniquely under the constraints $\langle U, \Phi_k \rangle_{L^2} = 0$, $k \geq 2$.
Note that the decomposition (\ref{bif-sol}) already incorporates the near-identity transformation
that removes quadratic terms in $a$ and ultimately leads to the normal-form equation,
derived in a similar context in \cite{KPP}.

At $a^2$, we obtain the inhomogeneous linear equation
\begin{equation}
\label{Phi-2-inhom-equation}
(-\Delta - \Omega_1^2) \Phi_2 = \frac{1}{2} \Omega_1 (\Omega + 6 U^2).
\end{equation}
Since $U$ is odd, the right-hand-side is even.  Thus, the Fredholm solvability condition is
satisfied and there exists a unique even solution for $\Phi_2$, which can be represented
in the form
\begin{equation}
\label{Phi-2}
\Phi_2(x) = -\frac{\Omega}{2 \Omega_1} + 3 \Omega_1 \tilde{\Phi}_2(x),
\end{equation}
where $\tilde{\Phi}_2$ is uniquely defined from the linear inhomogeneous equation
\begin{equation}
\label{solvability-0}
(-\Delta - \Omega_1^2) \tilde{\Phi}_2 = U^2.
\end{equation}

At $a^3$, we obtain another inhomogeneous linear equation
\begin{equation}
(-\Delta - \Omega_1^2) \Phi_3 = \Omega U + 6 \Omega_1 U \Phi_2 + 2 U^3.
\end{equation}
The right-hand-side is now odd and the Fredholm solvability condition produces a nontrivial
equation for $\Omega$:
\begin{equation}
\label{solvability}
\Omega \| U \|_{L^2}^2 + 6 \Omega_1 \langle U^2, \Phi_2 \rangle_{L^2} + 2 \| U \|_{L^4}^4 = 0.
\end{equation}
Substituting (\ref{Phi-2}) into (\ref{solvability}), we obtain
\begin{equation}
\label{solvability-1}
\Omega \| U \|_{L^2}^2 = 9 \Omega_1^2 \langle U^2, \tilde{\Phi}_2 \rangle_{L^2} + \| U \|_{L^4}^4.
\end{equation}

We need to show that the right-hand-side of equation (\ref{solvability-1}) is negative, which
yields $\Omega < 0$. In view of the decomposition (\ref{bif-sol}),
for $a$ sufficiently small, the new solution $\Phi$ represents a positive asymmetric
standing wave satisfying the stationary NLS equation (\ref{statNLS}) with $\Lambda \lesssim \Lambda_0$.
This would imply the first assertion of the lemma.

Using the explicit representations (\ref{eigenfunction-Omega-1}) and (\ref{eigenvalue-Omega-1}),
we obtain an explicit solution of the linear inhomogeneous equation (\ref{solvability-0}):
\begin{equation*}
\quad \quad \quad
\tilde{\Phi}_2(x) = \left\{ \begin{array}{l}
A \cos(\Omega_1 x) - \frac{1}{6 \Omega_1^2} \left[ \cos(2 \Omega_1 x) + 3 \right], \quad \quad \quad
\quad \quad \quad \quad \quad \quad \quad \quad \quad \quad \quad \;\; x \in I_0, \\
B \cos(\Omega_1 (x-L-\pi)) + \frac{\sin^2(L \Omega_1)}{6 \Omega_1^2 \cos^2(\pi \Omega_1)} \left[
\cos(2 \Omega_1(x-L-\pi)) - 3 \right], \quad \quad x \in I_+, \end{array} \right.
\end{equation*}
where $A$ and $B$ are constants of integration to be defined, the symmetry $\tilde{\Phi}_2(-x) = \tilde{\Phi}_2(x)$
can be used, and the homogeneous sinusoidal solutions are thrown away since $\sin(\pi \Omega_1) \neq 0$.

Using Kirchhoff boundary conditions (\ref{bc-2}), we uniquely determine constants $A$ and $B$ from
the following linear system of algebraic equations:
\begin{eqnarray}
\label{lin-syst-A-B} \quad \quad
\left[ \begin{matrix} \cos(L \Omega_1) & - \cos(\pi \Omega_1) \\
\sin(L \Omega_1) & 2 \sin(\pi \Omega_1) \end{matrix} \right]
\left[ \begin{matrix} A \\ B \end{matrix} \right]
= \frac{1}{6 \Omega_1^2}
\left[ \begin{matrix}
3 + \cos(2 L \Omega_1) + \frac{\sin^2(L \Omega_1)}{\cos^2(\pi \Omega_1)} \left( \cos(2 \pi \Omega_1) - 3 \right) \\
2 \sin(2 L \Omega_1) - 4 \frac{\sin^2(L \Omega_1)}{\cos^2(\pi \Omega_1)} \sin(2 \pi \Omega_1)
\end{matrix} \right].
\end{eqnarray}
Using the transcendental equation (\ref{eigenvalue-Omega-1}), we can see that the second entry in
the right-hand side of (\ref{lin-syst-A-B})
is zero. As a result, the second equation of the system (\ref{lin-syst-A-B}) yields
\begin{equation}
\label{transcendental-expr-0}
\sin(L \Omega_1) A + 2 \sin(\pi \Omega_1) B = 0.
\end{equation}
Let us prove that the system (\ref{lin-syst-A-B}) admits the unique solution in the following explicit form:
\begin{equation}
\label{lin-solution-A-B}
A = \frac{1}{2 \Omega_1^2} \cos^3(L \Omega_1), \quad B = -\frac{\sin^2(L \Omega_1) \cos^2(L \Omega_1)}{2 \Omega_1^2 \cos(\pi \Omega_1)}.
\end{equation}
Indeed, the constraint (\ref{transcendental-expr-0}) is satisfied with the solution (\ref{lin-solution-A-B}). Furthermore,
the first equation of the system (\ref{lin-syst-A-B}) is satisfied with the exact solution (\ref{lin-solution-A-B}) if and only if
the following transcendental equation is met:
\begin{equation}
\label{trascendental-expr-1}
1 - \frac{\sin^2(L \Omega_1)}{\cos^2(\pi \Omega_1)} + \frac{1}{3} \left[ \cos(2 L \Omega_1)
+ \frac{\sin^2(L \Omega_1)}{\cos^2(\pi \Omega_1)} \cos(2 \pi \Omega_1) \right] = \cos^2(L \Omega_1).
\end{equation}
Using the transcendental equation (\ref{eigenvalue-Omega-1}), we rewrite equation (\ref{trascendental-expr-1}) in the equivalent form
\begin{equation}
\label{trascendental-expr-2}
1 - \frac{\sin^2(L \Omega_1)}{\cos^2(\pi \Omega_1)} = \frac{3}{4} \cos^2(L \Omega_1),
\end{equation}
which is satisfied identically, thanks again to the transcendental equation (\ref{eigenvalue-Omega-1}).

Using (\ref{lin-solution-A-B}) in the expression for $\tilde{\Phi}_2$,
we rewrite the expression for $U^2 + 9 \Omega_1^2 \tilde{\Phi}_2$ explicitly as
\begin{equation*}
\left\{ \begin{array}{l}
\frac{9}{2} \cos^3(L \Omega_1) \cos(\Omega_1 x) -2  \cos(2 \Omega_1 x) - 4, \quad \quad \quad \quad \quad \quad \quad
\quad \quad \quad \quad \quad \quad \quad \quad \quad \quad \quad \quad \;\; x \in I_0, \\
\frac{\sin^2(L \Omega_1)}{\cos^2(\pi \Omega_1)} \left[ -\frac{9}{2} \cos^2(L \Omega_1) \cos(\pi \Omega_1) \cos(\Omega_1 (x-L-\pi))
+ 2 \cos(2 \Omega_1(x-L-\pi)) - 4 \right], \quad  x \in I_+. \end{array} \right.
\end{equation*}
After computations of the integrals and simplifications with the help of the transcendental equation (\ref{eigenvalue-Omega-1}),
the right-hand side of equation (\ref{solvability-1}) is simplified to the form
\begin{eqnarray}
\label{positomega}
&  9 \Omega_1^2 \langle U^2, \tilde{\Phi}_2 \rangle_{L^2}  + \| U \|_{L^4}^4  \\
\notag
& = \hspace{.1cm} - 3 L \left( 1 - \frac{\sin (2 L \Omega_1)}{ 2 L \Omega_1} \right)
-6 \pi \frac{\sin^4(L \Omega_1)}{\cos^4(\pi \Omega_1)}
   -  \frac{3 \sin^3(L \Omega_1) \cos(L \Omega_1)}{\Omega_1 \cos^2(\pi \Omega_1)}
\left[ 1 + 2 \cos^2(L \Omega_1) \right].
\end{eqnarray}
By Proposition \ref{prop-root}, we have $2 L \Omega_1 < \pi$ for every $L$,
so that every term in (\ref{positomega}) is negative. Therefore, $\Omega < 0$ in \eqref{solvability-1}.

Thus, the first assertion of the lemma is proved. In order to prove the second assertion of the lemma,
which states that the positive asymmetric standing wave $\Phi$ given by the decomposition (\ref{bif-sol})
is a minimizer of the constrained minimization problem (\ref{minimizer}) for $Q_0 \gtrsim Q_0^*$, we need to
compute the negative eigenvalues of the operators $L_+$ and $L_-$ in (\ref{L-plus}) and (\ref{L-minus})
for $\Phi$. Since $L_- \Phi = 0$ and $\Phi$ is positive, the zero
eigenvalue of $L_-$ is the smallest eigenvalue of $L_-$. The smallest eigenvalue is simple. It corresponds
to the phase rotation of the standing wave $\Phi$.

Since $\Phi$ bifurcates from the constant solution (\ref{ground-1}), the operator $L_+$ has a unique
negative eigenvalue and a simple zero eigenvalue at $\Lambda = \Lambda_0$. We shall now construct a regular perturbation
expansion for small $|\Lambda - \Lambda_0|$ in order to prove that the zero eigenvalue becomes
a small positive eigenvalue for the bifurcating solution (\ref{bif-sol}). Therefore, we expand
\begin{eqnarray*}
L_+ = - \Delta - \Omega_1^2 - 6 a \Omega_1 U - 6 a^2 \Omega_1 \Phi_2 - 6 a^2 U^2 - a^2 \Omega +
\mathcal{O}_{L^{\infty}}(a^3).
\end{eqnarray*}
The zero eigenvalue of $-\Delta - \Omega_1^2$ corresponds again to the eigenfunction $U$.
For $a$ sufficiently small, we expand the eigenvalue $\lambda$ and the eigenfunction $u$
of the operator $L_+$:
\begin{equation}
\label{bif-eig}
\lambda = a^2 \Lambda + \mathcal{O}(a^3), \quad
u(x) = U(x) + a U_1(x) + a^2 U_2(x) + \mathcal{O}_{H^2}(a^3),
\end{equation}
where the  corrections $\Lambda$ and $\{ U_n \}_{n \geq 1} \in \mathcal{D}(\Delta)$ are defined uniquely
under the constraints $\langle U, U_k \rangle_{L^2} = 0$, $k \geq 1$.
At $a$, we obtain the inhomogeneous linear equation
\begin{equation}
(-\Delta - \Omega_1^2) U_1 = 6 \Omega_1 U^2,
\end{equation}
which has the unique even solution $U_1 = 6 \Omega_1 \tilde{\Phi}_2$.
At $a^2$, we obtain the inhomogeneous linear equation
\begin{equation}
(-\Delta - \Omega_1^2) U_2 = 6 \Omega_1 U U_1 + 6 \Omega_1 \Phi_2 U + 6 U^3 - \Omega U + \Lambda U.
\end{equation}
With the account of (\ref{Phi-2}), (\ref{solvability-1}), and $U_1 = 6 \Omega_1 \tilde{\Phi}_2$,
the Fredholm solvability condition yields
\begin{equation}
\label{solvability-2}
\Lambda \| U \|_{L^2}^2 = -18 \Omega_1^2 \langle U^2, \tilde{\Phi}_2 \rangle_{L^2} - 2 \| U \|_{L^4}^4.
\end{equation}
Comparison with (\ref{solvability-1}) yields $\Lambda = - 2 \Omega$.
Since we have already proved that $\Omega < 0$, we obtain $\Lambda > 0$, so that
for $|\Lambda - \Lambda_0|$ sufficiently small, the operator $L_+$ has a small positive
eigenvalue bifurcating from the zero eigenvalue as $\Lambda \to \Lambda_0$. Thus,
the operator $L_+$ has only one simple negative eigenvalue for $\Lambda \lesssim \Lambda_0$.

It remains to show that $Q$ computed at the positive asymmetric standing wave $\Phi$
given by the decomposition (\ref{bif-sol}) is an increasing function of the amplitude
parameter $a$. In this case, the slope condition $\frac{d}{d\Lambda} Q(\Phi) < 0$ holds and
the only negative eigenvalue of operator $L_+$ is removed by the constraint in $L^2_c$
defined by (\ref{constrained-L-2}). From (\ref{charge}), (\ref{bif-sol}), (\ref{Phi-2-inhom-equation}), and (\ref{solvability-0}),
we obtain
\begin{eqnarray*}
Q(\Phi) & = & Q_0^* + a^2 \left( \Omega_1 \langle 1, \Phi_2 \rangle_{L^2} + \| U \|^2_{L^2} \right) + \mathcal{O}(a^3) \\
& = & Q_0^* - a^2 \left( \Omega (L + 2\pi) + 2 \| U \|^2_{L^2} \right) + \mathcal{O}(a^3).
\end{eqnarray*}
We observe $\Omega (L + 2\pi) + 2 \| U \|^2_{L^2} < 0$ if and only if
\begin{equation}
\label{inequality-to-prove}
(L+2 \pi ) (9 \Omega_1^2 \langle U^2, \tilde{\Phi}_2 \rangle_{L^2}  + \| U \|_{L^4}^4) + 2 \| U \|_{L^2}^4  <0.
\end{equation}
Using \eqref{positomega} and
\begin{eqnarray*}
\| U \|_{L^2}^2 & = & L - \sin(L \Omega_1) \cos(L \Omega_1) + 2 \frac{\sin^2(L \Omega_1)}{\cos^2(\pi \Omega_1)}
\left[ \pi + \sin(\pi \Omega_1) \cos(\pi \Omega_1) \right] \\
& = & L + 2 \pi \frac{\sin^2(L \Omega_1)}{\cos^2(\pi \Omega_1)},
\end{eqnarray*}
the left-hand side of (\ref{inequality-to-prove}) can be written as
\begin{eqnarray*}
 &  -3(L + 2 \pi ) \left\{   L \left( 1 - \frac{\sin (2 L \Omega_1)}{ 2 L \Omega_1} \right)    + 2\pi
\frac{\sin^4(L \Omega_1)}{\cos^4(\pi \Omega_1)}  + \frac{ \sin^3(L \Omega_1) \cos(L \Omega_1)}{\Omega_1 \cos^2(\pi \Omega_1)}
\left[ 1 + 2 \cos^2(L \Omega_1) \right]  \right\}    \\
& \hspace{1cm} + 2 \left[ L + 2 \pi   \frac{\sin^2(L \Omega_1)}{\cos^2(\pi \Omega_1)}   \right]^2.
\end{eqnarray*}
We regroup these terms as the sum $I + II + III$, where
\begin{align*}
I & =   -3(L + 2 \pi ) \left[   \frac{3L}{4}  + 2\pi
\frac{\sin^4(L \Omega_1)}{\cos^4(\pi \Omega_1)} \right] + 2 \left[ L + 2 \pi   \frac{\sin^2(L \Omega_1)}{\cos^2(\pi \Omega_1)}   \right]^2 \\
II & = -\frac{3}{4} (L + 2 \pi ) L \left[ 1 - \frac{\sin (2 L \Omega_1)}{ 2 L \Omega_1} \right] \\
III & =   -3(L + 2 \pi ) \frac{\sin(L \Omega_1) \cos(L \Omega_1)}{\Omega_1} \left\{
\frac{ \sin^2(L \Omega_1)}{\cos^2(\pi \Omega_1)}
\left[ 1 + 2 \cos^2(L \Omega_1) \right] - \frac{3}{4} \right\}.
\end{align*}
After expanding the brackets, the first term becomes
$$
I = -\frac{1}{4} L^2 -\frac{1}{2} \pi L \left[ 9 - 16 \frac{\sin^2(L \Omega_1)}{\cos^2(\pi \Omega_1)}
+ 12 \frac{\sin^4(L \Omega_1)}{\cos^4(\pi \Omega_1)} \right] - 4 \pi^2 \frac{\sin^4(L \Omega_1)}{\cos^4(\pi \Omega_1)},
$$
where every term is negative because $f(x) = 9 - 16 x^2 + 12 x^4 \geq \frac{11}{3} > 0$. Hence, $I < 0$.
Furthermore, $II < 0$ holds. Since $L \Omega_1 < \frac{\pi}{2}$ by Proposition \ref{prop-root},
then $III < 0$ if and only if $g(\Omega_1) > 0$, where
$$
g(\Omega_1) := \frac{ \sin^2(L \Omega_1)}{\cos^2(\pi \Omega_1)}
\left[ 1 + 2 \cos^2(L \Omega_1) \right] - \frac{3}{4}.
$$
Using equation (\ref{trascendental-expr-2}), we rewrite the left-hand-side as follows:
\begin{eqnarray*}
g(\Omega_1) & = & \left[ 1 - \frac{3}{4} \cos^2(L \Omega_1) \right]
\left[ 1 + 2 \cos^2(L \Omega_1) \right] - \frac{3}{4}\\
& = & \frac{1}{4} \left[ 1 - \cos^2(L \Omega_1) \right] \left[ 1 + 6 \cos^2(L \Omega_1) \right],
\end{eqnarray*}
from which it follows that $g(\Omega_1) > 0$ for every $L > 0$. Thus, $III < 0$, so that $Q(\Phi)$ is an increasing
function of the amplitude parameter $a$.

Since $\frac{d}{d\Lambda} Q(\Phi) < 0$, the operator $L_+$ does not have a negative eigenvalue
in the constrained space $L^2_c$, so that $\Phi$ is a ground state of the constrained minimization
problem (\ref{minimizer}). The statement of the lemma is proved.
\end{proof}

\begin{remark}
\label{remark-4}
Asymptotic expansions similar to the ones used in the proof of Lemma \ref{lemma-3} can be
developed for the second (symmetry-preserving) bifurcation of the constant standing wave (\ref{ground-1})
at $\Lambda = -\frac{1}{2} \omega_1^2$ or $Q_0 = Q_0^{**} := \frac{1}{2} (L+2\pi) \omega_1^2$,
where the positive eigenvalue $\omega_1^2$ in $\sigma(-\Delta)$
corresponds to the even eigenfunction. As a result of this bifurcation, a new family of
positive symmetric standing waves exists  for $\Lambda \lesssim -\frac{1}{2} \omega_1^2$.
The positive symmetric wave is not, however, the ground state of the constrained minimization problem
(\ref{minimizer}) for $Q_0 \approx Q_0^{**}$, because the operator $L_+$ has two negative
eigenvalue and a simple zero eigenvalue at $\Lambda = -\frac{1}{2} \omega_1^2$.
Even if the zero eigenvalue becomes small positive eigenvalue for $\Lambda \lesssim - \frac{1}{2} \omega_1^2$
and if one negative eigenvalue is removed by a constraint in $L^2_c$, there operator $L_+$ still has one
negative eigenvalue in $L^2_c$.
\end{remark}

\begin{remark}
Lemma \ref{lemma-3} and the result stated in Remark \ref{remark-4} prove the second assertion of Theorem \ref{theorem-graph}.
\end{remark}

\section{Proof of Theorem \ref{theorem-limit}}

Here we study standing wave solutions of the stationary NLS equation (\ref{statNLS}) in the limit $\Lambda \to -\infty$.
The standing wave solutions are represented asymptotically
by a solitary wave of the stationary NLS equation on the infinite line.
We use the scaling transformation
\begin{equation}
\label{scaling}
\Phi(x) = |\Lambda|^{\frac{1}{2}} \Psi(z), \quad z = |\Lambda|^{\frac{1}{2}} x,
\end{equation}
and consider the positive solutions of the stationary NLS equation (\ref{statNLS}) for $\Lambda < 0$.
The stationary problem can then be written in the equivalent form
\begin{equation}
\label{statNLS-limit-graph}
-\Delta_z \Psi + \Psi - 2 \Psi^3 = 0, \qquad z \in J_- \cup J_0 \cup J_+,
\end{equation}
where $\Delta_z$ is the Laplacian operator in variable $z$ and the
intervals on the real line are now given by
$$
J_- := \left[ - (L+2\pi) \mu, - L \mu \right], \quad
J_0 := \left[ - L \mu, L \mu \right], \quad
J_+ := \left[ L \mu,(L+2\pi) \mu \right],
$$
with $\mu := |\Lambda|^{\frac{1}{2}}$. The Kirchhoff boundary conditions (\ref{bc-1}) and (\ref{bc-2}) are to be
used at the two junction points.

The stationary NLS equation $-\Delta_z \Psi + \Psi - 2 \Psi^3 = 0$ on the infinite
line is satisfied by the solitary wave
\begin{equation}
\label{NLS-soliton}
\Psi_{\infty}(z) = {\rm sech}(z), \quad z \in \mathbb{R}.
\end{equation}
To yield a suitable approximation of the stationary equation (\ref{statNLS-limit-graph}) on the dumbbell graph
$J_- \cup J_0 \cup J_+$, we have to satisfy the Kirchhoff boundary conditions at the two junction points.
Two particular configurations involving a single solitary wave
will be considered below: one where the solitary wave is located in the central line segment and
the other one where the solitary wave is located in one of the two loops.
As follows from numerical results reported on Figures \ref{stat2}, \ref{stat3}, \ref{statcomp}, and \ref{stat8} below,
these configurations are continuations of the two families of
positive non-constant standing waves in Lemma \ref{lemma-3} and Remark \ref{remark-4}.

\subsection{Symmetric solitary wave}

We are looking for the symmetric standing wave
\begin{equation}
\label{symmetry-even}
\Psi_0(-z) = \Psi_0(z), \quad z \in J_0, \quad \Psi_-(-z) = \Psi_+(z), \quad z \in J_+.
\end{equation}
We will first provide an approximation of the solitary wave
with the required Kirchhoff boundary conditions by using
the limiting solitary wave (\ref{NLS-soliton}). Then, we will
develop analysis based on the fixed-point iterations
to control the correction terms to this approximation.
The following lemma summarizes the corresponding result.

\begin{lemma}
\label{lemma-soliton-line}
There exist $\mu_0 > 0$ sufficiently large and a positive $\mu$-independent constant $C$
such that the stationary NLS equation (\ref{statNLS-limit-graph})
for $\mu \in (\mu_0,\infty)$ admits a symmetric standing wave
$\Psi$ near $\Psi_{\infty}$ satisfying the estimate
\begin{equation}
\label{estimate-1}
\| \Psi - \Psi_{\infty} \|_{L^{\infty}(J_- \cup J_0 \cup J_+)} \leq C \mu^{3/2} e^{-L \mu}.
\end{equation}
\end{lemma}

\begin{proof}
The proof consists of two main steps.\\

{\bf Step 1: Approximation.} We denote the approximation of $\Psi$ in $J_0$ by $G_0$ and set it to
\begin{equation}
\label{approx-1}
G_0(z) = {\rm sech}(z), \quad z \in J_0.
\end{equation}
The approximation of $\Psi$ in $J_+$, denoted by $G_+$, cannot be defined from a solution
of the linear equation $G_+'' - G_+ = 0$ because the
second-order differential equation does not provide three parameters
to satisfy the three Kirchhoff boundary conditions:
\begin{equation}
\label{approx-2}
G_+(L \mu) = G_+((L+2\pi) \mu) = G_0(L \mu), \quad
G'_+(L \mu) - G_+'((L+2\pi) \mu) = G_0'(L \mu).
\end{equation}
Instead, we construct a polynomial approximation to the Kirchhoff boundary conditions,
which does not solve any differential equation. Using quadratic polynomials,
we satisfy the boundary conditions (\ref{approx-2}) for the solitary wave
(\ref{approx-1}) with the following approximation
\begin{equation}
\label{approx-3}
G_+(z) = {\rm sech}(L \mu) \left[ 1 + \frac{\tanh(L \mu)}{4 \pi \mu} (z - L \mu) (z - (L+2\pi) \mu) \right].
\end{equation}
Note that the maximum of $G_+$ occurs at the middle point of $J_+$ at $z = (L+\pi) \mu$,
and for sufficiently large $\mu$, we have
\begin{equation}
\label{approx-4}
\| G_+ \|_{L^{\infty}(J_+)} \leq C \mu e^{-L \mu}
\end{equation}
for a positive $\mu$-independent constant $C$. Also, the first and second derivatives of $G_+$ do not exceed the
upper bound in (\ref{approx-4}).\\

{\bf Step 2: Fixed-point arguments.}
Next, we consider the correction terms to the approximation $G$ in (\ref{approx-1}) and (\ref{approx-3}).
Using the decomposition $\Psi = G + \psi$, we obtain the persistence problem in the form
\begin{equation}
\label{persistence-problem-graph}
L_{\mu} \psi = {\rm Res}(G) + N(G,\psi),
\end{equation}
where $L_{\mu} := -\Delta_z + 1 - 6 G(z)^2$, ${\rm Res}(G) := \Delta_z G - G + 2 G^3$, and
$N(G,\psi) := 6 G \psi^2 + 2 \psi^3$.
Since the approximation $G$ satisfies Kirchhoff boundary conditions, which are linear and homogeneous,
the correction term $\psi$ is required to satisfy the same Kirchhoff boundary conditions.
The residual term ${\rm Res}(G)$
is supported in $J_-$ and $J_+$ and it satisfies the same estimate as in (\ref{approx-4}).
Transferring this estimate to the $L^2$ norm, since the length of $J_+$ grows linearly in $\mu$,
we have for all sufficiently large $\mu$,
\begin{equation}
\label{approx-5}
\| {\rm Res}(G) \|_{L^2(J_- \cup J_0 \cup J_+)} \leq C \mu^{3/2} e^{-L \mu},
\end{equation}
where the positive constant $C$ is $\mu$-independent.

The operator $L_{\mu}$ is defined on $L^2(J_- \cup J_0 \cup J_+)$ with the domain in $\mathcal{D}(\Delta_z)$,
that incorporates homogeneous Kirchhoff boundary conditions. As $\mu \to \infty$, the operator $L_{\mu}$
converges pointwise to the operator
$$
L_{\infty} := -\frac{d^2}{dz^2} + 1 - 6 {\rm sech}^2(z) : \; H^2(\R) \to L^2(\R),
$$
which has a one-dimensional kernel spanned
by $\Psi'_{\infty}(z)$, whereas the rest of the spectrum of $L_{\infty}$ includes an isolated eigenvalue
at $-3$ and the continuous spectrum for $[1,\infty)$. Since the operator $L_{\infty}$ is invertible in the space of
even functions, it follows that the operator $L_{\mu} : \mathcal{D}(\Delta_z) \to L^2(J_- \cup J_0 \cup J_+)$,
is also invertible with a bounded inverse on the space of even functions if $\mu$ is sufficiently large.
In other words, there is a positive $\mu$-independent constant $C$ such that for every even $f \in L^2(J_- \cup J_0 \cup J_+)$
and sufficiently large $\mu$, the even function $L_{\mu}^{-1} f$ satisfies the estimate
\begin{equation}
\label{approx-6}
\| L_{\mu}^{-1} f \|_{H^2(J_- \cup J_0 \cup J_+)} \leq C \| f \|_{L^2(J_- \cup J_0 \cup J_+)}.
\end{equation}
Hence we can analyze the fixed-point problem
\begin{equation}
\label{fixed-point-graph}
\psi = L_{\mu}^{-1} \left[ {\rm Res}(G) + N(G,\psi) \right], \quad \psi \in \mathcal{D}(\Delta_z)
\end{equation}
with the contraction mapping method. By using (\ref{approx-5}), (\ref{approx-6}), and the Banach
algebra properties of $H^2(J_- \cup J_0 \cup J_+)$ in the estimates of the nonlinear term $N(G,u)$,
we deduce the existence of a small unique solution $\psi \in \mathcal{D}(\Delta_z)$
of the fixed-point problem (\ref{fixed-point-graph}) satisfying the estimate
\begin{equation}
\label{estimate-fixed-point}
\| \psi \|_{H^2(J_- \cup J_0 \cup J_+)} \leq C \mu^{3/2} e^{-L \mu},
\end{equation}
for sufficiently large $\mu$ and a positive $\mu$-independent constant $C$.
By the construction above, there exists a solution $\Psi = G + \psi$ of
the stationary NLS equation (\ref{statNLS-limit-graph}) that is close
to the solitary wave (\ref{NLS-soliton}) placed symmetrically in the central line segment.
The estimate (\ref{estimate-1}) is obtained from (\ref{estimate-fixed-point}) by Sobolev's
embedding of $H^2(J_- \cup J_0 \cup J_+)$ to $L^{\infty}(J_- \cup J_0 \cup J_+)$.
\end{proof}

\begin{remark}
The method in the proof of Lemma \ref{lemma-soliton-line}
cannot be used to argue that $\Psi$ is positive, although
positivity of $\Psi$ is strongly expected. In particular, $G_+$ is not positive on $J_+$
and the correction term $\psi_+$ is comparable with $G_+$ in $J_+$.
Similarly, the approximation $G$ of the solution $\Psi$ is not unique, although
for every $G$, there exists a unique correction $\psi$ by the contraction mapping
method used in analysis of the fixed-point problem (\ref{fixed-point-graph}). We will obtain a better result
in Lemma \ref{remark-soliton-line} with a more sophisticated analytical technique in order to
remove these limitations of Lemma \ref{lemma-soliton-line}.
\end{remark}

\subsection{Solitary wave in the ring}

We are now looking for an approximation of the solution $\Psi$ of the stationary
NLS equation (\ref{statNLS-limit-graph}), which represents as $\mu \to \infty$
a solitary wave residing in one of the rings, e.g. in $J_+$.
Because of the Kirchhoff boundary conditions in $J_+$, the approximation of $\Psi$ in $J_+$,
denoted by $G_+$, must be symmetric with respect to the middle point in $J_+$. Therefore,
we could take
\begin{equation}
\label{approx-1-ring}
G_+(z) = {\rm sech}(z - L \mu - \pi  \mu), \quad z \in J_+.
\end{equation}
However, the method used in the proof of Lemma \ref{lemma-soliton-line} fails
to continue the approximation (\ref{approx-1-ring}) with respect to finite values of
parameter $\mu$. Indeed, the linearization operator $L_{\mu} : \mathcal{D}(\Delta_z) \to L^2(J_- \cup J_0 \cup J_+)$
defined on the approximation $G_+$ has zero eigenvalue in the limit $\mu \to \infty$,
which becomes an exponentially small eigenvalue for large values of $\mu$.
Since no spatial symmetry can be used for the correction term $\psi_+$ to $G_+$ on $J_+$
because of the Kirchhoff boundary conditions in $J_+$,
it becomes very hard to control the projection of $\Psi_+$ to the
subspace of $\mathcal{D}(\Delta_z)$ related to the smallest eigenvalue of $L_{\mu}$.

To avoid the aforementioned difficulty and to prove persistence of the approximation (\ref{approx-1-ring}),
we develop here an alternative analytical technique. We solve the existence problem on $J_+$
in terms of Jacobi elliptic functions with an unknown parameter and then transform
the existence problem on $J_- \cup J_0$ with the unknown parameter to the fixed-point
problem. After a unique solution is obtained, we define a unique value
of the parameter used in the Jacobi elliptic functions.
The following lemma summarizes the corresponding result.

\begin{lemma}
\label{lemma-soliton-ring}
There exist $\mu_0 > 0$ sufficiently large and a positive $\mu$-independent constant $C$
such that the stationary NLS equation (\ref{statNLS-limit-graph})
for $\mu \in (\mu_0,\infty)$ admits a unique positive asymmetric standing wave
$\Psi$ given by
\begin{equation}
\label{dnoidal}
\Psi_+(z) = \frac{1}{\sqrt{2-k^2}} {\rm dn}\left(\frac{z- L \mu - \pi \mu}{\sqrt{2-k^2}};k\right), \quad z \in J_+,
\end{equation}
and satisfying the estimate
\begin{equation}
\label{dnoidal-estimates}
\| \Psi_0 \|_{H^2(J_0)} + \|\Psi_- \|_{H^2(J_-)} \leq C e^{-\pi \mu},
\end{equation}
where ${\rm dn}(\xi;k)$ is the Jacobi elliptic function defined for the elliptic modulus parameter
$k \in (0,1)$. The unique value of $k$ satisfies the asymptotic expansion
\begin{equation}
\label{dnoidal-expansion}
\sqrt{1 - k^2} = \frac{4}{\sqrt{3}} e^{-\pi \mu} \left[ 1 + \mathcal{O}(\mu e^{-2 \pi \mu},e^{- 4 L \mu}) \right]
\quad \mbox{\rm as} \quad \mu \to \infty.
\end{equation}
\end{lemma}

\begin{remark}
It is well-known (see, e.g., Lemma 2 in \cite{Pelin} for similar estimates) that
if $\mu \to \infty$ and $k \to 1$ according to the asymptotic expansion (\ref{dnoidal-expansion}),
then the dnoidal wave (\ref{dnoidal}) is approximated
by the solitary wave (\ref{approx-1-ring}) on $J_+$ such that
\begin{equation}
\label{estimate-2}
\| \Psi_+ - G_+ \|_{L^{\infty}(J_+)} \leq C e^{-\pi \mu},
\end{equation}
where the positive constant $C$ is $\mu$-independent. Therefore, for sufficiently large $\mu$,
we have justified the bound
$$
\| \Psi - \Psi_{\infty}(\cdot - L \mu - \pi \mu) \|_{L^{\infty}(G_- \cup G_0 \cup G_+)} \leq C e^{-\pi \mu},
$$
for the asymmetric standing wave of Lemma \ref{lemma-soliton-ring}.
\end{remark}

\begin{proof}
The proof of Lemma \ref{lemma-soliton-ring} consists of three main steps.\\

{\bf Step 1: Dnoidal wave solution.}
The second-order differential equation (\ref{statNLS-limit-graph}) is integrable
and all solutions can be studied on the phase plane $(\Psi,\Psi')$. The trajectories on the phase plane
correspond to the level set of the first-order invariant
\begin{equation}
\label{invariant}
I := \left( \frac{d \Psi}{d z} \right)^2 - \Psi^2 + \Psi^4 = {\rm const}.
\end{equation}
The level set of $I$ is shown on Figure \ref{fig3}.
There are two families of periodic solutions. One family is sign-indefinite and the corresponding
trajectories on the phase plane $(\Psi,\Psi')$ surround
the three equilibrium points. This family is expressed in terms
of the Jacobi cnoidal function. The other family of periodic solutions is
strictly positive and the corresponding trajectories on the phase plane $(\Psi,\Psi')$
are located inside the positive homoclinic orbit. This other family
is expressed in terms of the Jacobi dnoidal function.

\begin{figure}[htp]
\includegraphics[width=12cm]{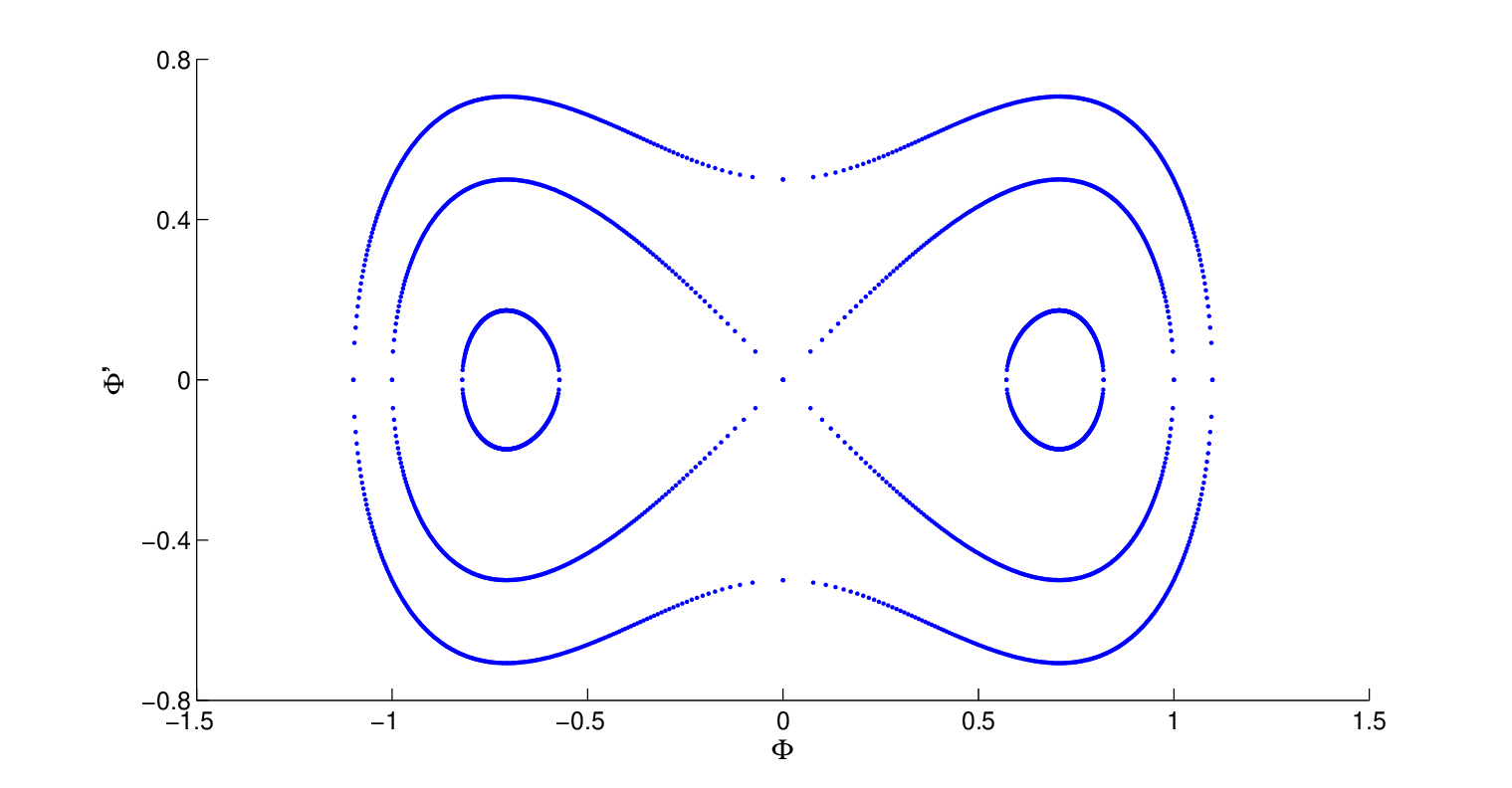}
\caption{Level set (\ref{invariant}) on the phase plane $(\Psi,\Psi')$.}
\label{fig3}
\end{figure}

Because of the Kirchhoff boundary condition $\Psi_+(L \mu) = \Psi_+(L\mu + 2 \pi \mu)$,
we consider a trajectory on the phase plane $(\Psi,\Psi')$, which is symmetric
about the middle point in $J_+$ at $z = (L+\pi) \mu$.   Since the trajectory
is supposed to converge to $G_+$ given by (\ref{approx-1-ring}) as $\mu \to \infty$,
we select an incomplete orbit with $\Psi_+'(L \mu) = -\Psi_+'(L\mu + 2 \pi \mu)$. For the trajectory
inside the homoclinic orbit, the corresponding solution of the first-order invariant (\ref{invariant})
is given by the exact expression (\ref{dnoidal}). We now define
\begin{eqnarray}
\label{def-p}
p(k,\mu) & := & \Psi_+(L \mu) =  \frac{1}{\sqrt{2-k^2}} {\rm dn}\left(\frac{\pi \mu}{\sqrt{2-k^2}};k\right), \\
\label{def-q}
q(k,\mu) & := & \Psi_+'(L \mu) = \frac{k^2}{2-k^2} {\rm sn}\left(\frac{\pi \mu}{\sqrt{2-k^2}};k\right) {\rm cn}\left(\frac{\pi \mu}{\sqrt{2-k^2}};k\right)
\end{eqnarray}
and consider the range of the values of $k$ for which $q(k,\mu) \geq 0$.
Since ${\rm cn}(\xi;k)$ vanishes at $\xi = K(k)$, where $K(k)$ is the complete elliptic integrals of the first kind,
we can define $k_*(\mu)$ from the unique root of the equation
\begin{equation}
\label{root-k-star}
\pi \mu = \sqrt{2 - k_*^2} K(k_*).
\end{equation}
Therefore, $q(k_*(\mu),\mu) = 0$. On the other hand, we have
\begin{equation}
\label{limiting-1}
q(k,\mu) \to q_*(\mu) := \tanh(\pi \mu) {\rm sech}(\pi \mu) = 2 e^{-\pi \mu} + \mathcal{O}(e^{-3\pi \mu}) \quad
\mbox{\rm as} \quad k \to 1.
\end{equation}

Since (see 8.113 in \cite{Grad})
\begin{equation}
\label{expansion-K}
K(k) = \log\left(\frac{4}{\sqrt{1-k^2}}\right) + \mathcal{O}\left((1-k^2)|\log(1-k^2)|\right)\quad
\mbox{\rm as} \quad k \to 1,
\end{equation}
the root $k_*(\mu)$ of the transcendental equation (\ref{root-k-star}) satisfies the asymptotic expansion
$$
\sqrt{1 - k_*^2} = 4 e^{-\pi \mu} + \mathcal{O}(e^{-3 \pi \mu}) \quad \mbox{\rm as} \quad \mu \to \infty.
$$
Thus, the entire interval $(k_*(\mu),1)$ is exponentially small in terms of large $\mu$.
With this asymptotic in mind, we compute the limiting values of the function $p(k,\mu)$ for $k \in (k_*(\mu),1)$.
At one end, we obtain
\begin{equation}
\label{limiting-2}
p(k_*(\mu),\mu) = \frac{\sqrt{1 - k_*^2}}{\sqrt{2 - k_*^2}} = 4 e^{-\pi \mu} +  \mathcal{O}(e^{-3 \pi \mu}) \quad \mbox{\rm as} \quad \mu \to \infty,
\end{equation}
whereas at the other end, we obtain
\begin{equation}
\label{limiting-3}
p(k,\mu) \to p_*(\mu) := {\rm sech}(\pi \mu) = 2 e^{-\pi \mu} +  \mathcal{O}(e^{-3 \pi \mu}) \quad \mbox{\rm as} \quad k \to 1.
\end{equation}

Figure \ref{fig5} shows the dependencies of $p$ and $q$ versus $k$ in $(k_*(\mu),1)$ for a particular value $\mu = 2$.
The graph illustrates that $q(k,\mu)$ is a monotonically increasing function with respect to $k$ in $k \in (k_*(\mu),1)$ from
$0$ to $q_*(\mu)$, whereas $p(k,\mu)$ is a monotonically decreasing function from $p(k_*(\mu),\mu)$ to $p_*(\mu)$.\\

\begin{figure}[htp]
\includegraphics[width=12cm]{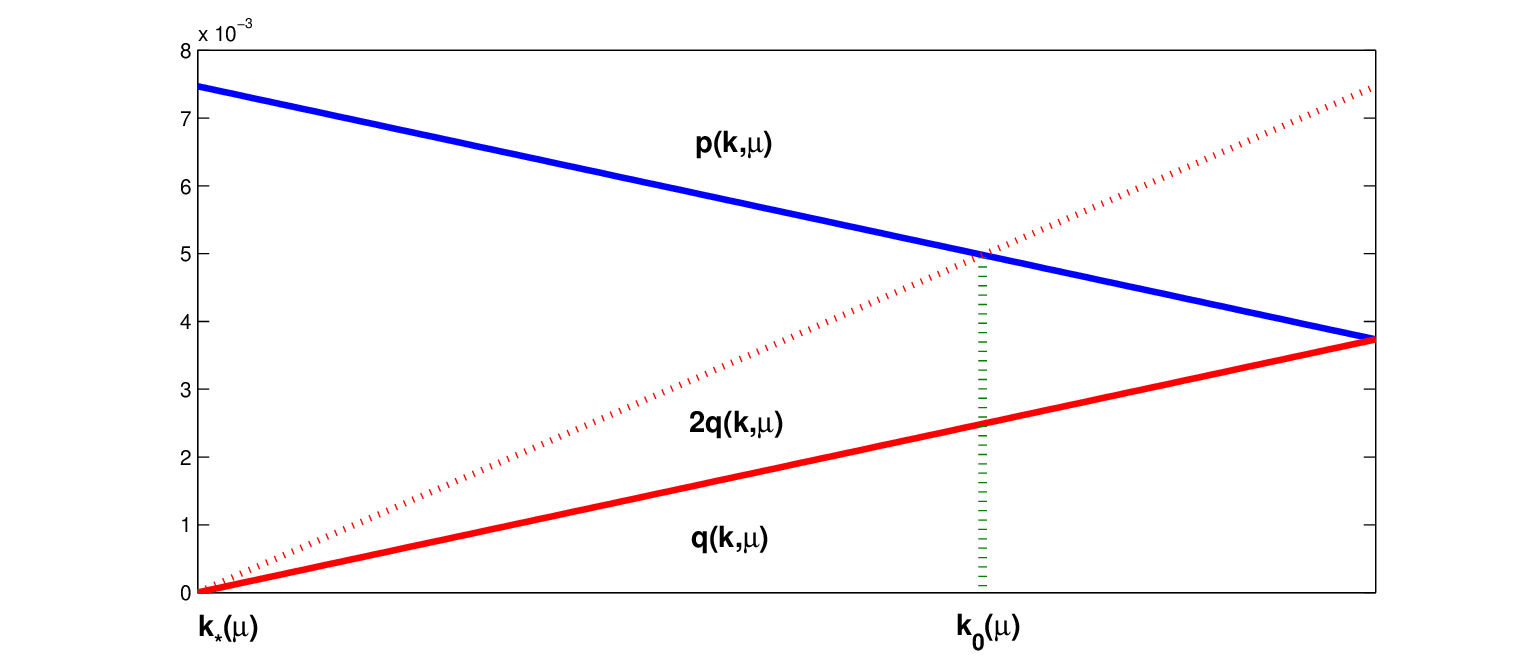}
\caption{The graphs of $p$ (solid blue), $q$ (solid red), and $2 q$ (dashed red) versus $k$ in $(k_*(\mu),1)$
for $\mu = 2$. The dotted line shows a graphical solution of the equation $p = 2q$ at $k_0(\mu)$.}
\label{fig5}
\end{figure}

{\bf Step 2: Fixed-point arguments.} By substituting the explicit solution (\ref{dnoidal}) to the stationary
NLS equation (\ref{statNLS-limit-graph}), we close the system of two second-order differential equations
\begin{eqnarray}
\left\{ \begin{array}{l}
-\Psi_0'' + \Psi_0 - 2 \Psi_0^2 = 0, \quad z \in J_0, \\
-\Psi_-'' + \Psi_- - 2 \Psi_-^2 = 0, \quad z \in J_-, \end{array} \right. \label{system-ode}
\end{eqnarray}
with five boundary conditions
\begin{eqnarray}
\label{system-ode-bc}
\left\{ \begin{array}{l}
\Psi_-(-L\mu) = \Psi_-(-L\mu - 2 \pi \mu) = \Psi_0(-L\mu), \\
\Psi_-'(-L\mu) - \Psi_-'(-L\mu - 2 \pi \mu) = \Psi_0'(-L \mu), \\
\Psi_0 (L \mu) = p(k,\mu), \quad \Psi_0' (L \mu) = 2 q(k,\mu),\end{array} \right.
\end{eqnarray}
where $p(k,\mu)$ and $q(k,\mu)$ are explicit functions of an unknown parameter $k$ defined in
the interval $(k_*(\mu),1)$.
Thus, although the five boundary conditions over-determine the system of differential equations (\ref{system-ode}),
the parameter $k$ can be used to complete the set of unknowns.
For this step, we will neglect the last boundary condition in (\ref{system-ode-bc})
given by $\Psi_0' (L \mu) = 2 q(k,\mu)$.

Let us define the approximation of the solution to the system (\ref{system-ode}),
denoted by $(G_-,G_0)$ by the solution of the linear system
\begin{eqnarray}
\left\{ \begin{array}{l}
- G_0'' + G_0 = 0, \quad z \in J_0, \\
-G_-'' + G_- = 0, \quad z \in J_-, \end{array} \right. \label{system-ode-linear}
\end{eqnarray}
subject to the four boundary conditions
\begin{eqnarray}
\left\{ \begin{array}{l} G_-(-L\mu) = G_-(-L\mu - 2 \pi \mu) = G_0(-L\mu), \\
G_-(-L\mu) - G_-'(-L\mu - 2 \pi \mu) = G_0'(-L \mu), \\
G_0(L \mu) = p(k,\mu). \end{array} \right.
\end{eqnarray}
The boundary-value problem has the unique solution
\begin{align}
\label{approximation-G-minus}
G_-(z) & =  p(k,\mu) \frac{\cosh(z + L \mu + \pi \mu)}{\cosh(\pi \mu) \cosh(2L \mu) + 2 \sinh(\pi \mu) \sinh(2 L \mu)}, \\
\label{approximation-G-0}
G_0(z) & =  p(k,\mu) \frac{\cosh(\pi \mu) \cosh(z + L \mu) + 2 \sinh(\pi \mu) \sinh(z + L \mu)}{
\cosh(\pi \mu) \cosh(2L \mu) + 2 \sinh(\pi \mu) \sinh(2 L \mu)}.
\end{align}
It follows from the explicit solution (\ref{approximation-G-minus}) and (\ref{approximation-G-0}) that
for $\mu > 0$ sufficiently large, we have
\begin{eqnarray}
\label{estimate-G}
\| G_- \|_{H^2(J_-)} \leq C p(k,\mu) e^{-2L \mu}, \quad \| G_0 \|_{H^2(J_0)} \leq C p(k,\mu),
\end{eqnarray}
where $p(k,\mu) = \mathcal{O}(e^{-\pi \mu})$ as $\mu \to \infty$ for every $k \in (k_*(\mu),1)$
and $C$ is a positive $\mu$-independent constant.

\begin{remark}
Compared with the estimates (\ref{approx-4}) and (\ref{approx-5}), where we are losing $\mu^{1/2}$
between the $L^{\infty}$ and $L^2$ bounds, the estimate (\ref{estimate-G}) works equally well
in $L^{\infty}$ and $L^2$ thanks to the integration of the explicit solutions (\ref{approximation-G-minus})
and (\ref{approximation-G-0}).
\end{remark}

Applying the decomposition
$$
\Psi_-(z) = G_-(z) + \psi_-(z), \quad \Psi_0(z) = G_0(z) + \psi_0(z),
$$
we obtain the following persistence problem
\begin{eqnarray}
\left\{ \begin{array}{l}
- \psi_0'' + \psi_0 = 2 (G_0 + \psi_0)^3, \quad z \in J_0, \\
- \psi''_- + \psi_- = 2 (G_- + \psi_-)^3, \quad z \in J_-, \end{array} \right. \label{system-ode-persistence}
\end{eqnarray}
subject to the four homogeneous boundary conditions
\begin{eqnarray}
\left\{ \begin{array}{l}
\psi_-(-L\mu) = \psi_-(-L\mu - 2 \pi \mu) = \psi_0(-L\mu), \\
\psi_-'(-L\mu) - \psi_-'(-L\mu - 2 \pi \mu) = \psi_0'(-L \mu), \\
\psi_0(L \mu) = 0.\end{array} \right. \label{homo-boun-cond}
\end{eqnarray}
Since $1 - \Delta_z$ is invertible on $L^2(J_- \cup J_0)$ with a bounded inverse in $H^2(J_- \cup J_0)$
due to the four symmetric homogeneous boundary conditions (\ref{homo-boun-cond})
and the inhomogeneous term is estimated by using the bound (\ref{estimate-G}),
a contraction mapping method applies to the fixed-point problem (\ref{system-ode-persistence}).
As a result, there exists a small unique solution for $\psi_-$ and $\psi_0$
of the persistence problem  (\ref{system-ode-persistence}) satisfying the estimate
\begin{eqnarray}
\label{estimate-psi}
\| \psi_- \|_{H^2(J_-)} + \| \psi_0 \|_{H^2(J_0)} \leq C p(k,\mu)^3,
\end{eqnarray}
where $p(k,\mu) = \mathcal{O}(e^{-\pi \mu})$ as $\mu \to \infty$ for every $k \in (k_*(\mu),1)$
and the positive constant $C$ is $\mu$-independent.\\

{\bf Step 3: Unique value for the parameter $k$ in the interval $(k_*(\mu),1)$.}
It remains to satisfy the fifth boundary condition in (\ref{system-ode-bc}), which can be written
in the following form
\begin{equation}
\label{root-finding}
2 q(k,\mu) = G_0'(L \mu) + \psi_0'(L \mu) = p(k,\mu) +  \mathcal{O}(e^{-\pi \mu - 4 L \mu},e^{-3 \pi \mu}) \quad \mbox{\rm as} \quad \mu \to \infty,
\end{equation}
where we have used the exact result (\ref{approximation-G-0}) yielding
$$
G_0'(L \mu) = p(k,\mu) \left( 1 + e^{-2 L \mu}
\frac{2 \sinh(\pi \mu) - \cosh(\pi \mu)}{\cosh(\pi \mu) \cosh(2L \mu) + 2 \sinh(\pi \mu) \sinh(2 L \mu)} \right),
$$
as well as the estimate (\ref{estimate-psi})
with the account that $p(k,\mu) = \mathcal{O}(e^{-\pi \mu})$ as $\mu \to \infty$ for every $k \in (k_*(\mu),1)$.

Figure \ref{fig5} illustrates graphically that $q(k,\mu)$ is monotonically increasing with respect to $k$ in the interval $(k_*(\mu),1)$
from $0$ to $q_*(\mu) = 2 e^{-\pi \mu} +  \mathcal{O}(e^{-3 \pi \mu})$,
whereas $p(k,\mu)$ is monotonically decreasing with respect to $k$ in the interval $(k_*(\mu),1)$
from $p(k_*(\mu),\mu) = 4 e^{-\pi \mu} +  \mathcal{O}(e^{-3 \pi \mu})$
to $p_*(\mu) = 2 e^{-\pi \mu} +  \mathcal{O}(e^{-3 \pi \mu})$. Therefore,
there exists exactly one solution $k_0(\mu) \in (k_*(\mu),1)$ of the equation (\ref{root-finding}).

In order to prove monotonicity of $p$ and $q$ in $k$, as well as
the asymptotic expansion (\ref{dnoidal-expansion}), we study the rate of change
of the functions $p$ and $q$ with respect to $k$. From the explicit expression (\ref{def-p}),
we obtain
\begin{align}
\label{def-der-p}
\partial_k p (k, \mu) & =  \frac{k}{\sqrt{(2-k^2)^3}}  {\rm dn}\left(\frac{\pi \mu}{\sqrt{2-k^2}};k\right)  \\
\nonumber
& \phantom{t}  -  \frac{ \pi \mu k^3}{ (2-k^2)^2}   {\rm sn}\left(\frac{\pi \mu}{\sqrt{2-k^2}};k\right) {\rm cn}\left(\frac{\pi \mu}{\sqrt{2-k^2}};k\right) \\
& \phantom{t}  + \frac{1}{\sqrt{2-k^2}} \partial_k {\rm dn}\left(\xi;k\right)  \biggr|_{\xi = \frac{\pi \mu}{\sqrt{2-k^2}}}.
\nonumber
\end{align}
For every $k \in (k_*(\mu),1)$ and sufficiently large $\mu$, the first term in (\ref{def-der-p}) is exponentially small of
the order of $\mathcal{O}(e^{-\pi \mu})$, whereas the second term is larger of the order of $\mathcal{O}(\mu e^{-\pi \mu})$.
Nevertheless, we show that the last term in (\ref{def-der-p}) is dominant as it is exponentially large as $\mu \to \infty$.
To show this, we use the following result proved in Appendix A.

\begin{proposition}
\label{prop-elliptic}
For every $\xi \in \mathbb{R}$, it is true that
\begin{align}
\label{sn-der}
& {\rm sn}(\xi;1) =  \tanh(\xi), \quad \partial_k {\rm sn}(\xi;1) = -\frac{1}{2} \left[ \sinh(\xi) \cosh(\xi) - \xi \right] {\rm sech}^2(\xi), \\
\label{cn-der}
& {\rm cn}(\xi;1)  =  {\rm sech}(\xi), \quad \partial_k {\rm cn}(\xi;1) = \frac{1}{2} \left[ \sinh(\xi) \cosh(\xi) - \xi \right] \tanh(\xi) {\rm sech}(\xi), \\
\label{dn-der}
& {\rm dn}(\xi;1)  =  {\rm sech}(\xi), \quad \partial_k {\rm dn}(\xi;1) = -\frac{1}{2} \left[ \sinh(\xi) \cosh(\xi) + \xi \right] \tanh(\xi) {\rm sech}(\xi).
\end{align}
Moveover, if $\mu$ is sufficiently large, then for every $\xi \in (0,\pi \mu)$ and every $k \in (k_*(\mu),1)$,
there is a positive $\mu$-independent constant $C$ such that
\begin{equation}
\label{bound-second-derivative}
|\partial_k {\rm sn}(\xi;k) - \partial_k {\rm sn}(\xi;1)| +
| \partial_k {\rm cn}(\xi;k) - \partial_k {\rm cn}(\xi;1)| +
|\partial_k {\rm dn}(\xi;k) - \partial_k {\rm dn}(\xi;1)| \leq C \mu e^{-\pi \mu}.
\end{equation}
\end{proposition}

From (\ref{dn-der}) and (\ref{bound-second-derivative}),
we obtain the dominant contribution of (\ref{def-der-p}) for every $k \in (k_*(\mu),1)$:
\begin{equation}
\label{der-p-dominant}
\partial_k p (k, \mu)  = -\frac{1}{4} e^{\pi \mu} + \mathcal{O}(\mu e^{-\pi \mu}) \quad \mbox{\rm as} \quad \mu \to \infty.
\end{equation}
Similarly, we differentiate (\ref{def-q}) in $k$, use (\ref{sn-der}), (\ref{cn-der}), and (\ref{bound-second-derivative}),
and obtain the asymptotic expansion for every $k \in (k_*(\mu),1)$:
\begin{equation}
\label{der-q-dominant}
\partial_k q (k, \mu) = \frac{1}{4} e^{\pi \mu} + \mathcal{O}(\mu e^{-\pi \mu}) \quad \mbox{\rm as} \quad \mu \to \infty.
\end{equation}
It follows from (\ref{der-p-dominant}) and (\ref{der-q-dominant}) that $p(k,\mu)$ and $q(k,\mu)$ are monotonically
decreasing and increasing functions with respect to $k$ as $k \to 1$, in agreement with the behavior on
Figure \ref{fig5}. Furthermore, the algebraic equation (\ref{root-finding}) can be analyzed in
the asymptotic limit of large $\mu$. Indeed, multiplying (\ref{root-finding}) by $e^{-\pi \mu}$, we obtain
\begin{align}
\nonumber
& \phantom{t}  4 e^{-2\pi \mu} + \frac{1}{2} (k-1) + \mathcal{O}(e^{-4 \pi \mu},(k-1) \mu e^{-2\pi \mu})\\
& =
2 e^{-2\pi \mu} - \frac{1}{4} (k-1) + \mathcal{O}(e^{-4 \pi \mu},(k-1) \mu e^{-2\pi \mu},e^{-2\pi \mu - 4 L \mu}),
\label{root-finding-asymptotics}
\end{align}
where remainder terms are all smooth in their variables.
By the Implicit Function Theorem, we obtain the unique root of the algebraic equation (\ref{root-finding-asymptotics})
denoted by $k_0(\mu)$. The root satisfies the asymptotic expansion
$$
k_0(\mu) = 1 - \frac{8}{3} e^{-2\pi \mu} \left[ 1 + \mathcal{O}(\mu e^{-2 \pi \mu},e^{- 4 L \mu}) \right],
$$
which justifies the asymptotic expansion (\ref{dnoidal-expansion}). Furthermore, the
bound (\ref{dnoidal-estimates}) follows from estimates (\ref{estimate-G}) and (\ref{estimate-psi}).

Finally, because the perturbation term $(\psi_-,\psi_0)$ is triply exponentially small,
whereas the leading-order approximation $(G_-,G_0)$ is exponentially small and positive,
we deduce that $(\Psi_-,\Psi_0)$ is positive on $J_- \cup J_0$.
From the exact representation (\ref{dnoidal}), we also know that $\Psi_+$ is positive on $J_+$. Thus, the asymmetric
standing wave is positive on $J_- \cup J_0 \cup J_+$. The proof of the lemma is complete.
\end{proof}

The same method in the proof of Lemma \ref{lemma-soliton-ring}
can be applied to construct the symmetric solitary wave described in Lemma \ref{lemma-soliton-line}.
However, because of the Kirchhoff boundary conditions, we need to take the symmetric orbit
outside of the homoclinic orbit on Figure \ref{fig3}.
The following lemma summarizes the corresponding result.

\begin{lemma}
\label{remark-soliton-line}
There exist $\mu_0 > 0$ sufficiently large and a positive $\mu$-independent constant $C$
such that the stationary NLS equation (\ref{statNLS-limit-graph})
for $\mu \in (\mu_0,\infty)$ admits a unique positive symmetric standing wave
$\Psi$ given by
\begin{equation}
\label{cnoidal}
\Psi_0(z) = \frac{k}{\sqrt{2k^2-1}} {\rm cn}\left(\frac{z}{\sqrt{2k^2-1}};k\right), \quad z \in J_0
\end{equation}
and satisfying the estimate
\begin{equation}
\label{cnoidal-estimates}
\| \Psi_+ \|_{H^2(J_+)} + \|\Psi_- \|_{H^2(J_-)} \leq C e^{-L \mu},
\end{equation}
where ${\rm cn}(\xi;k)$ is the Jacobi elliptic function defined for the elliptic modulus parameter
$k \in (0,1)$. The unique value for $k$ satisfies the asymptotic expansion
\begin{equation}
\label{cnoidal-expansion}
\sqrt{1 - k^2} = \frac{4}{\sqrt{3}} e^{-L \mu} \left[ 1 + \mathcal{O}(\mu e^{-2 L \mu},e^{- 4 \pi \mu}) \right]
\quad \mbox{\rm as} \quad \mu \to \infty.
\end{equation}
\end{lemma}

\begin{proof}
We only outline the minor differences in the computations compared to the proof given in Lemma \ref{lemma-soliton-ring}.
For the trajectory outside the homoclinic orbit, the corresponding solution of the first-order invariant (\ref{invariant})
is given by the exact expression (\ref{cnoidal}). We now define
\begin{align}
p(k,\mu) & :=  \Psi_0(L \mu) =  \frac{k}{\sqrt{2k^2-1}} {\rm cn}\left(\frac{L \mu}{\sqrt{2k^2 -1}};k\right), \\
q(k,\mu) & :=  -\Psi_0'(L \mu) =
 \frac{k}{2k^2-1} {\rm sn}\left(\frac{L \mu}{\sqrt{2k^2-1}};k\right) {\rm dn}\left(\frac{L \mu}{\sqrt{2k^2-1}};k\right).
\end{align}
The trajectory is already even in $z$. We consider the range of the values of $k$ for which $p(k,\mu) \geq 0$.
Therefore, $k$ is defined in $(k_*(\mu),1)$, where $k_*(\mu)$ is the root of the algebraic equation
$$
\mu L = \sqrt{2k_*^2-1} K(k_*),
$$
which is expanded asymptotically as
$$
\sqrt{1 - k_*^2} = 4 e^{-L \mu} + \mathcal{O}(e^{-3 L \mu}) \quad \mbox{\rm as} \quad \mu \to \infty.
$$
Again, the interval $(k_*(\mu),1)$ is exponentially small as $\mu \to \infty$.

Figure \ref{fig5b} shows the dependencies of $p$ and $q$ versus $k$ in $(k_*(\mu),1)$ for a particular value $\mu = 2$.
The graph illustrates that $p(k,\mu)$ is a monotonically increasing function with respect to $k$ from $0$ at $k = k_*(\mu)$
to
$$
p_*(\mu) := {\rm sech}(L \mu) = 2 e^{-L \mu} +  \mathcal{O}(e^{-3 L \mu})
$$
as $k \to 1$, whereas $q(k,\mu)$ is a monotonically decreasing function in $k$ from
$$
q(k_*(\mu),\mu) = \frac{k_* \sqrt{1 - k_*^2}}{2 k_*^2-1} = 4 e^{-L \mu} +  \mathcal{O}(e^{-3 L \mu})
$$
at $k = k_*(\mu)$ to
$$
q_*(\mu) := {\rm tanh}(L \mu) {\rm sech}(L \mu) = 2 e^{-L \mu} + \mathcal{O}(e^{-3 L \mu})
$$
as $k \to 1$. Again, there is a unique root $k = k_0(\mu)$ in $(k_*(\mu),1)$ of the algebraic equation
$$
q(k,\mu) = 2 p(k,\mu) + \mathcal{O}(e^{-3 L \mu},e^{-L \mu - 4 \pi \mu}) \quad \mbox{\rm as} \quad \mu \to \infty.
$$
The bound (\ref{cnoidal-estimates}), the asymptotic expansion (\ref{cnoidal-expansion}),
and the positivity of the symmetric wave are proved by similar estimates
to those in Lemma \ref{lemma-soliton-ring}.
\end{proof}

\begin{figure}[htp]
\includegraphics[width=12cm]{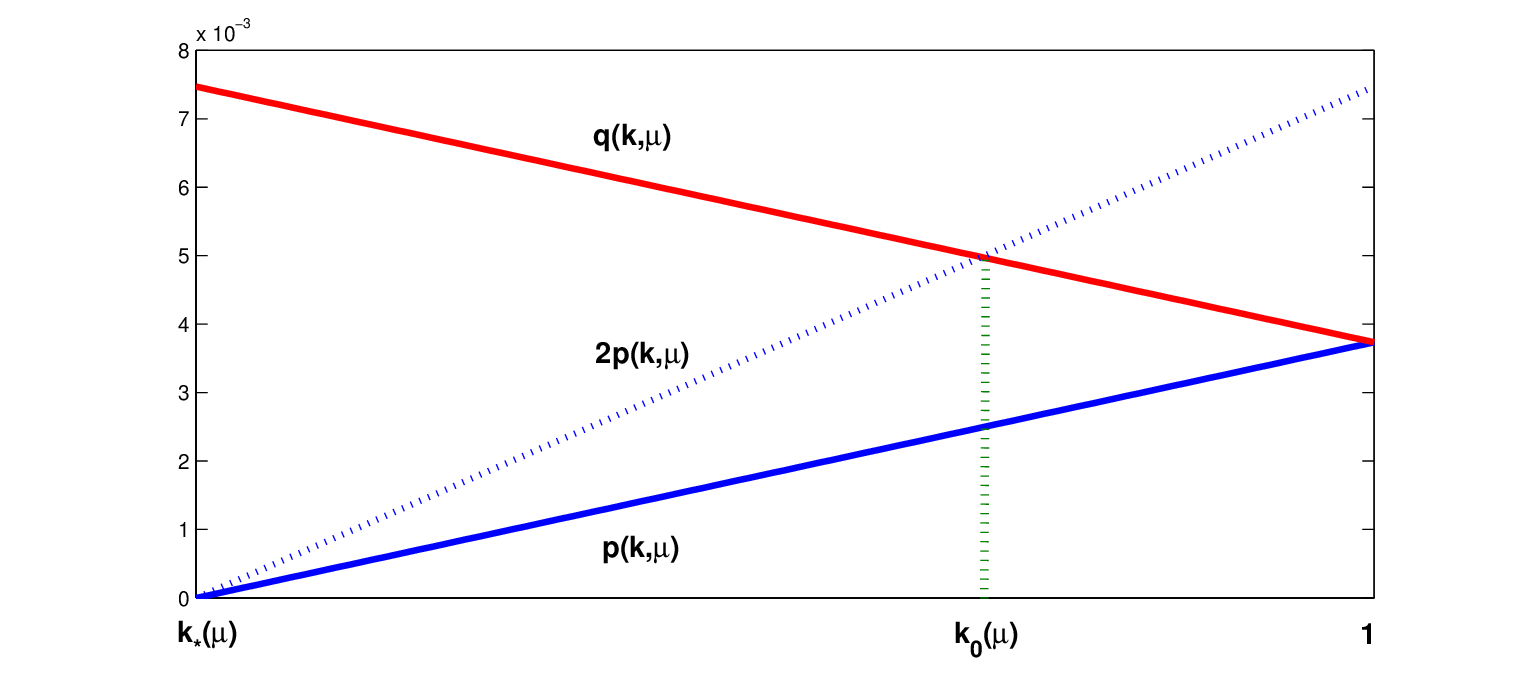}
\caption{The graphs of $p$ (solid blue), $q$ (solid red), and $2 p$ (dashed blue) versus $k$ in $(k_*(\mu),1)$
for $\mu = 2$ and $L = \pi$. The dotted line shows a graphical solution of the equation $q = 2p$ at $k_0(\mu)$.}
\label{fig5b}
\end{figure}

\begin{remark}
Lemma \ref{remark-soliton-line} gives an improvement of Lemma \ref{lemma-soliton-line}, since the symmetric standing
wave $\Psi$ is now proved to be strictly positive and the estimate (\ref{estimate-1}) is now improved to be
\begin{equation}
\label{estimate-3}
\| \Psi - \Psi_{\infty} \|_{L^{\infty}(J_- \cup J_0 \cup J_+)} \leq C e^{-L \mu},
\end{equation}
where the positive constant $C$ is $\mu$-independent.
\end{remark}

\subsection{Energy levels for the symmetric and asymmetric waves}

There exists a simple argument why the ground state of the constrained minimization problem (\ref{minimizer})
is represented by a single solitary wave in the asymptotic limit $\Lambda \to -\infty$.
Indeed, computing asymptotically $E_0$ and $Q_0$ at the representation (\ref{ground-3}), we obtain
\begin{equation}
\label{limit-E-Q}
E_0 \sim -\frac{2}{3} |\Lambda|^{\frac{3}{2}}, \quad
Q_0 \sim 2 |\Lambda|^{\frac{1}{2}}, \quad \mbox{\rm as} \quad
\Lambda \to -\infty.
\end{equation}
For $N$ solitary waves packed in the same graph $I_- \cup I_0 \cup I_+$ at different points, we estimate
$Q_0$ and $E_0$ roughly by multiplying (\ref{limit-E-Q}) by $N$, so that if $Q_0$ is preserved, then
\begin{equation}
|\Lambda|^{\frac{1}{2}} \sim \frac{Q_0}{2N} \quad \Rightarrow \quad E_0 \sim -\frac{Q_0^3}{12 N^2}.
\end{equation}
Therefore, the standing wave of minimal energy as $\Lambda \to -\infty$ corresponds to $N = 1$.
Two single solitary waves are given by Lemmas \ref{lemma-soliton-ring} and \ref{remark-soliton-line}.
The following lemma clarifies the energy levels for the symmetric and asymmetric waves.

\begin{lemma}
\label{lemma-limit}
For sufficiently large negative $\Lambda$, the symmetric wave of Lemma \ref{remark-soliton-line}
has a smaller $Q_0$ at a fixed $\Lambda$ compared to the asymmetric wave of Lemma \ref{lemma-soliton-ring}.
Consequently, it has larger $E_0$ at fixed $Q_0$ for sufficiently large $Q_0$.
\end{lemma}

\begin{proof}
By the scaling transformation (\ref{scaling}), we have $Q_0 = \mu \tilde{Q}_0$,
where
$$
\tilde{Q}_0 = \| \Psi \|_{L^2(J_- \cup J_0 \cup J_+)}^2 \quad \mbox{\rm and} \quad \mu = |\Lambda|^{1/2}.
$$
We represent $\tilde{Q}_0$ at the asymmetric wave of Lemma \ref{lemma-soliton-ring}
as the sum of three terms
$$
\tilde{Q}_0 = \| \Psi_+ \|_{L^2(J_+)}^2 + \| \Psi_0 \|_{L^2(J_0)}^2 + \| \Psi_- \|_{L^2(J_-)}^2.
$$
By the estimate (\ref{dnoidal-estimates}), there exists a positive $\mu$-independent constant $C$ such that
for sufficiently large $\mu$, we have
\begin{equation}
\label{estimate-charge-1}
\| \Psi_0 \|_{L^2(J_0)}^2 + \| \Psi_- \|_{L^2(J_-)}^2 \leq C e^{-2 \pi \mu}.
\end{equation}
On the other hand, by the explicit expression (\ref{dnoidal}), we have
\begin{align}
\nonumber
\| \Psi_+ \|_{L^2(J_+)}^2 & =  \frac{2}{\sqrt{2-k^2}} \int_0^{\pi \mu/\sqrt{2-k^2}} {\rm dn}(\xi;k)^2 d\xi \\
& =  \frac{2}{\sqrt{2-k^2}} \int_0^{K(k)} {\rm dn}(\xi;k)^2 d\xi + \frac{2}{\sqrt{2-k^2}}
\int_{K(k)}^{\pi \mu/\sqrt{2-k^2}} {\rm dn}(\xi;k)^2 d\xi. \label{estimate-charge-2}
\end{align}
It follows from the asymptotic expansions (\ref{dnoidal-expansion}) and (\ref{expansion-K}) that
the second term has the following asymptotic behavior as $\mu \to \infty$:
\begin{eqnarray}
\label{estimate-charge-0}
\left| \int_{K(k)}^{\pi \mu/\sqrt{2-k^2}} {\rm dn}(\xi;k)^2 d\xi \right| =
\int_{\pi \mu + \mathcal{O}(\mu e^{-\pi \mu})}^{\pi \mu + \log(\sqrt{3}) + \mathcal{O}(\mu e^{-2\pi \mu}, e^{-4L \mu})}
{\rm dn}(\xi;k)^2 d\xi = \mathcal{O}(e^{-2\pi \mu}),
\end{eqnarray}
since ${\rm dn}(\xi;k)^2 = \mathcal{O}(e^{-2\pi \mu})$ for every $\xi = \mathcal{O}(\mu)$ and $k \in (k_*(\mu),1)$
as $\mu \to \infty$. Therefore, the second term in (\ref{estimate-charge-2}) is comparable with the other
remainder terms in (\ref{estimate-charge-1}).
We will show that the first term in (\ref{estimate-charge-1}) has the larger value as $\mu \to \infty$.
We recall that (see 8.114 in \cite{Grad})
\begin{align*}
E(k) & :=  \int_0^{K(k)} {\rm dn}(\xi;k)^2 d\xi \\
 & =  1 + \frac{1}{2} (1-k^2) \left[ \log \frac{4}{\sqrt{1-k^2}} - \frac{1}{2} \right] + \mathcal{O}\left((1-k^2)^2|\log(1-k^2)|\right)\quad
\mbox{\rm as} \quad k \to 1,
\end{align*}
where $E(k)$ is a complete elliptic integral of the second kind. As a result, we obtain
\begin{eqnarray*}
\frac{2}{\sqrt{2-k^2}} \int_0^{K(k)} {\rm dn}(\xi;k)^2 d\xi = 2 + \frac{16}{3}
e^{-2\pi \mu} \left[ \pi \mu + \log(\sqrt{3}) - \frac{3}{2} \right]  + \mathcal{O}(\mu e^{-4 \pi \mu},e^{-2 \pi \mu - 4L \mu}).
\end{eqnarray*}
Combining this estimate with (\ref{estimate-charge-1}), (\ref{estimate-charge-2}), and (\ref{estimate-charge-0}),
we obtain for the asymmetric wave of Lemma \ref{lemma-soliton-ring} that
\begin{equation}
Q_0^{\rm asym} = 2 |\Lambda|^{1/2} + \frac{16}{3} \pi |\Lambda| e^{-2 \pi |\Lambda|^{1/2}} +
\mathcal{O}\left(|\Lambda|^{1/2} e^{-2 \pi |\Lambda|^{1/2}}\right) \quad
\mbox{\rm as} \quad \Lambda \to -\infty.
\end{equation}

We now report a similar computation for the symmetric wave of Lemma \ref{remark-soliton-line}.
We represent $\tilde{Q}_0$ for the symmetric wave as the sum of two terms
$$
\tilde{Q}_0 = \| \Psi_0 \|_{L^2(J_0)}^2 + 2 \| \Psi_+ \|_{L^2(J_+)}^2.
$$
By the estimate (\ref{cnoidal-estimates}), there exists a positive $\mu$-independent constant $C$ such that
for sufficiently large $\mu$, we have
\begin{equation}
\label{estimate-charge-1-cn}
 \| \Psi_+ \|_{L^2(J_+)}^2 \leq C e^{-2L \mu}.
\end{equation}
On the other hand, by the explicit expression (\ref{cnoidal}), we have
\begin{align}
\nonumber
\| \Psi_0 \|_{L^2(J_0)}^2 & =  \frac{2k^2}{\sqrt{2k^2-1}} \int_0^{\mu L/\sqrt{2k^2-1}} {\rm cn}(\xi;k)^2 d\xi \\
\nonumber
& =  \frac{2}{\sqrt{2k^2-1}} \int_0^{\mu L/\sqrt{2k^2-1}} {\rm dn}(\xi;k)^2 d\xi - \frac{2(1-k^2) \mu L}{2k^2-1} \\
& =  \frac{2}{\sqrt{2k^2-1}} \int_0^{K(k)} {\rm dn}(\xi;k)^2 d\xi
- \frac{2(1-k^2) \mu L}{2k^2-1} + \frac{2}{\sqrt{2k^2-1}}
\int_{K(k)}^{\mu L/\sqrt{2k^2-1}} {\rm dn}(\xi;k)^2 d\xi.
\label{estimate-charge-2-cn}
\end{align}
By the same estimate as in (\ref{estimate-charge-0}), the last term in (\ref{estimate-charge-2-cn}) is comparable with
the estimate (\ref{estimate-charge-1-cn}), whereas the other two terms give a larger contribution.
We now compute these terms explicitly
\begin{eqnarray*}
\frac{2}{\sqrt{2k^2-1}} E(k) - \frac{2(1-k^2) \mu L}{2k^2-1}  = 2 - \frac{16}{3} e^{-2 L \mu} \left[ L \mu +
\log(\sqrt{3}) - \frac{3}{2} \right] + \mathcal{O}(\mu e^{-4 L \mu},e^{-2L \mu - 4 \pi \mu}).
\end{eqnarray*}
Combining this estimate with (\ref{estimate-charge-1-cn}) and (\ref{estimate-charge-2-cn}),
we obtain for the symmetric wave of Lemma \ref{remark-soliton-line} that
\begin{equation}
Q_0^{\rm sym} = 2 |\Lambda|^{1/2} - \frac{16}{3} L |\Lambda| e^{-2 L |\Lambda|^{1/2}}
+ \mathcal{O}\left(|\Lambda|^{1/2} e^{-2 L |\Lambda|^{1/2}}\right) \quad
\mbox{\rm as} \quad \Lambda \to -\infty.
\end{equation}
For sufficiently large $|\Lambda|$, we have $Q_0^{\rm sym} < 2 |\Lambda|^{1/2} < Q_0^{\rm asym}$, which proves the
first assertion of the lemma.

We shall now prove that this estimate for $Q_0$ at a fixed $\Lambda$ can be transferred to the similar estimate for
$E_0$ at a fixed $Q_0$. This is done from the variational principle for the standing wave solutions
of the stationary NLS equation (\ref{statNLS}):
$$
\frac{d E_0}{d \Lambda} = \Lambda \frac{d Q_0}{d \Lambda},
$$
which implies that
\begin{equation}
\label{slope-E-versus-Q}
\frac{d E_0}{d Q_0} = \Lambda,
\end{equation}
where $\Lambda$ is supposed to be expressed from $Q_0$. Since $Q_0 = \mathcal{O}(|\Lambda|^{1/2})$ as $\Lambda \to -\infty$,
the dependence $\Lambda \mapsto Q_0$ is a decreasing diffeomorphism, which can be inverted. Indeed, for the symmetric wave
of Lemma \ref{remark-soliton-line}, we have
$$
|\Lambda|^{1/2} = \frac{1}{2} Q_0 + \frac{2}{3} L Q_0^2 e^{-L Q_0} + \mathcal{O}(Q_0 e^{-L Q_0}), \quad \mbox{\rm as} \quad
Q_0 \to \infty,
$$
so that (\ref{slope-E-versus-Q}) implies that
$$
E_0^{\rm sym} + \frac{1}{12} Q_0^3 = \int_{+\infty}^{Q_0} \left[ -\frac{2}{3} L Q_0^3 e^{-L Q_0} + \mathcal{O}(Q_0^2 e^{-L Q_0}) \right] d Q_0
> 0 \quad \mbox{\rm as} \quad Q_0 \to \infty.
$$
On the other hand, for the asymmetric wave
of Lemma \ref{lemma-soliton-ring}, we have
$$
|\Lambda|^{1/2} = \frac{1}{2} Q_0 - \frac{2}{3} \pi Q_0^2 e^{-\pi Q_0} + \mathcal{O}(Q_0 e^{-\pi Q_0}), \quad \mbox{\rm as} \quad
Q_0 \to \infty,
$$
so that (\ref{slope-E-versus-Q}) implies that
$$
E_0^{\rm asym} + \frac{1}{12} Q_0^3 = \int_{+\infty}^{Q_0} \left[  \frac{2}{3} \pi Q_0^3 e^{-\pi Q_0}
+ \mathcal{O}(Q_0^2 e^{-\pi Q_0}) \right] d Q_0 < 0 \quad \mbox{\rm as} \quad Q_0 \to \infty.
$$
Therefore, for sufficiently large $Q_0$, we have
$E_0^{\rm asym} < -\frac{1}{12} Q_0^3 < E_0^{\rm sym}$, which proves the second assertion of the lemma.
\end{proof}

\begin{lemma}
\label{lemma-eigenvalue}
There exists $\tilde{Q}_0^{**} \in (Q_0^{**},\infty)$ such that
the symmetric wave is a local constrained minimizer of energy 
for fixed $Q_0 \in (\tilde{Q}_0^{**},\infty)$. In particular, the second eigenvalue of the linearization operator $L_+$
at the symmetric wave of Lemma \ref{remark-soliton-line} is strictly positive.
\end{lemma}

\begin{proof}
By using the scaling transformation (\ref{scaling}), we transform the linearized operators $L_+$ and $L_-$
given by (\ref{L-plus}) and (\ref{L-minus}) to the form $L_{\pm} = |\Lambda| \ell_{\pm}$, where
\begin{align*}
\ell_+ & =  - \Delta_z + 1 - 6 \Psi^2, \\
\ell_- & =  - \Delta_z + 1 - 2 \Psi^2,
\end{align*}
where both operators are defined on the domain $\mathcal{D}(\Delta_z)$ in $L^2(J_- \cup J_- \cup J_+)$.
We consider the symmetric standing wave $\Psi$ given by Lemma \ref{remark-soliton-line}.

Since $\ell_- \Psi = 0$ and $\Psi(z) > 0$ for every $z \in J_- \cup J_- \cup J_+$ by Lemma \ref{remark-soliton-line},
the operator $\ell_-$ is positive definite. Therefore, we only need to show that
the operator $\ell_+$ has a simple negative eigenvalue and no zero eigenvalue. It is clear that
$\ell_+$ is not positive definite because $\langle \ell_+ \Psi, \Psi \rangle_{L^2} = - 4 \| \Psi \|_{L^4}^4 < 0$.

In the limit $\mu \to \infty$, $\ell_+$ converges pointwise to the operator
$$
\ell_+^{\infty} = -\frac{d^2}{d z^2} + 1 - 6 {\rm sech}^2(z) : H^2(\mathbb{R}) \to L^2(\mathbb{R}),
$$
which admits a simple negative eigenvalue and a simple zero eigenvalue. Therefore, we only need to
show that the simple zero eigenvalue of $\ell_+^{\infty}$ becomes a positive eigenvalue of $\ell_+$ for large but finite $\mu$.

We note that $\ell_+ \Psi'(z) = 0$, although $\Psi'(z)$ does not satisfy the Kirchhoff boundary conditions
in $\mathcal{D}(\Delta_z)$. To correct the boundary conditions, we write an eigenfunction $U \in \mathcal{D}(\Delta_z)$ of
the eigenvalue problem $\ell_+ U = \lambda U$ in the product form $U(z) = a(z) \Psi'(z)$.
The amplitude function $a : J_- \cup J_- \cup J_+ \to \mathbb{R}$ ensures that the eigenfunction
$U$ satisfies the Kirchhoff boundary conditions.

If $\Psi$ is even, then $\Psi'$ is odd, whereas the amplitude $a$ is even with respect to $z$. We recall that
$$
\Psi_+'(L \mu) = -\Psi_+'(L \mu + 2 \pi \mu) = \frac{1}{2} \Psi_0'(L \mu),
$$
due to spatial symmetry of the component $\Psi_+$. Therefore, the continuity boundary conditions
$$
a_0(L \mu) \Psi_0'(L \mu) = a_+(L \mu) \Psi_+'(L \mu) = a_+(L \mu + 2\pi \mu) \Psi_+'(L \mu + 2\pi \mu)
$$
yield the boundary values for $a$:
\begin{equation}
\label{bc-a-1}
2 a_0(L \mu) = a_+(L \mu) = -a_+(L \mu + 2\pi \mu).
\end{equation}
On the other hand, $\Psi''$ is expressed by the stationary NLS equation (\ref{statNLS-limit-graph}),
so that $\Psi''$ is continuous at the vertex points. Therefore, the derivative boundary condition
$$
U_0'(L \mu) = U_+'(L \mu) - U_+'(L \mu + 2 \pi \mu)
$$
yields the boundary values for the derivative of $a$:
\begin{equation}
\label{bc-a-2}
\Psi_+'(L \mu) \left[ 2 a_0'(L \mu) - a_+'(L \mu) - a_+'(L \mu + 2 \pi \mu) \right] = \frac{3}{2} a_0(L \mu) \Psi_+''(L \mu).
\end{equation}
Thanks to the symmetry condition, we are looking for odd $U$ and even $a$, so that the conditions on $a$ at
the other vertex point repeat boundary conditions (\ref{bc-a-1}) and (\ref{bc-a-2}).

After the boundary conditions (\ref{bc-a-1}) and (\ref{bc-a-2}) are identified, we substitute the
product form $U(z) = a(z) \Psi'(z)$ into the eigenvalue problem $\ell_+ U = \lambda U$. After multiplying the
resulting equation by $\Psi'(z)$, we obtain
$$
- \frac{d}{dz} \left[ \left( \Psi' \right)^2 \frac{da}{dz} \right] = \lambda \left( \Psi' \right)^2 a.
$$
After multiplying this equation by $a$ and integrating by parts, we obtain
\begin{equation}
\label{quad-form}
\lambda \| a \Psi' \|_{L^2}^2 = \| a' \Psi' \|^2_{L^2} - \left[ a a' \left( \Psi' \right)^2 \right],
\end{equation}
where the squared brackets indicate the total jump at the vertex points:
$$
\left[ f \right] := f_-(-L \mu) - f_-(-L \mu - 2\pi \mu) + f_0(L \mu) - f_0(-L \mu) + f_+(L \mu + 2 \pi \mu) - f_+(L \mu).
$$
To compute the total jump explicitly, we use the symmetry on $a$ and $\Psi'$,
as well as the boundary conditions (\ref{bc-a-1}) and (\ref{bc-a-2}). As a result of straightforward computations, we obtain
$$
\left[ a a' \left( \Psi' \right)^2 \right] = 3 a_+(L \mu)^2 \Psi_+'(L \mu) \Psi_+''(L \mu).
$$
By Lemma \ref{remark-soliton-line}, for sufficiently large $\mu$, we have $\Psi_+''(L \mu) > 0$ and $\Psi_+'(L \mu) < 0$,
therefore, $\left[ a a' \left( \Psi' \right)^2 \right] < 0$.
The quadratic form (\ref{quad-form}) implies that the corresponding
eigenvalue $\lambda$ is positive. This completes the proof of the lemma.
\end{proof}

\begin{remark}
Lemmas \ref{lemma-limit} and \ref{lemma-eigenvalue} prove the assertions of Theorem \ref{theorem-limit}.
\end{remark}

\begin{remark}
The method of the proof of Lemma \ref{lemma-eigenvalue} is inconclusive
for the asymmetric standing wave given by Lemma \ref{lemma-soliton-ring}. Indeed, the total jump
condition without spatial symmetry of the asymmetric standing wave $\Psi$ are given by
$$
\left[ a a' \left( \Psi' \right)^2 \right] =
\frac{3}{2} a_+(L \mu)^2 \Psi_+'(L \mu) \Psi_+''(L \mu) -
\frac{3}{2} a_-(-L \mu)^2 \Psi_-'(-L \mu) \Psi_-''(-L \mu),
$$
where $\Psi'_{\pm}(\pm L \mu) > 0$ and $\Psi''_{\pm}(\pm L \mu) > 0$ if $\mu$ is sufficiently large.
Therefore, the sign of $\left[ a a' \left( \Psi' \right)^2 \right]$ depends on the balance
between $a_+(L \mu)$ and $a_-(-L \mu)$ relative to $\Psi_+(L\mu)$ and $\Psi_-(-L \mu)$.  \end{remark}

\section{Numerical approximations of the ground state}

Here we illustrate numerically the construction of the standing waves of the
stationary NLS equation (\ref{statNLS}) and the corresponding ground state
of the constrained minimization problem (\ref{minimizer}).

\subsection{Numerical Methods}

To compute solutions of the stationary NLS equation (\ref{statNLS}) for a given $\Lambda$,
we will use both a Newton's method
(largely the {\it Matlab} based program {\it nsoli}),
as well as the Petviashvilli method \cite{Pelin-book,Yang-book}.
Rigorous convergence estimates for the Petviashvili method have been established in \cite{PS}
and recently refined in \cite{OSSS}.  Often, we will use the Petviashvilli method
to initially land on a branch, then continue using the more delicate Newton's method machinery.
The basic approach to the Petviashvilli method starts
with an initial guess $u_0$, which will generally take to be a Gaussian function centered either
at the central link or at one of the two loops.  Then, for a given $\Lambda < 0$, we construct solutions to
the stationary NLS equation \eqref{statNLS} by defining
\begin{equation*}
u_{n+1} = M[u_n]^\gamma ( |\Lambda| - \Delta)^{-1} (2 |u_n|^2 u_n)
\end{equation*}
with
\begin{equation*}
M[u] = \frac{ \langle  (|\Lambda| - \Delta) u, u \rangle_{L^2}}{ 2 \|u\|^4_{L^4}}
\end{equation*}
and $\gamma = \frac{3}{2}$.  We iterate until $|u_{n+1}-u_n| < 10^{-14}$, then
declare such the final state to be a good numerical approximation to a fixed point.
Once we have constructed a solution that is centered either at the central link or
at one of the two loops, the Newton solver can be used to continue that branch with great accuracy.

In order to set up our discretization, we approximate the Laplacian operator $\Delta$
using a second-order symmetric finite difference
stencil with $N$ uniformly space grid points of size $h = 2 \pi/N$ on each
of the two loops of length $2 \pi$ and $M$ grid points on the interior
section of length $2L$.   In order to allow an approximation of the $\mathcal{O}(h^2)$ order,
we first ensure that the discretized Laplacian operator is symmetric by taking $L = m \pi$
for $m \in \NN$ or $1/m \in \NN$ and choose $N$ such that $M = N L/\pi \in \NN$.
To enforce the Kirchhoff boundary conditions \eqref{bc-1}, we take the higher order difference calculations
\begin{eqnarray*}
u_0'(-L) & = & \frac{-u_0(-L+2h)+4u_0(-L+h)-3u_0(-L)}{2h},  \\
u_-'(-L) & = & \frac{u_-(-L-2h)-4u_-(-L-h)+3u_-(-L)}{2h}, \\
u_-'(-L-2 \pi) & = & \frac{-u_-(-L-2\pi+2h)+4u_-(-L-2\pi+h)-3u_-(-L)}{2h},
\end{eqnarray*}
which allows us to replace $u_0 (-L) = u_- (-L) = u_- (-L-2 \pi)$ everywhere
it appears in the symmetric difference for the Laplacian.  There is a
symmetric argument for the other Kirchhoff boundary condition in \eqref{bc-2}.

\subsection{Numerical Findings}

Using the graph Laplacian approximated to the second order and various methods for
constructing solutions of the stationary NLS equation (\ref{statNLS}) for a given $\Lambda < 0$,
we attempt to verify various properties the ground state branches discussed in Theorems
\ref{theorem-graph} and \ref{theorem-limit}.    We present the results of our
various numerical studies in Figures \ref{stat2}, \ref{stat3}, \ref{statcomp}, \ref{stat8},
\ref{stat9}, and \ref{stat10}.   The code for computing these are made
publicly available at \url{www.unc.edu/~marzuola/mp_graph_code/}.

In Figure \ref{stat2}, we plot the form of the ground state computed using
Petviashvili's method with symmetric initial guess localized in the center link
for $L = \pi/2$ and a variety of $\Lambda$ values.  For values of $\Lambda = -0.01, -0.1, -1.5, -10.0$,
we observe the computed ground state go from the constant solution (\ref{ground-1}) as in Lemma \ref{lemma-1},
to the positive asymmetric wave as in Lemma \ref{lemma-3}, then in an intermediate region
to the positive asymmetric wave as in Lemma \ref{lemma-soliton-ring},
finally, settling on the positive symmetric wave as in  Lemma \ref{remark-soliton-line}.

\begin{figure}[htp]
\begin{tabular}{cc}
\includegraphics[width=5.25cm]{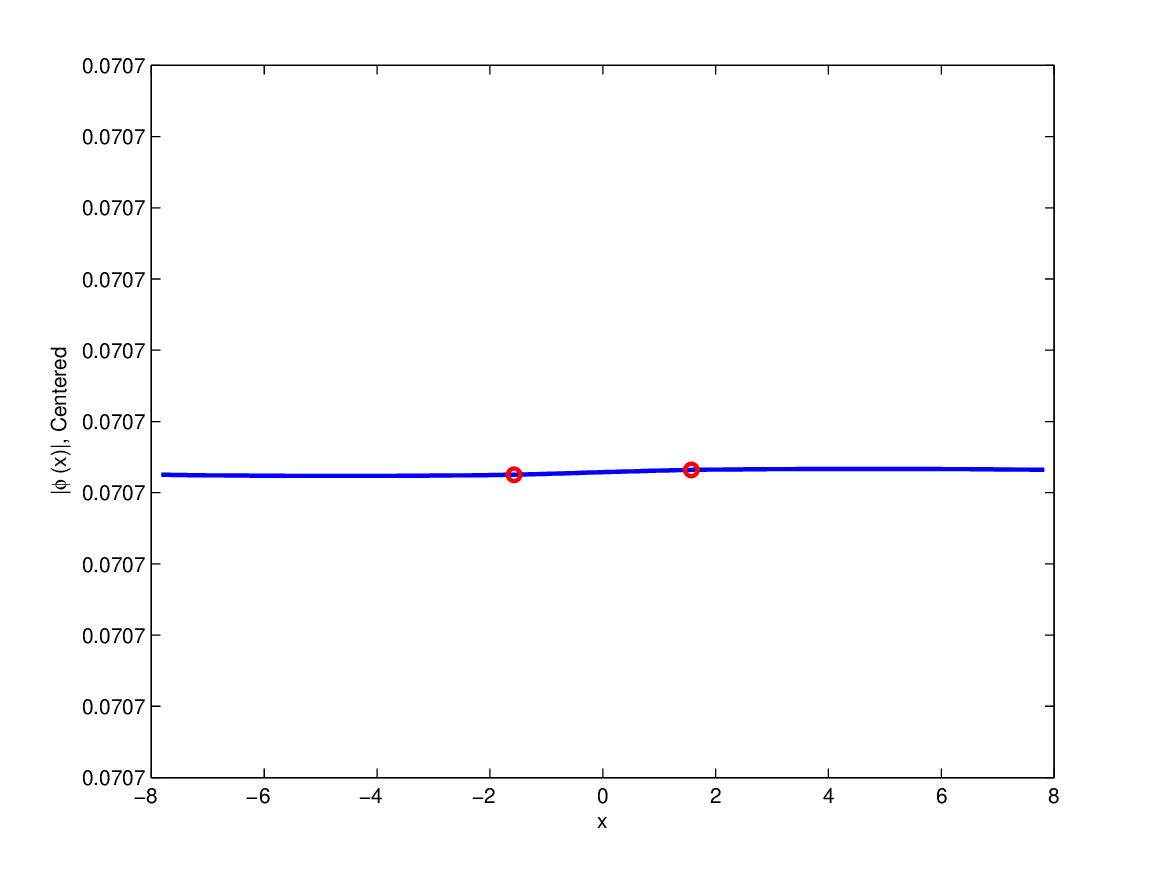} &
\includegraphics[width=5.25cm]{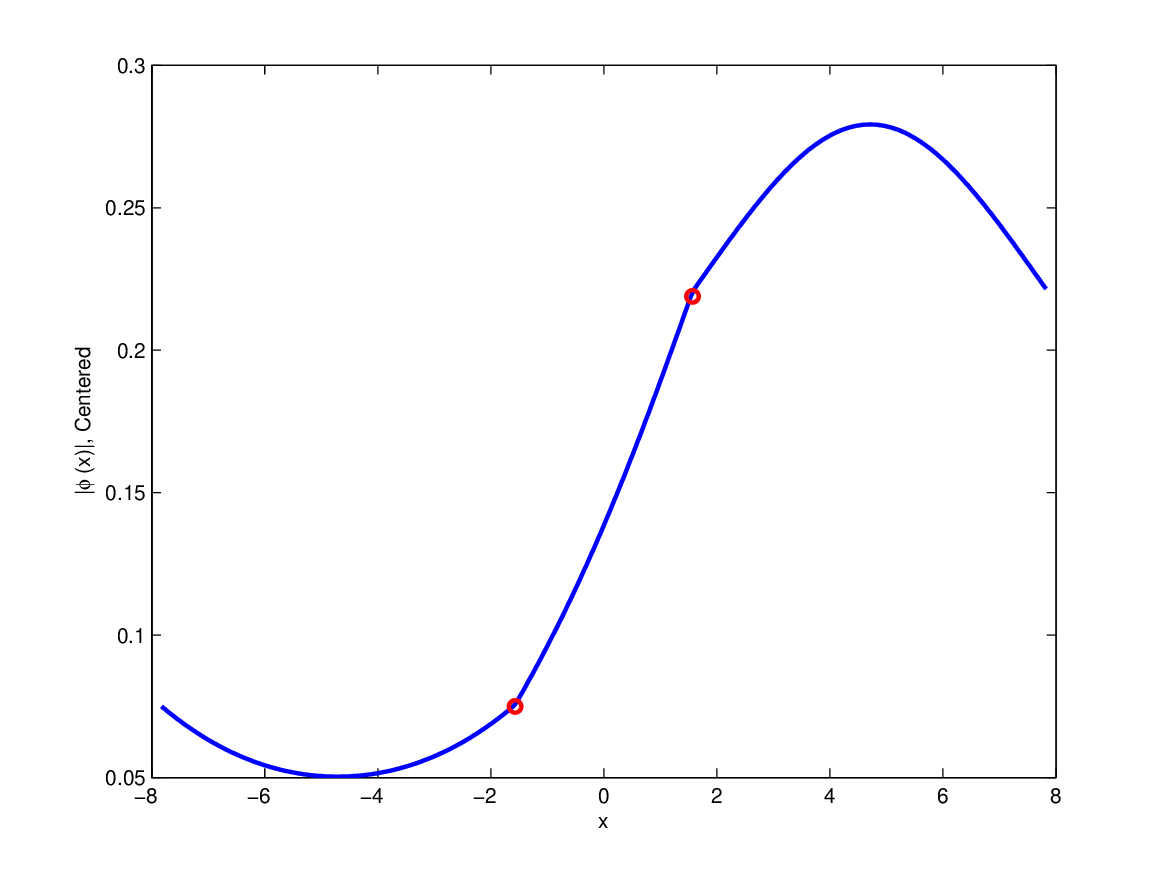} \\
\includegraphics[width=5.25cm]{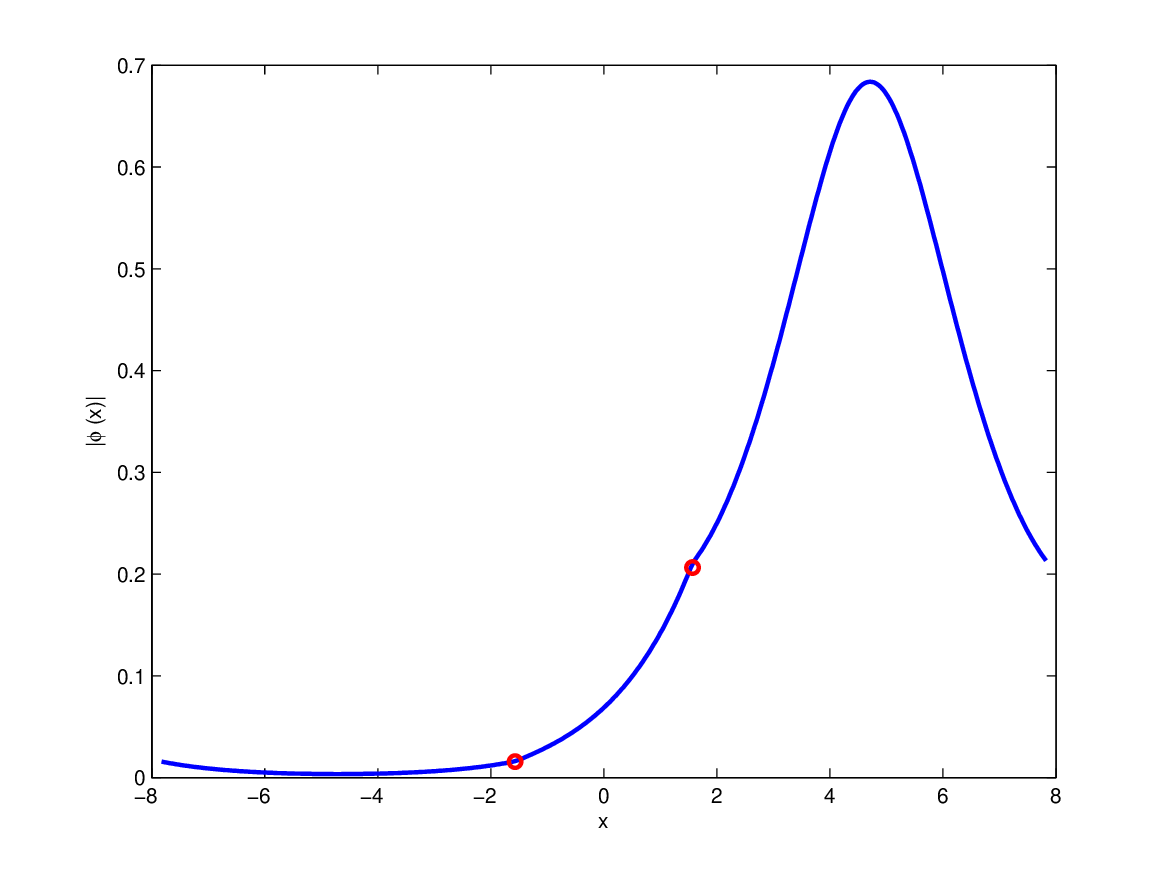} &
\includegraphics[width=5.25cm]{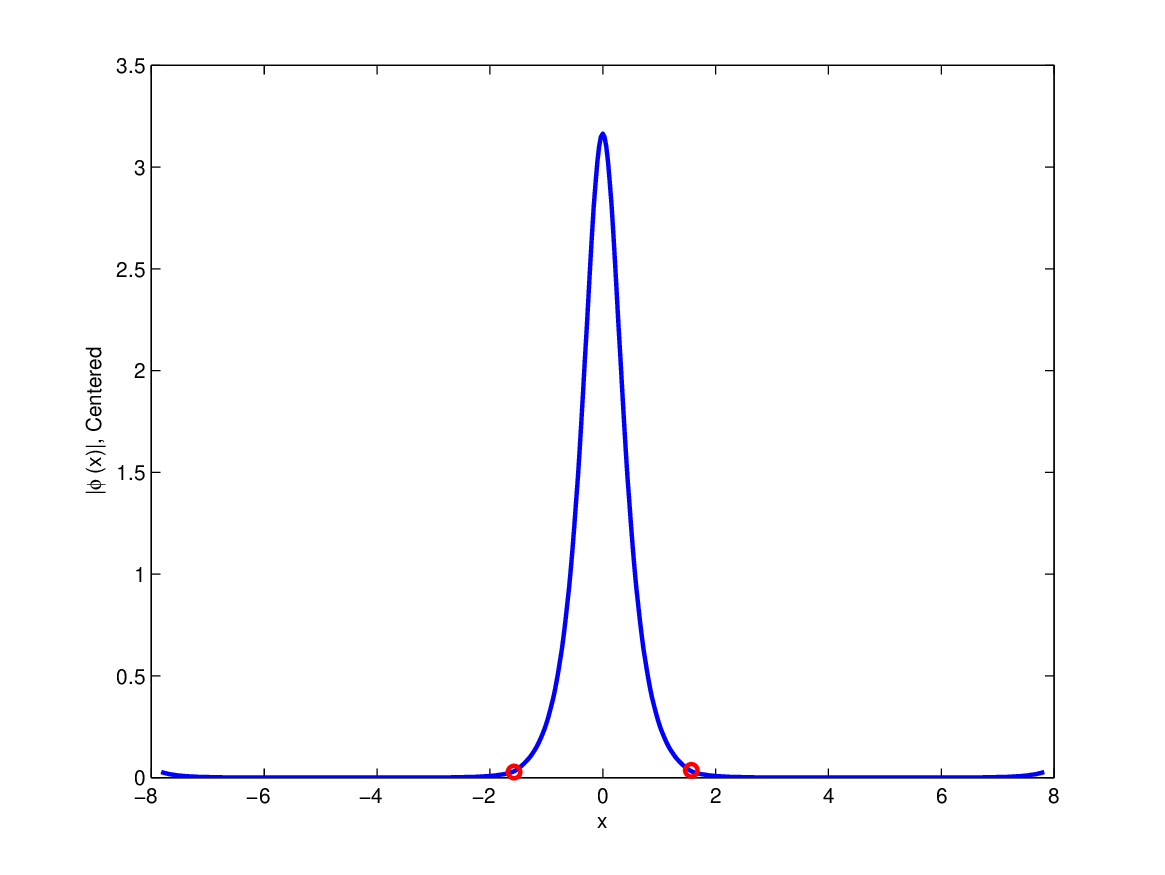} \\
\end{tabular}
\caption{Solutions of the stationary NLS equation (\ref{statNLS}) for $L = \pi/2$,
when the initial iterate is localized on the central link, for $\Lambda = -0.01$ (top left),
$\Lambda = -0.1$ (top right), $\Lambda = -1.5$ (bottom left), and $\Lambda = -10.0$ (bottom right).
The values at $\pm L$ are marked with a red circle. }
\label{stat2}
\end{figure}

\begin{figure}[htp]
\begin{tabular}{cc}
\includegraphics[width=5.25cm]{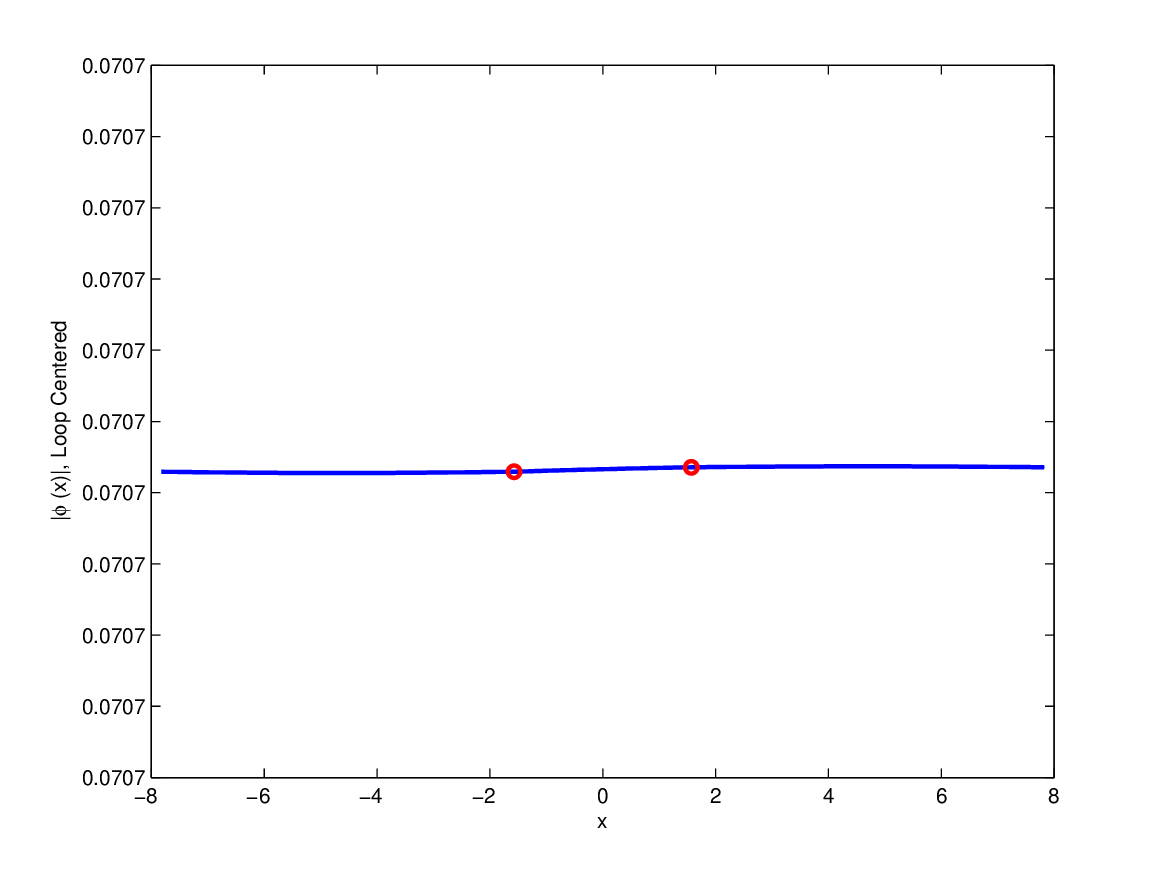} &
\includegraphics[width=5.25cm]{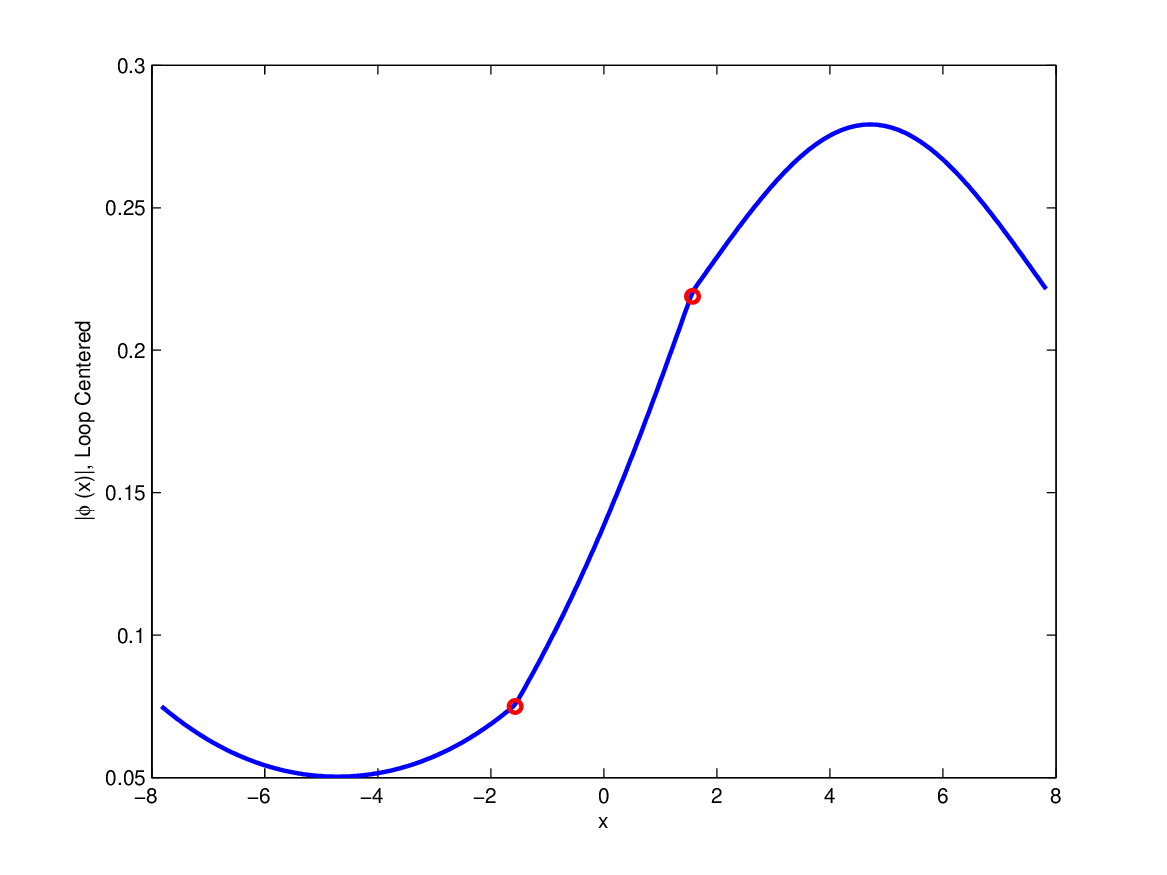} \\
\includegraphics[width=5.25cm]{barbell_nls_loopcent_omega_0pt5_LeqPi_over_2.eps} &
\includegraphics[width=5.25cm]{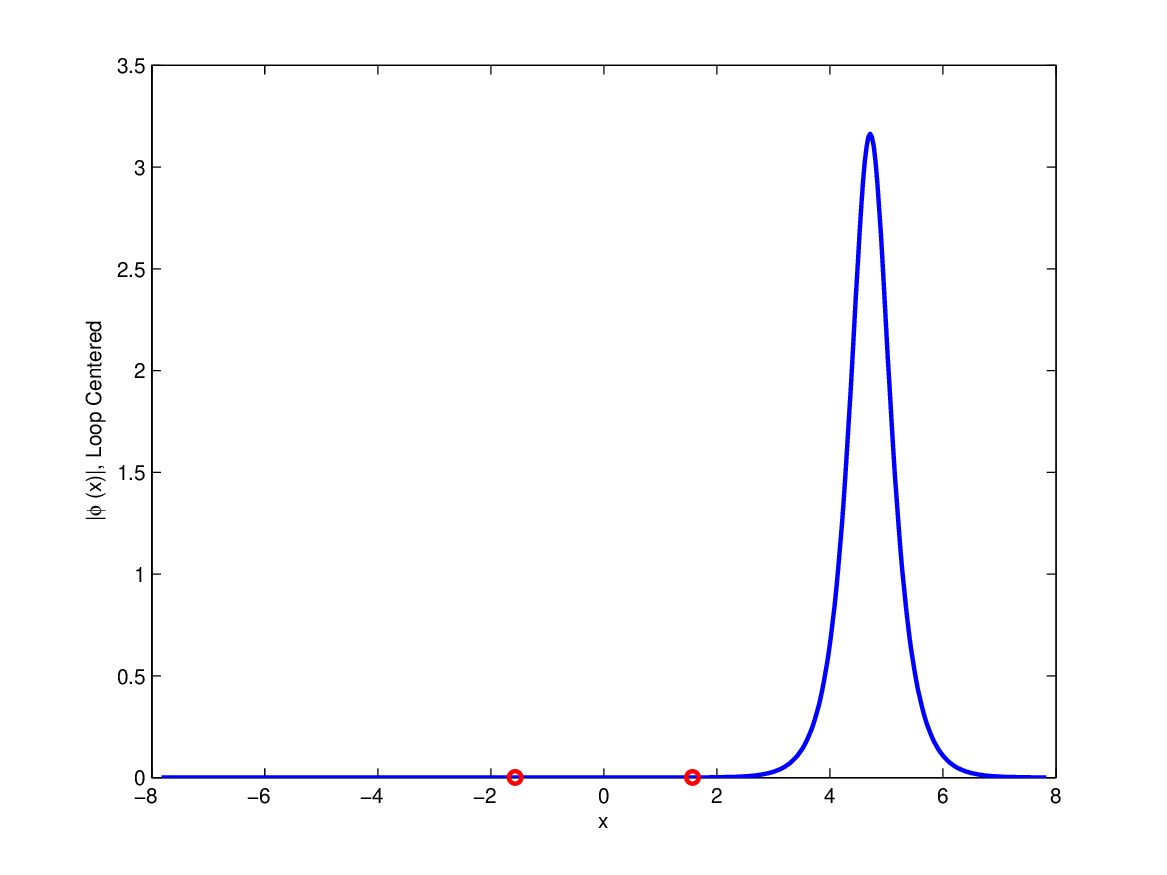} \\
\end{tabular}
\caption{The same as in Figure \ref{stat2} but for the initial iterate being localized on the right ring.}
\label{stat3}
\end{figure}

\begin{figure}[htp]
\begin{tabular}{cc}
\includegraphics[width=5.25cm]{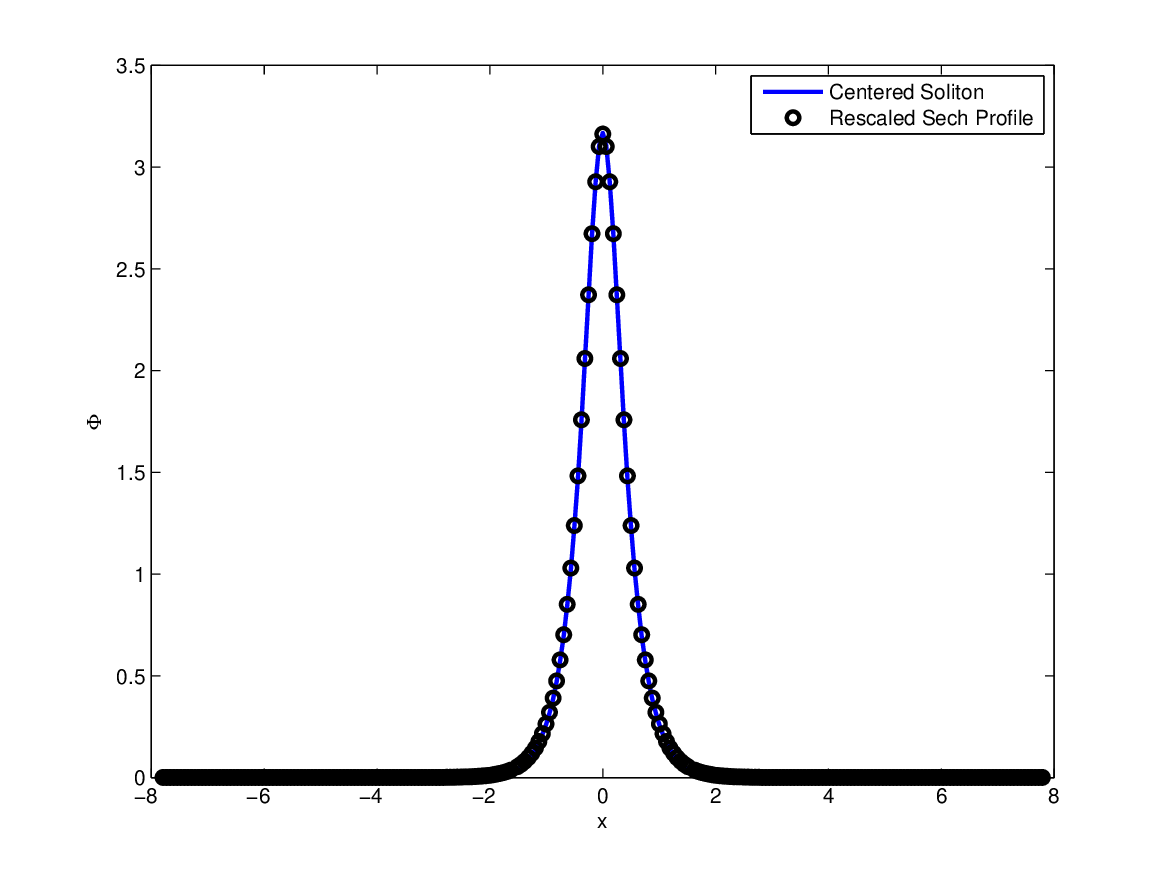} &
\includegraphics[width=5.25cm]{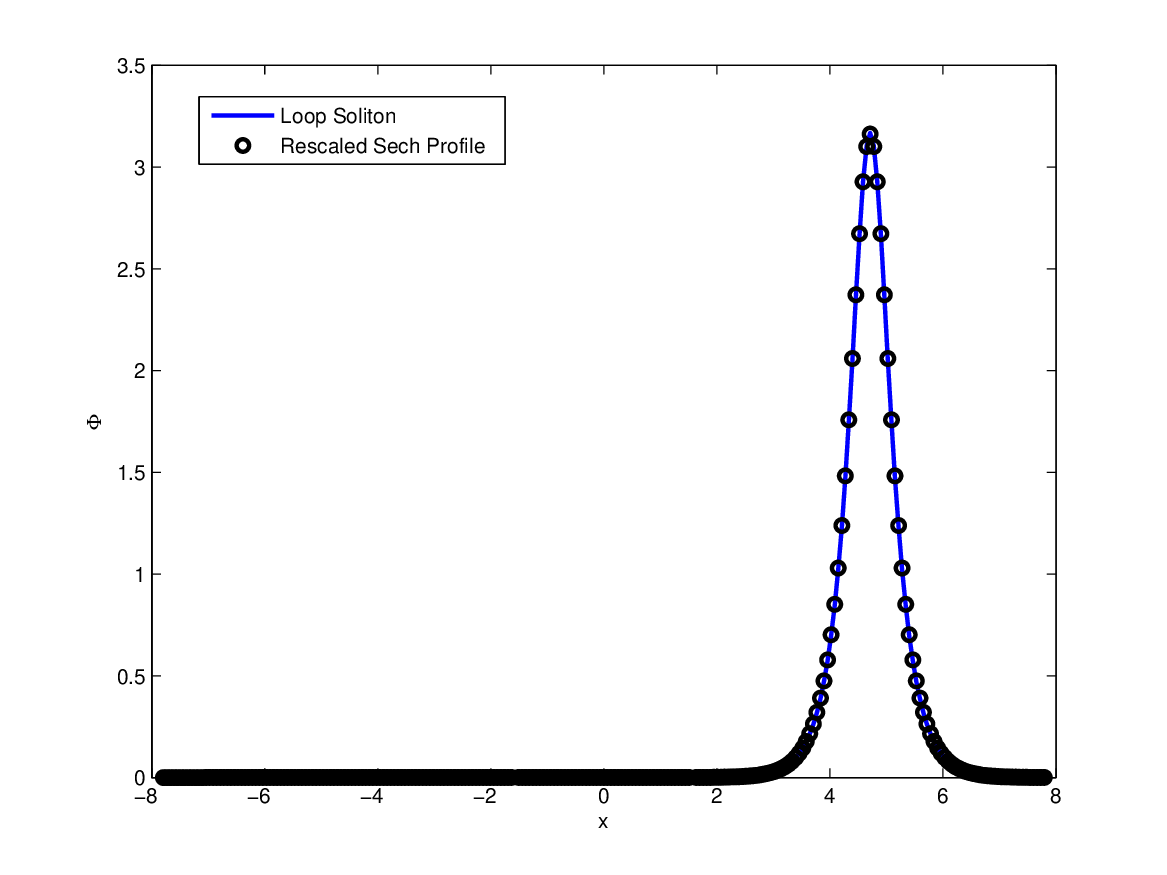} \\
\includegraphics[width=5.25cm]{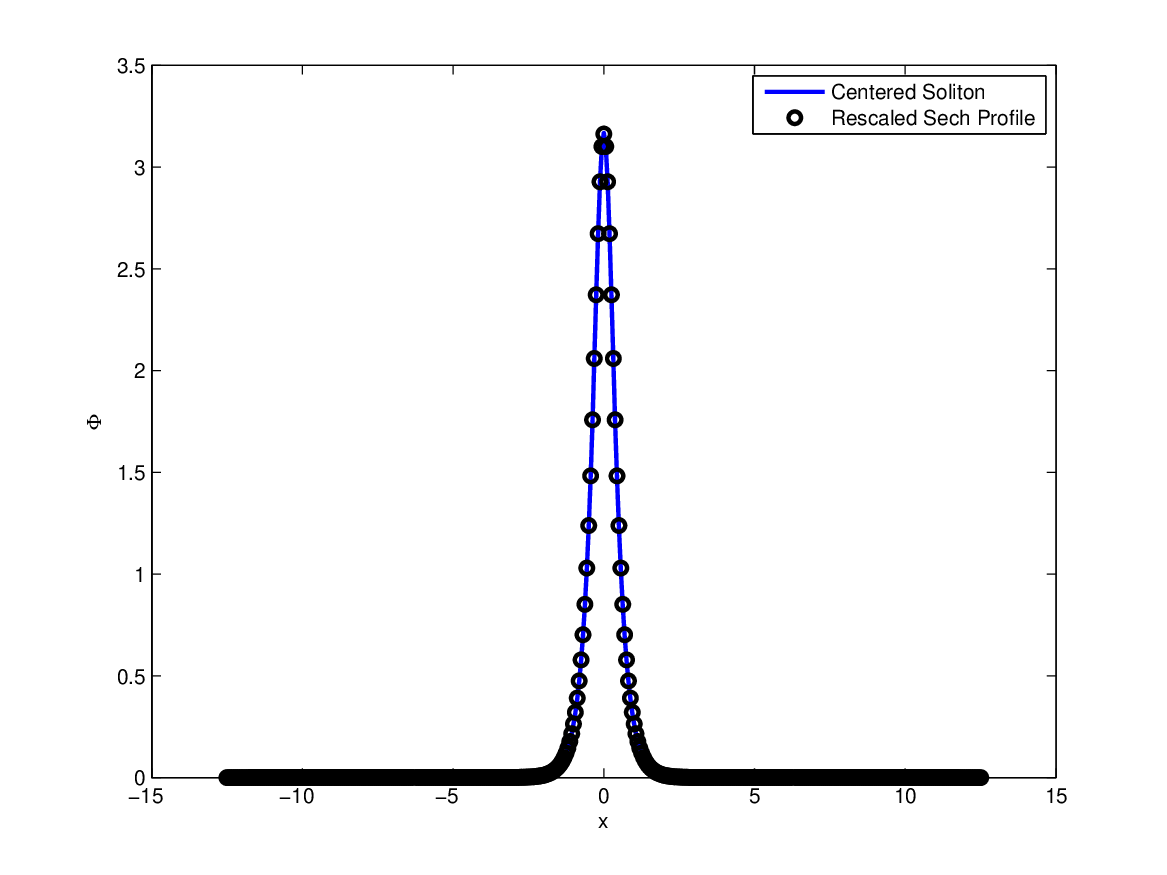} &
\includegraphics[width=5.25cm]{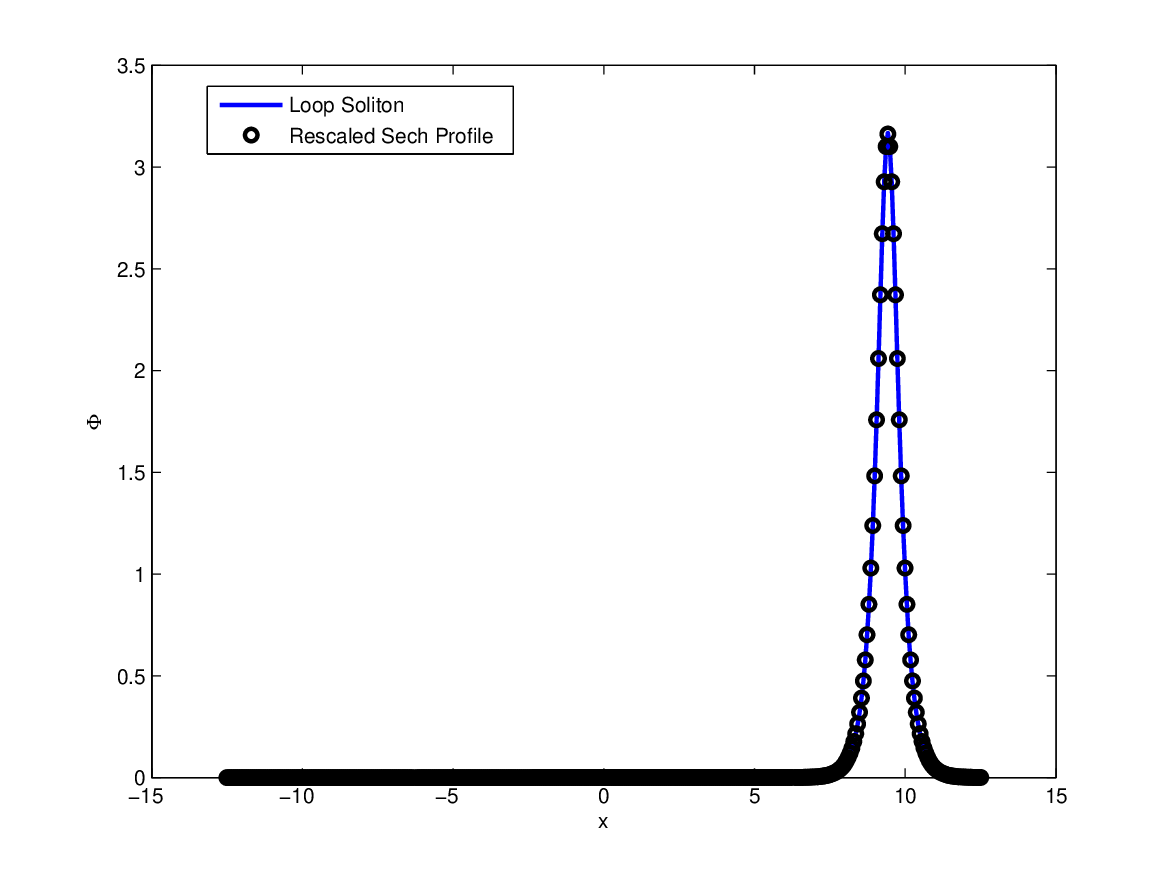} \\
\end{tabular}
\caption{Comparison of the standing waves (solid line) localized
in the central link (left) and in one of the rings (right) to the rescaled solitary wave
profile (dots) for $L = \pi/2$ (top) and $L = 2 \pi$ (bottom) for $\Lambda = -10.0$.}
\label{statcomp}
\end{figure}

\begin{figure}[htp]
\begin{tabular}{cc}
\includegraphics[width=7cm]{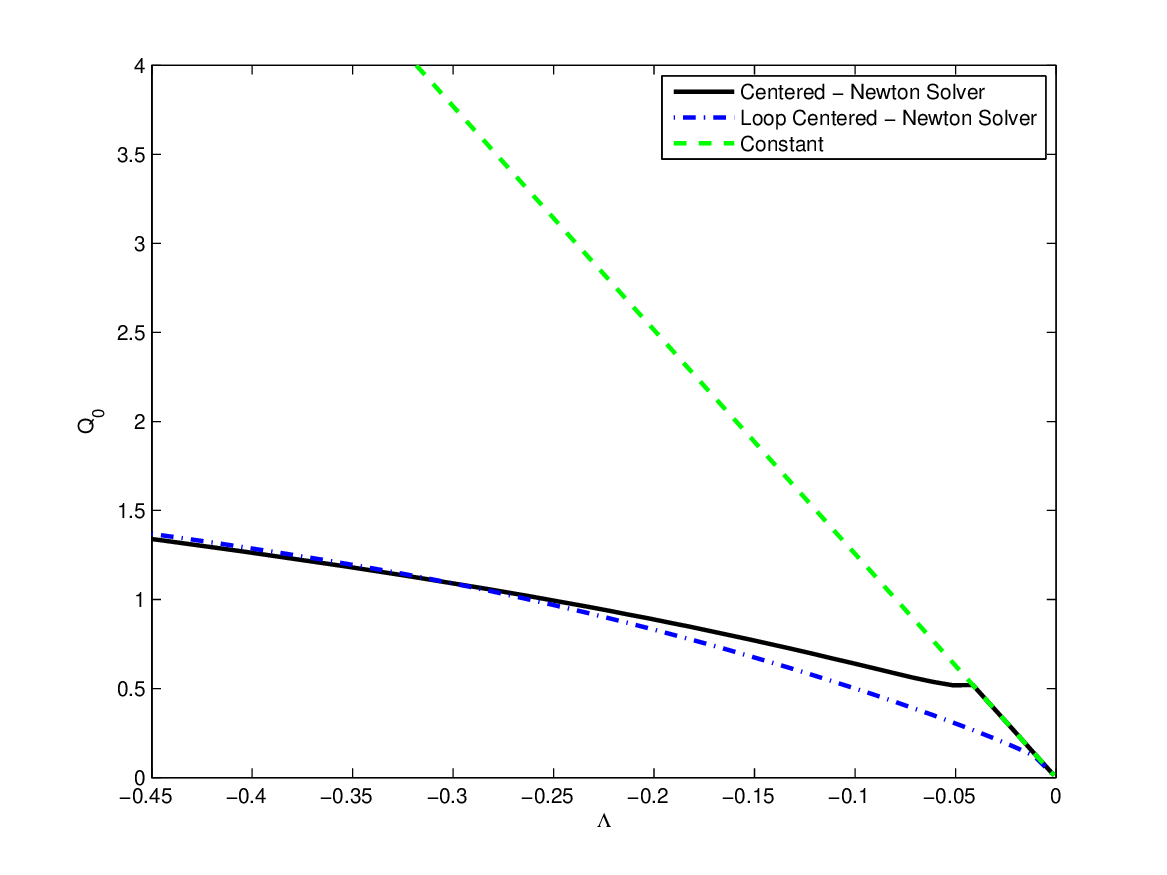} &
 \includegraphics[width=7cm]{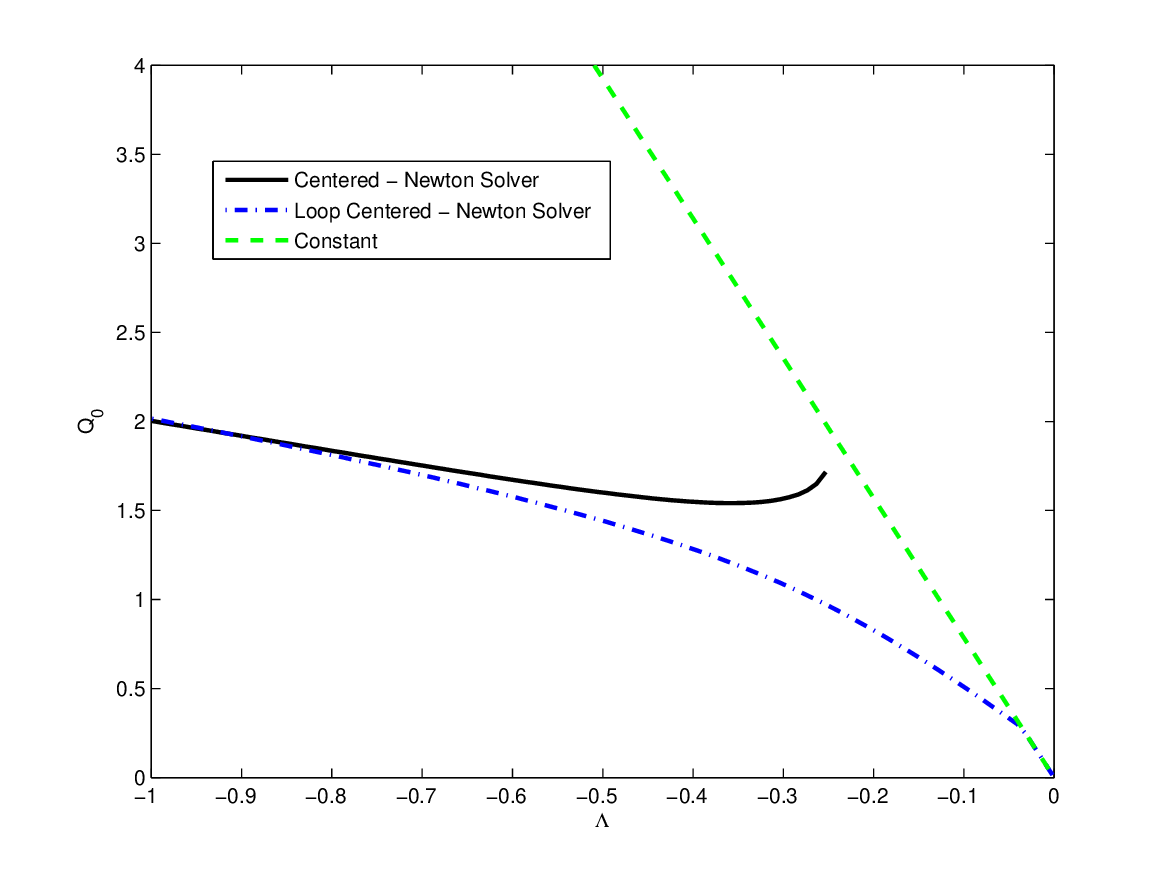}
\end{tabular}
\caption{The bifurcation diagram for two positive solutions localized in the central link and in one of the two rings
as well as for the constant solution for $L = 2 \pi$ (left) and $L =  \pi/2$ (right).}
\label{stat8}
\end{figure}

\begin{figure}[htp]
\includegraphics[width=7cm]{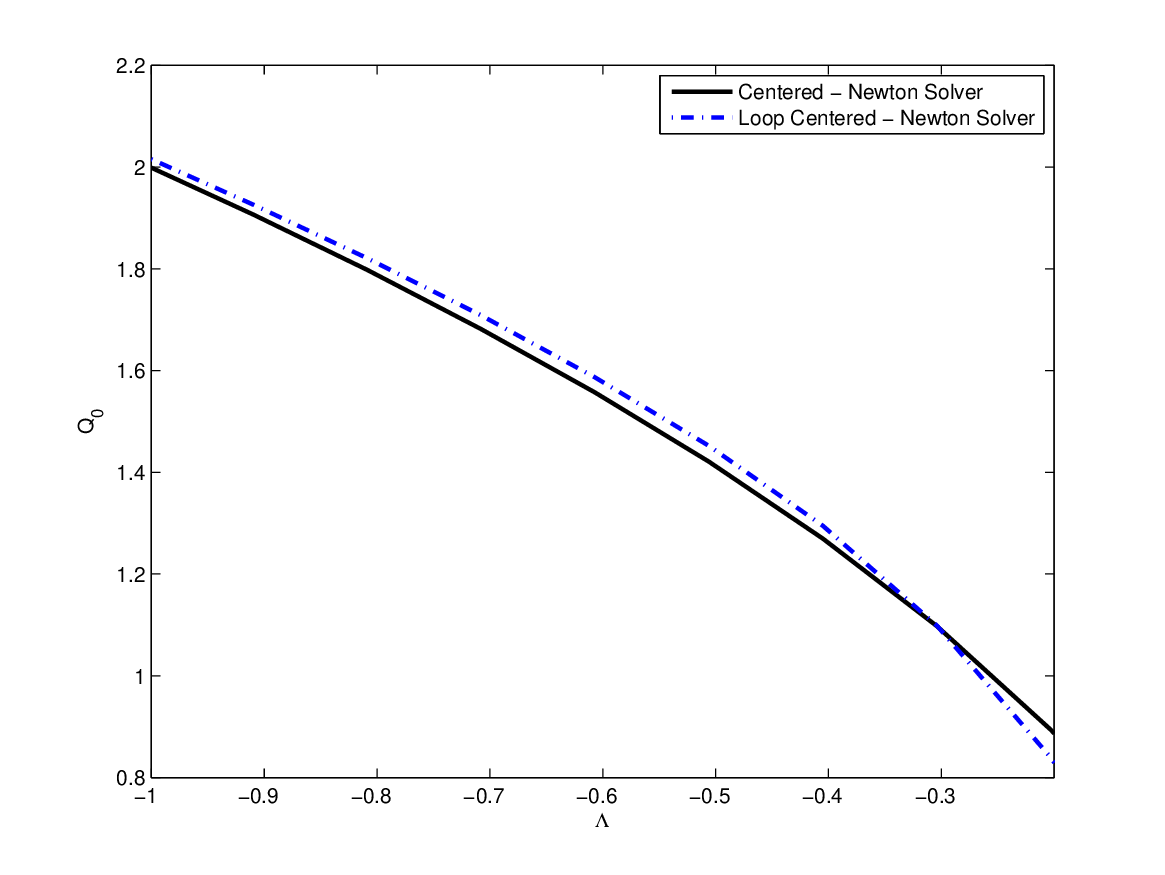}
 \includegraphics[width=7cm]{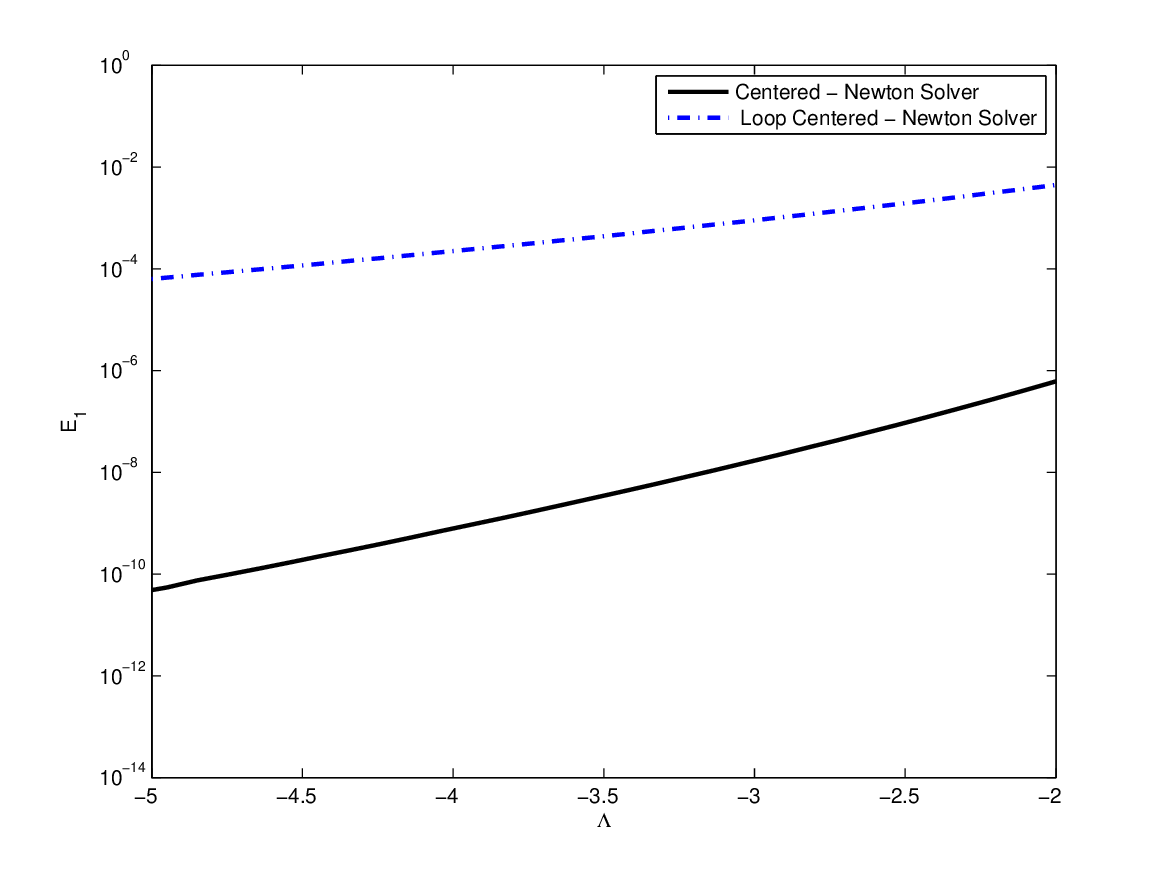}  \\
\caption{Left: comparison of the bifurcation diagram for positive solutions localized in the central link and in one of the two rings
for $L = 2 \pi$. Right:  the second eigenvalue of the linearization operator $L_+$ associated with the two positive
solutions plotted with respect to $\Lambda$ in a semilogy plot. }
\label{stat9}
\end{figure}

\begin{figure}[htp]
\includegraphics[width=7cm]{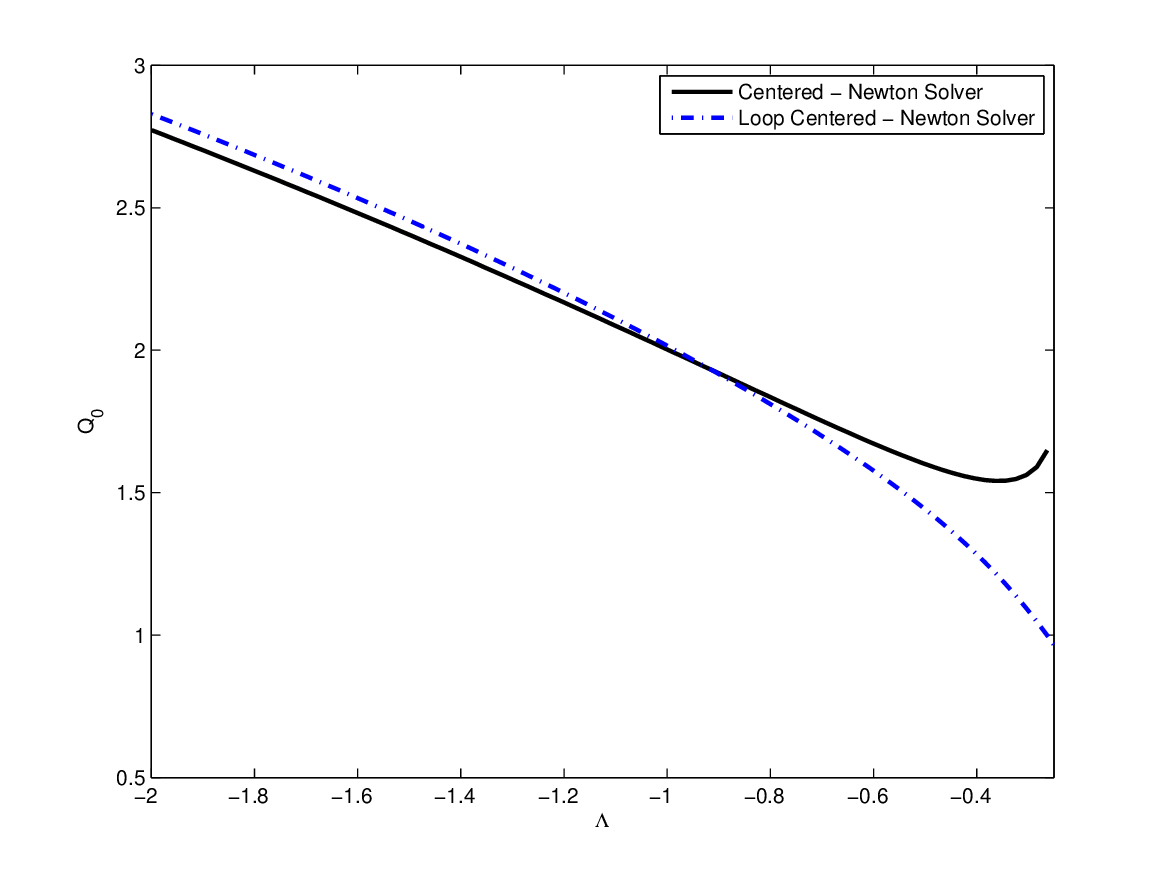}
 \includegraphics[width=7cm]{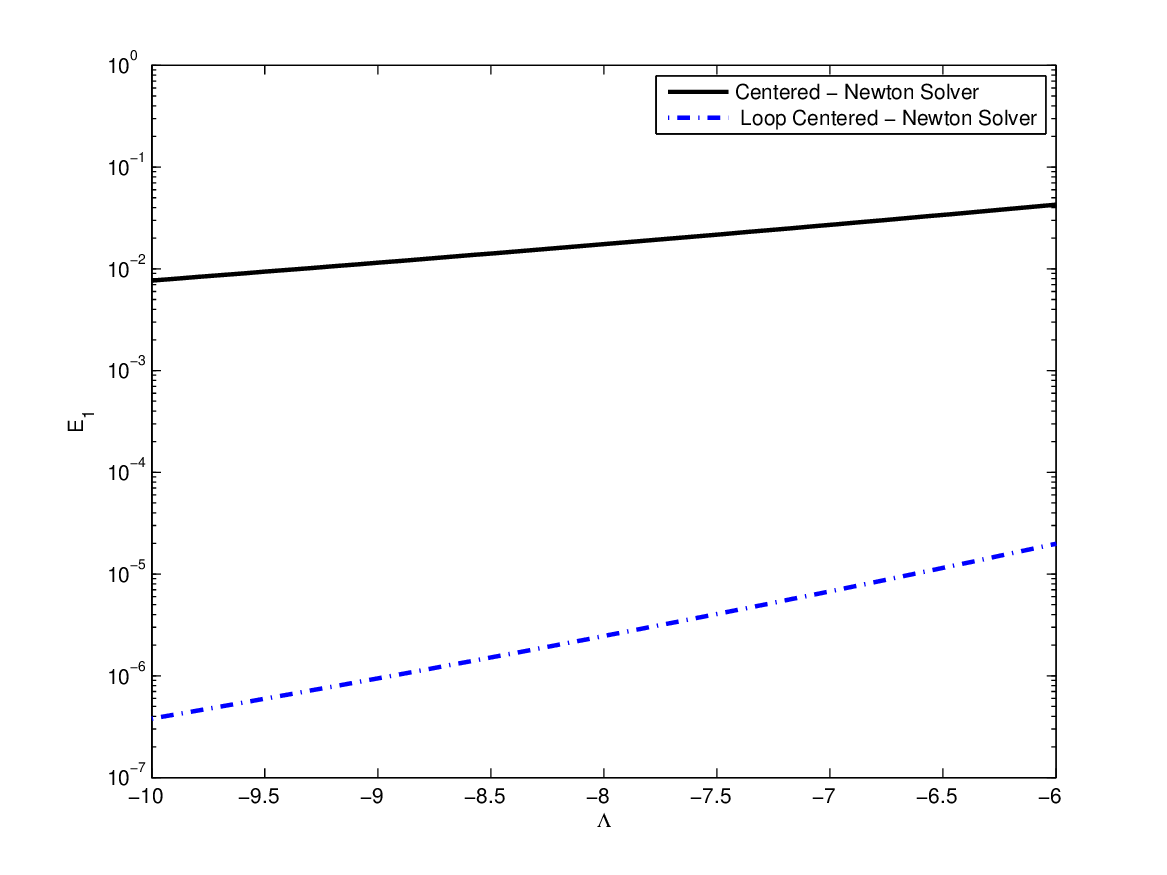}
\caption{The same as Figure \ref{stat9} but for $L =  \pi/2$.}
\label{stat10}
\end{figure}

Figure \ref{stat3} plots the form of the stationary solution
computed using Petviashvili's method with a loop centered initial iterate
for $L = \pi/2$ and $\Lambda$ values $ -0.01, -0.1, -1.5, -10.0$.
Compared to the outcome on Figure \ref{stat2}, we observe
that the positive asymmetric wave remains to exist at least up to $\Lambda = -10$.
Both positive waves (\ref{ground-2}) and (\ref{ground-3}) coexist for large
negative $\Lambda$ as in Theorem \ref{theorem-limit}.

Figure \ref{statcomp} plots standing waves for $\Lambda = -10$ localized in the central link and in one of the two rings
and compares them to the appropriately rescaled solitary wave (\ref{NLS-soliton}).
The agreement illustrate the representations (\ref{ground-2}) and (\ref{ground-3}) in Theorem \ref{theorem-limit}.

Figure \ref{stat8} shows the evolution of the charge $Q_0$ with respect to $\Lambda$ for
the constant solution and the two positive waves continuing from the constant state
for $L = 2 \pi$ (left) and $L = \pi/2$ (right).  We compute these branches by
using both Petviashvili's method for sufficiently large $|\Lambda|$, then continuing
with Newton's method towards the constant branch. The bifurcation
of the positive asymmetric wave from the constant state is the pitchfork bifurcation as
shown in Lemma \ref{lemma-3}. 

In Figures \ref{stat9} and \ref{stat10}, we numerically continue to large values of $|\Lambda|$
and explore two properties discussed in Lemmas \ref{lemma-limit} and \ref{lemma-eigenvalue} for
$L = 2 \pi$ and $L = \pi/2$ respectively. The left panels show the bifurcation diagram zoomed for an interval of large $|\Lambda|$.
The two positive waves are computed via Newton iteration from peaked states at $\Lambda = 10.0$
concentrated in either the loop or the center link.  In agreement with Lemma \ref{lemma-limit},
we can see that the positive symmetric wave has the smaller value of $Q_0$ for a fixed $\Lambda$.
We also observe that the values of $Q_0$ are much closer for large $L$ than for small $L$. Although only a short interval of
$\Lambda$ values is shown, the same trend continues all way to $\Lambda = -10$.

The right panels of Figures \ref{stat9} and \ref{stat10} show the evolution of the second eigenvalue of the linearized operator $L_+$
for both the positive waves. In agreement with Lemma \ref{lemma-eigenvalue}, the positivity of the second eigenvalue of $L_+$ for
the positive symmetric wave represented by \eqref{ground-3} is observed to be quite robust
for both $L = 2 \pi$ and $L = \pi/2$.  However, we also observe that the second eigenvalue of $L_+$
for the positive asymmetric wave represented by \eqref{ground-2} is also positive,
which verifies that the positive asymmetric and symmetric waves are both local constrained minimizers of 
the energy at a fixed charge.

\appendix

\section{Proof of Proposition \ref{prop-elliptic}}

Proposition \ref{prop-elliptic} generalizes formal asymptotic expansions given by formulas 16.15 in \cite{AS}:
\begin{align}
\label{sn-der-appendix}
& {\rm sn}(\xi;k)  =  \tanh(\xi) + \frac{1}{4} (1-k^2) \left[ \sinh(\xi) \cosh(\xi) - \xi \right] {\rm sech}^2(\xi) + \mathcal{O}((1-k^2)^2), \\
\label{cn-der-appendix}
& {\rm cn}(\xi;k)  =  {\rm sech}(\xi) - \frac{1}{4} (1-k^2) \left[ \sinh(\xi) \cosh(\xi) - \xi \right] \tanh(\xi) {\rm sech}(\xi) + \mathcal{O}((1-k^2)^2), \\
\label{dn-der-appendix}
& {\rm dn}(\xi;k)  =  {\rm sech}(\xi) + \frac{1}{4} (1-k^2) \left[ \sinh(\xi) \cosh(\xi) + \xi \right] \tanh(\xi) {\rm sech}(\xi) + \mathcal{O}((1-k^2)^2),
\end{align}
where the expansion is understood in the sense of the power series in $(1-k^2)$ as $k \to 1$ uniformly in $\xi$.
In the same sense, formulas (\ref{sn-der}), (\ref{cn-der}), and (\ref{dn-der}) follow from expansions
(\ref{sn-der-appendix}), (\ref{cn-der-appendix}), and (\ref{dn-der-appendix}). Here we give a rigorous proof
of these asymptotical representations, as well as the bound on the first derivative given by (\ref{bound-second-derivative}).
We only give the proof for the Jacobi elliptic function ${\rm dn}$. The proof for other Jacobi elliptic functions
is similar.

First, let us recall the basic identities for the Jacobi elliptic functions
\begin{equation}
\label{elliptic-identities}
{\rm dn}^2(\xi;k) + k^2 {\rm sn}^2(\xi;k) = 1, \quad {\rm cn}^2(\xi;k) + {\rm sn}^2(\xi;k) = 1,
\end{equation}
from which it follows that ${\rm dn}(\xi;k)$ is monotonically decreasing from ${\rm dn}(0;k) = 1$
to ${\rm dn}(K(k);k) = \sqrt{1 - k^2}$, when $\xi$ changes from $0$ to $K(k)$, where
$K(k)$ is the complete elliptic integral. Also recall that ${\rm dn}(\xi;k)$ is an
even, $2K(k)$-periodic function of $\xi$ for every $k \in (0,1)$.

From the integral representation 8.144 in \cite{Grad}, for every $k \in (0,1)$ and every $\xi \in (0,K(k))$, we have
\begin{equation}
\label{dn-integral}
\xi = \int_{{\rm dn}(\xi;k)}^1 \frac{dt}{\sqrt{1-t^2} \sqrt{t^2 - 1 + k^2}}.
\end{equation}
Let $u_0(\xi) := {\rm dn}(\xi;1)$. Using the parametrization $t = {\rm sech}(x)$, we obtain
from (\ref{dn-integral}) for every $\xi \in \mathbb{R}^+$:
$$
\xi = \int_{u_0(\xi)}^1 \frac{dt}{t \sqrt{1-t^2}} = \int_0^{x_0} dx = x_0,
$$
where $x_0$ is a positive root of ${\rm sech}(x) = u_0$. Therefore, $u_0(\xi) = {\rm sech}(x_0) = {\rm sech}(\xi)$ as in
the first formula of (\ref{dn-der}).

Differentiating (\ref{dn-integral}) in $k$ and using the identities (\ref{elliptic-identities}),
we obtain
\begin{equation}
\label{dn-der-integral}
\partial_k {\rm dn}(\xi;k) = -k^3 {\rm sn}(\xi;k) {\rm cn}(\xi;k)
\int_{{\rm dn}(\xi;k)}^1 \frac{dt}{\sqrt{1-t^2} \sqrt{(t^2 - 1 + k^2)^3}}.
\end{equation}
Let $v_0(\xi) := \partial_k {\rm dn}(\xi;1)$. Using the first formulas in
(\ref{sn-der}) and (\ref{cn-der}), as well as the same parametrization $t = {\rm sech}(x)$,
we obtain
\begin{align}
\nonumber
v_0(\xi) & =  -\tanh(\xi) {\rm sech}(\xi) \int_{u_0(\xi)}^1 \frac{dt}{t^3 \sqrt{1-t^2}} \\
\nonumber
& =  - \tanh(\xi) {\rm sech}(\xi) \int_0^{\xi} \cosh^2(x) dx \\
\label{dn-der-explicit}
& =  - \frac{1}{2} \tanh(\xi) {\rm sech}(\xi) \left[ \sinh(\xi) \cosh(\xi) + \xi \right],
\end{align}
which justifies the second formula in (\ref{dn-der}).

It remains to justify the bound (\ref{bound-second-derivative}) on the first derivative of
$u (\xi;k) := {\rm dn}(\xi;k)$ in $k$ for $\xi \in (0,\pi\mu)$ and $k \in (k_*(\mu),1)$,
where $\mu$ is sufficiently large. We recall that $u(\cdot;k) $ satisfies the second-order
differential equation for every $k \in (0,1)$:
\begin{equation}
- \frac{d^2 u}{d \xi^2} + (2-k^2) u - 2 u^3 = 0, \quad \xi \in (-K(k),K(k)).
\end{equation}

Let us introduce the linearization operator
\begin{equation}
\label{linearization-dnoidal}
L_k := -\frac{d^2}{d \xi^2} + 2 - k^2 - 6 u^2(\xi;k) : H^2_{\rm per}(-K(k),K(k)) \to L^2_{\rm per}(-K(k),K(k)).
\end{equation}
As $k \to 1$, the linearization operator $L_k$ defined by (\ref{linearization-dnoidal}) converges
in some sense to the limiting linearization operator
\begin{equation}
\label{linearization-dnoidal-limit}
L_{k=1} := -\frac{d^2}{d \xi^2} + 1 - 6 {\rm sech}^2(\xi) : H^2(\mathbb{R}) \to L^2(\mathbb{R}),
\end{equation}
which has a negative eigenvalue at $-3$ associated with the even eigenfunction $u_0^2$,
the zero eigenvalue associated with the odd eigenfunctions $u_0'$, and the essential spectrum
at $[1,\infty)$, where $u_0(\xi) = {\rm dn}(\xi;1) = {\rm sech}(\xi)$.

For convenience, we drop the first argument $\xi$ in the definition of $u(\xi;k) \equiv u(k)$,
so that we can denote the partial derivative of $u$ with respect to $k$ by a prime. Then,
we note that
\begin{equation}
\label{linearization-1}
L_k u'(k) = 2 k u(k).
\end{equation}
Since $\partial_\xi u = - k^2 {\rm sn} (\xi;k) {\rm cn} (\xi;k)$ solves $L_k \partial_\xi u = 0$,
we obtain the unique $2K(k)$-periodic and even solution of the differential equation (\ref{linearization-1})
by variation of a constant:
\begin{equation}
\label{derivative-dnoidal-cnoidal}
u'(k) = - k {\rm sn}(\xi;k) {\rm cn}(\xi;k) \int_0^{\xi} \frac{d \xi'}{{\rm cn}^2(\xi';k)}.
\end{equation}
This representation is complementary to (\ref{dn-der-integral}) and it admits the same
expression as in (\ref{dn-der-explicit}) if the limiting values of the Jacobi elliptic functions
in (\ref{sn-der-appendix}), (\ref{cn-der-appendix}), and (\ref{dn-der-appendix}) are used.

Now we note that although the inhomogeneous equation (\ref{linearization-1}) can be uniquely solved
in the limit $k \to 1$ in $H^2(\mathbb{R})$ because $u_0$ is orthogonal to ${\rm ker}(L_{k=1}) = {\rm span}(u_0')$,
the limiting expression for $u'(1)$ contains an exponentially growing function of $\xi$ as per the explicit expression
(\ref{dn-der-explicit}). This is because the homogeneous equation $L_{k=1} v = 0$ admits an even
exponentially growing solution $v(\xi) = \cosh(\xi) + \tilde{v}(\xi)$, where $\tilde{v} \in H^2(\mathbb{R})$.

Using further differentiation of (\ref{linearization-1}) in $k$, we obtain a chain of linear inhomogeneous equations
\begin{align*}
L_k u''(k) & =  4 k u'(k) + 2 u(k) + 12 u(k) (u'(k))^2, \\
L_k u'''(k) & = 6 k u''(k) + 6 u'(k) + 36 u(k) u'(k) u''(k) + 12 (u'(k))^3, \\
L_k u''''(k) & =  8 k u'''(k) + 12 u''(k) + 48 u(k) u'(k) u'''(k) + 72 (u'(k))^2 u''(k) + 36 u(k) (u''(k))^2.
\end{align*}
Inspecting the right-hand sides of these linear inhomogeneous equations in the limit $k \to 1$
and inverting $L_{k=1}$ on the even right-hand sides, we obtain that all derivatives
of $u$ in $k$ are exponential growing functions of $\xi$ with the following growth rates:
\begin{equation*}
u''(1) = \mathcal{O}(\xi \sinh(\xi)), \quad u'''(1) = \mathcal{O}(\cosh(3\xi)), \quad
u''''(1) = \mathcal{O}(\xi \sinh(3 \xi)), \quad \mbox{\rm as} \quad \xi \to \pm \infty.
\end{equation*}
It follows by induction (the proof is omitted) that for every $k \in \mathbb{N}$, we have
\begin{equation}
u^{(2k+1)}(1) = \mathcal{O}( \cosh((2k+1) \xi)), \quad
u^{(2k+2)}(1) = \mathcal{O}(\xi \sinh((2k+1) \xi)), \quad \mbox{\rm as} \quad \xi \to \pm \infty,
\end{equation}
where the implicit constants grow polynomially in $k$. The $N$-th partial sum of the Taylor series
$$
S_N := \sum_{n=1}^{N} \frac{1}{n!} (k-1)^n u^{(n+1)}(1)
$$
converges for every $\xi \in (-\pi \mu,\pi\mu)$ and $k \in (k_*(\mu),1)$,
where $k_*(\mu) = 1 - 8 e^{-2\pi \mu} + \mathcal{O}(e^{-4 \pi \mu})$ as $\mu \to \infty$.
Therefore, $u'(k) - u'(1) = \lim_{N \to \infty} S_N$ is well-defined by the majorant power series
in the corresponding domain. Moreover, if $\mu$ is sufficiently large, then
$$
|u'(k) - u'(1)| \leq C |k-1| (1 + |\xi|) \cosh(\xi), \quad \xi \in (-\pi \mu,\pi \mu), \quad
k \in (k_*(\mu),1),
$$
where $C$ is a positive $\mu$-independent constant $C$. This bound is equivalent
to the third bound in (\ref{bound-second-derivative}).

\vspace{0.25cm}

{\bf Acknowledgements:} J.L.M. was supported in part by U.S. NSF DMS-1312874 and NSF CAREER Grant DMS-1352353
and is grateful to the Schr\"odinger Institute in Vienna and the Mathematical Sciences Research Institute
for hosting him during part of the completion of this work. The work of D.P. is supported by the Ministry of Education
and Science of Russian Federation (the base part of the state task No. 2014/133, project No. 2839).

\end{document}